\pdfoutput=1
\documentclass[11pt]{article} 
\usepackage[letterpaper,hmargin=2.5cm]{geometry} 
\usepackage{amsmath,amssymb,latexsym}
\usepackage{theorem} 
\usepackage{overpic}
\numberwithin{equation}{section}
\newtheorem{theorem}{Theorem}[section]
\newtheorem{lemma}[theorem]{Lemma}
\newtheorem{conjecture}[theorem]{Conjecture}
\newtheorem{corollary}[theorem]{Corollary}
\newtheorem{proposition}[theorem]{Proposition}
{\theorembodyfont{\rmfamily} \newtheorem{definition}[theorem]{Definition}
\newtheorem{remark}[theorem]{Remark}
\newtheorem{rmk}[theorem]{Remark}
\newtheorem{example}[theorem]{Example}}
{\theorembodyfont{\slshape} }
\newcommand{\field}[1]{\mathbb{#1}}
\newcommand{\R}{\field{R}}
\newcommand{\Z}{\field{Z}}
\newcommand{\N}{\field{N}}
\newcommand{\C}{\field{C}}
\renewcommand{\P}{\field{P}}

\newcommand{\EE}{{\mathcal E}}
\newcommand{\MM}{{\mathcal M}}

\renewcommand{\AA}{{\mathcal A}}
\newcommand{\sgn}{\mathop{\rm sgn}}

\newcommand{\cp}{\mathrm{cap}}
\newcommand{\lebesgue}{mes}
\newcommand{\supp}{\mathop{\rm supp}}
\newcommand{\vt}[1]{\boldsymbol{#1}}

\newcommand{\const}{{\rm const}}

\renewcommand{\Re}{\mathop{\rm Re}} 
\renewcommand{\Im}{\mathop{\rm Im}}
\newcommand{\dist}{\mathop{\rm dist}}
\newcommand{\isdef}{\stackrel{\text{\tiny def}}{=}}
\newcommand{\weakto}{\stackrel{\ast}{\longrightarrow}}
\newcommand{\grad}{\mathop{\rm grad}}

\DeclareRobustCommand{\qed}{%
\ifmmode 
\else \leavevmode\unskip\penalty9999 \hbox{}\nobreak\hfill \fi
\quad\hbox{\qedsymbol}}
\newcommand{\openbox}{\leavevmode
\hbox to.77778em{%
\hfil\vrule
\vbox to.675em{\hrule width.6em\vfil\hrule}%
\vrule\hfil}}
\newcommand{\qedsymbol}{\openbox}
\newcommand{\proofname}{Proof}
\newenvironment{proof}[1][\proofname]{\par
\normalfont \trivlist
\item[\hskip\labelsep   \itshape #1. ]
\ignorespaces
}{%
\qed\endtrivlist }

\begin{document}

\title{Critical measures, quadratic differentials, and weak limits of zeros of Stieltjes polynomials}

\author{A.~Mart\'{\i}nez-Finkelshtein, and E.~A.~Rakhmanov}


\maketitle

\begin{abstract}
We investigate the asymptotic zero distribution of Heine-Stieltjes polynomials -- polynomial solutions of a second order differential equations with complex polynomial coefficients. In the case when all zeros of the leading coefficients are all real, zeros of the  Heine-Stieltjes polynomials were interpreted by Stieltjes as discrete distributions minimizing an energy functional. In a general complex situation one deals instead with a critical point of the energy.
We introduce the notion of discrete and continuous critical measures (saddle points of the weighted logarithmic energy on the plane), and prove that a weak-* limit of a sequence of discrete critical measures is a continuous critical measure. Thus,  the limit zero distributions of the Heine-Stieltjes polynomials are given by continuous critical measures. We give a detailed description of such measures, showing their connections with quadratic differentials. In doing that, we obtain some results on the global structure of rational quadratic differentials on the Riemann sphere that have an independent interest.

The problem has a rich variety of connections with other fields of analysis, some of them are briefly mentioned in the paper.
\end{abstract}

\tableofcontents

\section{Generalized Lam\'{e} equation}

Let us start with a classical problem more than 125 year old. Given a set of pairwise distinct points fixed on the complex plane $\C$,
\begin{equation}\label{A}
\mathcal A=\{a_0, a_1, \dots, a_p\},
\end{equation}
($p\in \N$), and two polynomials,
\begin{equation}\label{defAandB}
    A(z)=\prod_{i=0}^p (z-a_i)\,, \qquad B(z)=\alpha z^p+\text{lower
degree terms} \in \P_p\,, \quad \alpha \in \C ,
\end{equation}
where we denote by $\P_n$ the set of all algebraic polynomials of degree $\leq n$, we are interested in the polynomial solutions of the
\emph{generalized Lam\'{e} differential} equation (in algebraic form),
\begin{equation}
\label{DifEq}
A(z)\, y''(z)+ B(z)\, y'(z)- n(n+\alpha -1) V_n(z)\, y(z)=0,
\end{equation}
where $V_n$ is a polynomial (in general, depending on $n$) of degree $\leq p-1$; if $\deg V=p-1$, then $V$ is monic.
An alternative perspective on the same problem can be stated in terms of the second order differential operator
$$
\mathcal L[y](z) \isdef A(z)\, y''(z)+ B(z)\, y'(z),
$$
and the associated generalized spectral problem (or multiparameter eigenvalue problem, see \cite{Volkmer1988}),
\begin{equation}\label{ODEclassic}
   \mathcal L[y](z)=n(n+\alpha -1) V_n(z)\, y(z)\,, \quad  n \in \N,
\end{equation}
where $V_n \in \P_{p-1}$ is the ``spectral polynomial''.

Special instances of equation \eqref{DifEq} are well known.  For instance, $p=1$ corresponds to the hypergeometric differential equation. Case $p=2$ was studied by Lam\'{e} in the 1830's in the special setting $B=A'/2$ ,  $a_j\in \R$, and $a_0 + a_1 + a_2 = 0$, in connection with the separation of variables in
the Laplace equation using elliptical coordinates (see e.g.~\cite[Ch.~23]{Whittaker96}). For the general situation of $p=2$ we get the Heun's equation, which still attracts interest and poses open questions (see \cite{Ronveaux95}).

Recently, equation \eqref{DifEq}  has also found other applications in studies as diverse as the construction of ellipsoidal and
sphero-conal $h$-harmonics of the Dunkl–Laplacian \cite{Volkmer2006}, \cite{Volkmer2008},  the
quantum asymmetric top \cite{Agnew2008,Bourget:2009ge,Grosset:2005hn}, or certain quantum completely integrable system called the generalized Gaudin spin chains \cite{Harnad1995}, and their thermodynamic limits.

Heine \cite{Heine1878} proved that for every $n \in \N$ there exist
at most
\begin{equation}\label{sigma}
\sigma(n)=\binom{n+p-1}{n}
\end{equation}
different polynomials $V_n$ such that \eqref{DifEq} (or \eqref{ODEclassic}) admits a polynomial solution $y=Q_n \in \P_{n}$. These particular $V_n$ are called \emph{Van Vleck polynomials}, and the corresponding
polynomial solutions $y=Q_n$ are known as \emph{Heine-Stieltjes} (or simply Stieltjes) polynomials.

Heine's theorem states that if the polynomials $A$ and $B$ are algebraically independent (that is, they do not satisfy any algebraic equation with integer coefficients) then for any $n\in \N$ there exist \emph{exactly} $\sigma(n)$ Van Vleck polynomials $V_n$, their degree is \emph{exactly} $p-1$, and for each $V_n$ equation \eqref{DifEq} has a unique (up to a constant factor) solution $y$ of degree $n$. The condition of algebraic independence of $A$ and $B$ is sufficient but not necessary. It should be noted that the original argument of Heine is far from clear, and even Szeg\H{o} \cite{szego:1975} cites his result in a rather ambiguous form. Recently a significant research on the algebraic aspects of this theory has been carried out by B.~Shapiro in \cite{Shapiro2008a}, and we refer the reader to his work for further details. In particular, it has been proved in \cite{Shapiro2008a} that for any polynomials $A$ and $B$ like in \eqref{defAandB} there exists $N\in \N$ such that for any $n\geq N$, there exist $\sigma(n)$ Van Vleck polynomials $V_n$ of degree exactly $p-1$ such that \eqref{DifEq} has a polynomial solution of degree exactly $n$.

Stieltjes discovered an electrostatic interpretation of zeros of the polynomials discussed in \cite{Heine1878}, which attracted common attention to the problem.
He studied the problem \eqref{DifEq} in a particular setting, assuming that $\AA \subset \R$ and that all residues $\rho_k$ in \begin{equation}\label{BoverA}
    \frac{B(x)}{A(x)}=
  \sum_{k=0}^p \frac{\rho_k}{x-a_k}
\end{equation}
are strictly positive (which is equivalent to the assumption that the zeros of $A$
alternate with those of $B$ and that the leading coefficient of
$B$ is positive).
He proved in \cite{Stieltjes1885} (see also \cite[Theorem
6.8]{szego:1975}) that in this case for each $n \in \N$ there are \emph{exactly}
$\sigma(n)$ different Van Vleck polynomials of degree $p-1$ and
the same number of corresponding Heine-Stieltjes polynomials $y$ of
degree $n$, given by all possible ways how the $n$ zeros of $y$
can be distributed in  the $p$ open intervals defined by $\AA$ (see Section \ref{sec:stieltjesmodel}).
Further generalizations of the work of Heine and Stieltjes 
followed several paths; we will mention only some of them. First,
under Stieltjes' assumptions ($\AA\subset \R$ and $\rho_k>0$), Van Vleck \cite{Vleck1898} and B\^{o}cher \cite{Bocher97} proved that the zeros of each $V_n$ belong to the convex hull of $\AA$ (see also the work of Shah 
\cite{Shah:68}, \cite{Shah:69b}, \cite{Shah:69}, \cite{Shah:70}). P\'{o}lya \cite{Polya12} 
showed that this is true for $\AA\subset \C$ if we keep the assumption of positivity of the residues $\rho_k$.  Marden
\cite{Marden66}, and later, Al-Rashed, Alam and Zaheer (see
\cite{Al-Rashed85}--\cite{Alam79},
\cite{Zaheer76}, \cite{Zaheer77}) established further results on
location of the zeros of the Heine-Stieltjes polynomials under
weaker conditions on the coefficients $A$ and $B$ of
\eqref{DifEq}. An electrostatic interpretation of these zeros in
cases when $\AA \subset \R$ and some residues $\rho_k$  are negative
has been studied by Gr\"{u}nbaum \cite{Grunbaum98}, and by Dimitrov and
Van Assche \cite{Dimitrov:00}. For some interlacing properties, see e.g.~\cite{Bourget:2009ry}.

We are interested in the asymptotic regime (so called semiclassical asymptotics) when $n$ (the degree of the Heine-Stieltjes polynomials) tends to infinity. First general result in this direction, based precisely on the Stieltjes model, is due to Mart\'{\i}nez-Finkelshtein and Saff \cite{MR2003j:33031}. There the limit distribution of zeros of Heine-Stieltjes polynomials has been established in terms of the traditional extremal problem for the weighted logarithmic energy on a compact set of the plane.

The main goal of this paper is to consider the weak-* asymptotics of the Heine-Stieltjes and Van Vleck polynomials in the general setting of $\AA\subset \C$ and $\rho_k\in \C$, which leads to a very different electrostatic problem - equilibrium problem in the conducting plane (with a finite exceptional set of points). It is essentially known that zeros of Heine-Stieltjes polynomials present a discrete critical measure -- saddle point of the discrete energy functional. A continuous  analogue of this notion leads to a concept of ``continuous'' critical measure, i.e.~critical point of the usual energy functional defined on Borel measures with respect to a certain class of local variations. 

We prove (Section \ref{sec:weakLimits}) that weak limit of discrete critical measures is a continuous critical measure (as the number of atoms or mass points tends to infinity). Thus, discrete critical measures are limit distributions of zeros of the Heine-Stieltjes polynomials.

To complete the description of the limit zero distributions of these polynomials we have to study more deeply the set of continuous critical measures. The problem, rather complex, is connected to many other classical problems of analysis, and has potentially a large circle of applications. In Section \ref{sec:criticalmeasuresandextremal} we mention a few connections, in particular, to minimal capacity problem and its generalizations.

In Section \ref{sec:qdandcriticalmeasures} we characterize critical measures in terms of trajectories of a (closed) rational quadratic differential on the Riemann sphere; for completeness of the reading we summarize basic results on quadratic differentials in Section \ref{sec:qd}. Further investigation of such differentials in carried out in Sections \ref{sec:p=2} (case $p=2$) and \ref{sec:generalP} (general case).

In the following two sections, \ref{sec:stieltjesmodel} and \ref{sec:criticalmeasures}, we discuss in some detail the concepts of the discrete and continuous equilibrium.

\section{Discrete  and continuous extremal measures}\label{sec:stieltjesmodel}

\subsection{Stieltjes electrostatic model: discrete equilibrium}

We denote by $\mathfrak M_n$ the class of uniform discrete measures on $\C$,
$$
\mathfrak M_n \isdef \left \{ \sum_{k=1}^n \delta_{z_k}, z_k\in \C \right\}\,, \quad \text{and} \quad \mathfrak M  \isdef \bigcup_{n\geq 1} \mathfrak M_n\,,
$$
where $\delta_x$ is a unit mass (Dirac delta) at $x$. With any polynomial $P(z)= \prod_{j=1}^n (z-\zeta_j)$ we associate its \emph{zero counting measure}
\[
\nu(P) =\sum_{j=1}^n   \delta_{\zeta_j}\in \mathfrak M_n\,,
\]
where the zeros are counted according to their multiplicity.

For  $\mu=\sum_{k=1}^n  \delta_{\zeta _k}\in \mathfrak M$ we define its (discrete) energy
$$
\mathcal E(\mu)   \isdef
  \sum_{i\neq j}     \log \frac{1}{|\zeta _i-\zeta _j|}\,,
  $$
 (if two or more $\zeta_j$'s coincide, then $\mathcal E(\mu)=+\infty$).
Additionally, given a real-valued function (\emph{external field})  $\varphi $, finite at $\supp(\mu)$, we consider the weighted energy
\begin{equation}\label{defWeightedEnergy}
    \EE_\varphi(\mu) \isdef \EE (\mu)+ 2 \sum_{k=1}^n \varphi(\zeta _k)  \,.
\end{equation}

In the above mentioned paper  \cite{Stieltjes1885} Stieltjes introduced the following extremal problem. For fixed subset  $\AA =\{ a_0, \dots, a_{p} \} \subset \R$, $a_{0}< \dots< a_{p}$, values $\rho_k\geq 0$, $k=0, 1, \dots, p$,  and an arbitrary vector $\vt{n}=(n_1, \dots, n_p)\in \Z_+^{p}$ (where $\Z_+=\N \cup \{0\}$), define $|\vt{n}|=n_1+\dots+n_p$, $\Delta_{j}\isdef [a_{j-1}, a_{j}]$, $j=1, \dots, p$, and $\Delta=\cup_{j=1}^p \Delta_{j}= [a_{0}, a_{p}]$. Consider the class of discrete measures
\begin{equation}
\label{discreteforStieltjes}
\mathfrak M_{|\vt{n}|}(\Delta, \vt{n})\isdef \left\{ \mu\in \mathfrak M_{|\vt{n}|}:\, \supp(\mu)\subset \Delta, \,\mu(\Delta_{j})=n_{j},\,  j=1, \dots, p\right \},
\end{equation}
and  the external field
\begin{equation}\label{admisibleField}
\varphi(x)=\Re\left( \Phi(x)\right), \quad \Phi(x)= \sum_{j=0}^p \frac{\rho_j}{2} \, \log \frac{1}{x-a_j}\,.
\end{equation}
We seek a measure $\mu^*=\mu^*(\vt{n})$ minimizing the weighted energy \eqref{defWeightedEnergy} in the class $\mathfrak M_{|\vt{n}|}(\Delta, n)$:
\begin{equation}
\label{mindiscreteStieltjes}
 \EE_\varphi(\mu^*)=\min \left\{  \EE_\varphi(\mu) :\, \mu  \in \mathfrak M_{|\vt{n}|}(\Delta, n) \right\} .
\end{equation}
In other words, we place $n_{j}$ unit electric charges on the conductor $\Delta_{j}$ and look for the equilibrium position of such a system of charges in the external field $\varphi$, if the interaction obeys the logarithmic law.

Stieltjes proved that the global minimum \eqref{mindiscreteStieltjes} provides the only equilibrium position, and that the zeros of the solution $y=Q_{\vt{n}}$ of \eqref{DifEq} are exactly points of the support of the extremal measure $\mu^*$ in
\eqref{mindiscreteStieltjes}: $\nu(Q_{\vt{n}})=\mu^*$. Actually, $\mu^*$ provides also the unique component-wise or point-wise minimum of $ \EE_\varphi$ (``Nash-type'' equilibrium).

The Stieltjes equilibrium problem \eqref{mindiscreteStieltjes} is a constrained one: the constraints are embedded in the definition of the class $\mathfrak M_{|\vt{n}|}(\Delta, n)$. A classical non-constrained version of the same problem leads to the (weighted) Fekete points. Given a compact $\Delta\subset \C$ and $n\in \N$, we want to find $\mu^*\in \mathfrak M_{n}(\Delta)\isdef \left\{ \mu\in \mathfrak M_{n}:\, \supp(\mu)\subset \Delta\right\}$ with
$$
 \EE_\varphi(\mu^*)=\min \left\{  \EE_\varphi(\mu) :\, \mu  \in \mathfrak M_{n}(\Delta) \right\} .
$$

Stieltjes' model for the hypergeometric case  ($p=1$) provides the well known electrostatic interpretation of the Jacobi polynomials. Zeros of the Jacobi polynomials $P_{n}^{(\alpha, \beta)}$ are also weighted Fekete points for $\Delta=[-1,1]$ and $\varphi(x)=\frac{\alpha+1}{2}\, \log\frac{1}{|x-1|}+\frac{\beta+1}{2}\, \log\frac{1}{|x+1|}$. Similarly, zeros of Laguerre and Hermite polynomials are weighted Fekete points for $\Delta=[0,+\infty)$,  $\varphi(x)=\frac{\alpha+1}{2}\, \log\frac{1}{|x|}+\frac{x}{2}$ and $\Delta=\R$,  $\varphi(x)=\frac{x^2}{2}$, respectively. It was pointed out in \cite{Ismail2000} that zeros of general orthogonal polynomials with respect to a measure on $\R$ may be interpreted as weighted Fekete points with an external field $\varphi=\varphi_{n}$ in general depending on the degree $n$.

Besides its elegance, the electrostatic model just described allows to establish monotonicity properties of the zeros of the Heine-Stieltjes polynomials as function of the parameters $\rho_k$. Furthermore, the minimization problem for the discrete energy it is based upon, admits substantial generalizations (one of them is subject of the present paper). The problem of the limit distribution of the discrete extremal points as $n\to \infty$ leads to the corresponding continuous energy problems.

\subsection{Extremal problem for Borel measures: continuous equilibrium}

We denote by $\MM$ (resp., $\MM_\R$) the set of all finite positive (resp., real) Borel measures $\mu$ with compact support $\supp (\mu)\subset \C$. Hereafter, $|\mu |$ stands for the total variation of $\mu \in \MM_\R$, and $\|\mu\|=|\mu|(\C)$.
For $n\in \N$,  let $ \MM_n\isdef \left \{ \mu \subset \MM:\, \|\mu \|=n\right\} $ be the set of positive Borel measures with total mass $n$ on $\C$.

With every measure $\mu \in \MM_\R$  we can associate its (continuous) logarithmic energy
\begin{equation}\label{defEnergyContinuous}
    E(\mu) \isdef  \iint \log \frac{1}{|x-y|}\, d\mu(x) d\mu(y)\,.
\end{equation}
Given the \emph{external field} $\varphi\in L^1(|\mu|)$, we consider also the weighted energy
\begin{equation}\label{defWeightedEnergyCont}
    E_\varphi(\mu) \isdef E (\mu)+ 2 \int \varphi\, d\mu  \,.
\end{equation}

If $\Gamma$ is a subset of $\C$, we denote by $\MM(\Gamma)$ (resp., $\MM_\R(\Gamma)$) the restriction of the corresponding families to measures supported on $\Gamma$. Again, a standard extremal problem of the potential theory is to seek for a global minimizer $ \lambda_{\Gamma,\varphi} \in \MM_{1}(\Gamma)$ such that
\begin{equation}\label{def_capacity}
   E_\varphi( \lambda_{\Gamma,\varphi})= \rho\isdef \min \left\{  E_\varphi(\mu):\,\mu  \in \MM_{1}(\Gamma) \right\}.
\end{equation}
It is well known that under certain conditions on $\varphi$ this minimizer $ \lambda_{\Gamma,\varphi}$ exists and is unique; it is called the \emph{equilibrium measure} of $\Gamma $ in the external field $\varphi$, see e.g.~\cite{Saff:97} for further details.  For $\varphi\equiv 0$, measure $\lambda _\Gamma =\lambda _{\Gamma , 0}$ is also known as the \emph{Robin measure} of $\Gamma $.

In terms of the extremal constant $\rho$ we can also define the \emph{weighted (logarithmic) capacity} of $\Gamma $,
$  \cp_\varphi (\Gamma ) =e^{-\rho}$. For $\varphi \equiv 0$ we simplify notation writing $\cp  (\Gamma)$ instead of $\cp_0 (\Gamma)$. If $\cp (\Gamma)=0$ then $\Gamma $ is a \emph{polar set}. Observe that $E(\mu)=+\infty$ for any $\mu\in \mathfrak M$, so that any finite set is polar.

There is a number of  properties characterizing the equilibrium measure $\lambda_{\Gamma,\varphi} $. For instance, if we define the logarithmic potential of $\mu\in \MM_\C$ by
$$
U^\mu(z)  \isdef  \int \log\frac{1}{|z-t|}\, d\mu(t)\,,
$$
then up to a polar subset of $ \Gamma $,
\begin{equation}\label{equilibriumSection6}
 U^{\lambda_{\Gamma,\varphi}}(z) +\varphi(z) \begin{cases} = \rho^*, & \text{if } z \in \supp (\lambda_{\Gamma,\varphi} ), \\
\geq \rho^*, & \text{if } z \in \Gamma ,
\end{cases}
\end{equation}
where $\rho^*$ is a constant related to $\rho$ and $\varphi$. Furthermore,
if $\Gamma $ and $\varphi$ are sufficiently regular,
\begin{equation}\label{maximinEqMeasure}
    \min_{z\in \Gamma } \left( U^{\lambda_{\Gamma,\varphi} }(z)+\varphi(z) \right)= \max_{ \mu \in \MM_{1}(\Gamma ) }\,  \min_{z\in \Gamma } \left( U^{\mu}(z) +\varphi(z)\right).
\end{equation}
This max-min property is a basis for applications of the equilibrium measure in the asymptotic theory of extremal (in particular, orthogonal) polynomials, see \cite{Gonchar:84}, \cite{Mhaskar/Saff:84},  \cite{Rakhmanov:84}, and also the monograph \cite{Saff:97}.

Like for the discrete measures, we will consider general external fields of the form
$$
\varphi(z)=\Re \Phi(z),
$$
where $\Phi$ is analytic, but in general multivalued. What we require in the sequel is that $\Phi'$ is holomorphic in $\C\setminus \AA$, allowing further construction below.

\begin{remark} \label{remark:vector}
Further generalizations of this construction can be obtained either considering several measures on respective sets interacting according to a certain law (vector equilibrium) \cite{Gonchar:85}, or including additional constraints. For instance, prescribing an upper bound on the density of the extremal measure on $\Gamma $ in \eqref{def_capacity} we obtain the so-called constrained equilibrium \cite{Dragnev97}, \cite{Rakhmanov:96}, relevant for the asymptotic description of polynomials of discrete orthogonality. Another way is to impose in \eqref{def_capacity} the size of $\mu$ on each component of $\Gamma $, such as it was done in \cite{MR2003j:33031}: if $\AA =\{ a_0, \dots, a_{p} \} \subset \R$, $a_{0}< \dots< a_{p}$,  $\Gamma_{j}\isdef [a_{j-1}, a_{j}]$, $j=1, \dots, p$,  $\Gamma=\cup_{j=1}^p \Gamma_{j}= [a_{0}, a_{p}]$, and
$\mathcal{N}$ is the standard simplex in $\R^{p-1}$,
\[
\mathcal{N}=\left\{\vt{\theta}=(\theta_1, \dots, \theta_p):\,
\theta_i \geq 0, \, i=1, \dots, p, \text{ and } \sum_{i=1}^p
\theta_i =1 \right\}\,,
\]
then for each  $\vt{\theta}=(\theta_1, \dots, \theta_p) \in \mathcal{N}$ we can consider the global minimum of the weighted energy $E_{\varphi}(\cdot)$ restricted to the class
$$
\mathcal M_{1}(\Gamma, \vt{\theta})\isdef \left\{ \mu \in \mathcal M_{1}:\, \supp(\mu)\subset \Gamma,\, \mu(\Gamma_{j})=\theta_{j},\, j=1, \dots, p-1\right \}.
$$ 
Again, for $\varphi$ like in \eqref{admisibleField} with $\rho_{j}\geq 0$ there exists a unique minimizing energy,  $\lambda_{\Gamma,\varphi} (\vt{\theta})$.
\end{remark}

\begin{remark}
It should be mentioned that a characterization of the weighted Fekete points on the real line and its continuous limit were used in \cite{MR2000j:31003} to prove new results on the support of an equilibrium (i.e.~extremal) measure in an analytic external field on $\R$.
\end{remark}

\subsection{Relation between discrete and continuous equilibria}

The transfinite diameter of a compact set $\Gamma $ is defined by the limit process when the number of Fekete points tends to infinity. It was P\'{o}lya who proved the remarkable fact that  the transfinite diameter of $\Gamma $  is equal to its capacity. Fekete observed further that the normalized counting measure of Fekete points converges to the equilibrium (Robin) measure of $\Gamma $. For the weighted analogue of this result, see \cite[Ch.~III]{Saff:97}.

The connection between the discrete and continuous equilibria allowed to use the Stieltjes model in \cite{MR2003j:33031} in order to obtain in this situation the limit distribution of zeros of Heine-Stieltjes polynomials. Namely, if for each vector $\vt{n}=(n_1, \dots, n_p)\in \Z_+^{p}$ we  consider the discrete extremal measure $\mu^*(\vt{n})$ introduced in \eqref{mindiscreteStieltjes}, and assume that $|\vt{n}|\to \infty$ in such a way that each fraction $n_k/ |\vt{n}|$ has a limit,
$$
\lim_{|\vt{n}|\to \infty} \frac{n_{j}}{|\vt{n}|}=\theta_{j}, \quad j=1, \dots, p,
$$
then $\mu^*(\vt{n})/ |\vt{n}|$ weakly converges to the equilibrium measure $\lambda_{\Gamma,0} (\vt{\theta})$,  $\Gamma =[a_{0}, a_{p}]$, $\vt{\theta}=(\theta_1, \dots, \theta_p)$, defined in the previous section. In a certain sense, this can be regarded as a generalization of the just mentioned classical result of Fekete.

\section{Discrete and continuous critical measures} \label{sec:criticalmeasures}

According to a well-known result of Gauss, there are no stable equilibrium configurations (i.e.~local minima of the energy) in an a conducting open set under a harmonic external field. Unstable equilibria usually do not attract much attention from a point of view of Physics. However, as we will show further, they constitute a rich and relevant object that appear naturally in many fields of analysis.

We introduce now the concept that plays the leading role in this paper: the family of measures providing saddle points for the logarithmic energy on the plane, with a separate treatment of the discrete and continuous cases.

\subsection{Discrete critical measures}

We start with the following definition:
\begin{definition} \label{def:AcriticalDisc}
Let $\Omega $ be a domain on $\C$, $\AA \subset \Omega$ a subset of zero capacity, and $\varphi$ be a $C^1$ real-valued function in $\Omega\setminus \AA$. A measure
\begin{equation}\label{defMucritDiscrete}
    \mu=\sum_{k=1}^n \delta_{\zeta_k} \in \mathfrak M_n, \quad \zeta_i \neq \zeta_j \text{ for } i\neq j,
\end{equation}
is a discrete \emph{$(\mathcal A, \varphi)$-critical measure} in $\Omega$, if $\supp(\mu)\subset \C\setminus \AA$, and for the weighted discrete energy $\EE_\varphi(\mu)=\EE_\varphi(\zeta _1, \dots, \zeta _n)$ we have
\begin{equation}\label{gradient}
    \grad \EE_\varphi(\zeta_1, \dots, \zeta_k)=0\,,
\end{equation}
or equivalently,
$$
\frac{\partial}{\partial z} \, \EE_\varphi(\zeta_1, \dots, z, \dots \zeta_n)\big|_{z=\zeta_k} =0, \quad k = 1, \dots, n, \qquad \frac{\partial}{\partial z} = \frac{1}{2}\, \left(\frac{\partial}{\partial x} - i \frac{\partial}{\partial y}\right).
$$
More generally, if $\varphi=\Re \Phi$, where $\Phi$ is an analytic (in general, multivalued) function in $\Omega$ with a single-valued derivative $\Phi'$, then this definition does not need any modification.
\end{definition}
In the sequel we omit the mention to $\Omega$ if $\Omega =\C$. 

The following proposition is just a reformulation of Eq.~\eqref{DifEq} in this new terminology:
\begin{proposition}
\label{characterizationHS}
Assume that $\AA=\{a_0, a_1, \dots, a_p \}$, $p\in \N$, is a set of pairwise distinct points on $\C$, and the external field $\varphi$ is given by \eqref{admisibleField}.   
Then
\begin{equation}\label{muDiscrete}
\mu=\sum_{k=1}^n \delta_{\zeta _k} \in \mathfrak M_n, \quad \zeta _i \neq \zeta _j \text{ for } i\neq j,
\end{equation}
supported on $\C\setminus \AA$, is a discrete $(\mathcal A, \varphi)$-critical measure if and only if there exists a polynomial $V_n \in \P_{p-1}$ such that  $y(z)=y_n(z)=\prod_{k=1}^n (z-\zeta _k)$ is a solution of the differential equation \eqref{DifEq}, with
$$
 \frac{B(x)}{A(x)}=
  \sum_{k=0}^p \frac{\rho_k}{x-a_k}.
$$
\end{proposition}
In other words, discrete $(\mathcal A, \varphi)$-critical measures with external field generated by complex charges fixed at $\AA$  correspond precisely to zeros of Heine-Stieltjes polynomials.
\begin{proof}
A straightforward computation shows that for $z\neq w$,
$$
2 \, \frac{\partial}{\partial z}\, \log |z-w| =\frac{1}{z-w}\,.
$$
Hence,
$$
2 \, \frac{\partial}{\partial \zeta_k} \, \EE (\zeta_1,   \dots, \zeta_n)  = - 2 \, \frac{\partial}{\partial \zeta_k } \, \sum_{i\neq j} \log |\zeta _i- \zeta _j| = -  \sum_{j\neq k} \frac{1}{\zeta _k - \zeta _j}.
$$
On the other hand, the multivalued function $\varphi$ has a single-valued derivative given by (see \eqref{admisibleField})
$$
2 \, \frac{\partial}{\partial z}\, \varphi(z)= -\sum_{j=0}^p \frac{\rho_j}{2} \, \frac{1}{z - a _j}=  \Phi'(z).
$$
Thus, using the notation from \eqref{BoverA}, we can rewrite condition \eqref{gradient} as
\begin{equation}\label{almostLame1}
2 \, \left( \sum_{j\neq k} \frac{1}{\zeta _k - \zeta _j} -  2 \Phi'(\zeta _k)\right) = 2 \, \sum_{j\neq k} \frac{1}{\zeta _k - \zeta _j} +  \frac{B}{A}(\zeta _k) =0\,, \quad k = 1, \dots, n,
\end{equation}
and with $ y(z)\isdef \prod_{i=1}^n (z-\zeta _i)$ this identity takes the form
\begin{equation}\label{almostLame}
\left(\frac{y''}{y'} +  \frac{B }{ A }\right)(\zeta_k)=0\,, \quad k=1, \dots, n\,.
\end{equation}
As a consequence, polynomial
$$
A(z)\, y ''(z) +B(z) y '(z)\in \P_{n+p-1}\,,
$$
is divisible by $y$, so there exists a polynomial $\widetilde V_{n}\in \P_{p-1}$ such that
$$
A(z)\, y ''(z) + B(z) y '(z)= \widetilde V_n(z) y(z)\,,
$$
which concludes the proof.
\end{proof}

In the sequel we will make use of the following uniform boundedness of the supports of the discrete critical measures, corresponding to a sequence of external fields of the form
\begin{equation}\label{extFieldMoreBis}
\varphi_n=\Re \Phi_n, \quad   \Phi_n(z)=-\sum_{k=0}^p \frac{\rho_k(n)}{2} \, \log (z-a_k )\,,
\end{equation}
where $\rho_k(n)\in \C$.

\begin{proposition} \label{prop:boundedness}
Let $\mu_n\in \mathfrak M_n$, $n\in \N$, be a discrete $(\AA, \varphi_n)$-critical measure corresponding to an external field \eqref{extFieldMoreBis}. If
\begin{equation}\label{liminfcondition}
    \liminf_n  \Re \sum_{k=0}^p \frac{\rho_k(n) }{n} > -\frac{1}{2},
\end{equation}
then $\bigcup_n \supp(\mu_n)$ is bounded in $\C$.
\end{proposition}
In other words, if we assume that in \eqref{DifEq} the coefficient $B=B_n$ may depend on $n$, but $B_n/n$ is bounded (in such a way that \eqref{liminfcondition} holds), then the zeros of the Heine-Stieltjes polynomials are also uniformly bounded.

\begin{proof}
Let $\mu_n=\sum_{k=1}^n \delta_{\zeta _k(n)} \in \mathfrak M_n$, and assume that $|\zeta _1(n)|\geq \dots \geq |\zeta _n(n)|$. Since $|\zeta _1(n)|>0$, by \eqref{almostLame1},
$$
\sum_{j=2}^n  \frac{1}{1 - \zeta _j(n)/\zeta _1(n)} =  -\sum_{k=0}^p \rho_k(n) \,\frac{\zeta _1(n)}{\zeta _1(n)-a_k} \,.
$$
But  
$$
|\zeta _j(n)/\zeta _1(n)|\leq 1 \quad \Rightarrow \quad \Re\left(\frac{1}{1 - \zeta _j(n)/\zeta _1(n)} \right)\geq 1/2,
$$
so that
$$
\frac{1}{n-1}\, \Re \sum_{k=0}^p \rho_k(n) \,\frac{\zeta _1(n)}{\zeta _1(n)-a_k} \leq -\frac{1}{2}.
$$
Hence, if $\zeta _1(n)\to \infty$ along a subsequence of $\N$, then
$$
\liminf_n  \Re \sum_{k=0}^p \frac{\rho_k(n) }{n} \leq -\frac{1}{2},
$$
which contradicts our assumptions.
\end{proof}
\begin{rmk}
It was proved in \cite{Shapiro2008a} that for a fixed $\varphi$ of the form \eqref{extFieldMoreBis} (that is, $\rho_k(n)\equiv \rho_k$, $k=0, \dots, p$), the zeros of the Heine-Stieltjes polynomials accumulate on the convex hull of $\AA$.
\end{rmk}
\begin{rmk}\label{rmk:discreteunbounded}
Condition  \eqref{liminfcondition} is in general necessary for the assertion of Proposition \ref{prop:boundedness}. Indeed, for $p=0$, $a_0=0$, and
$$
\varphi_n(z)=\frac{n-1}{2}\, \log|z|,
$$
any discrete uniform measure supported at the scaled zeros of unity, that is,
$$
\mu_n=\sum_{k=1}^n \delta_{\zeta _k(n)} \in \mathfrak M_n, \quad \zeta_k(n)=\zeta_n\, e^{2\pi i k /n}, \quad \zeta_n \in \C\setminus \{0\},
$$
is $(\AA, \varphi_n)$-critical, which is easily established using \eqref{almostLame1} and \eqref{almostLame}. Obviously, for $\zeta_n\to \infty$ the support of $\mu_n$ is not uniformly bounded in $n$.
\end{rmk}

\subsection{Continuous critical measures}

Unlike in the discrete case, we provide now a variational definition for the continuous critical measure.

Any smooth complex-valued function $h$ in the closure  $\overline \Omega $ of a domain $\Omega$ generates a local variation of $\Omega$ by $ z
\mapsto z^t=z+ t \, h(z)$, $t\in \C$. It is easy to see that $ z \mapsto z^t $ is injective for small
values of the parameter $t$. The transformation above induces a variation of sets $ e \mapsto
e^t \isdef \{z^t:\, z \in e\}$, and (signed) measures: $ \mu \mapsto
\mu^t$, defined by $\mu^t(e^t)=\mu(e) $; in the differential form, the pullback measure $\mu^t$
can be written as $d\mu^t (x^t)=d\mu(x)$.

\begin{definition} \label{def:AcriticalCont}
Let $\Omega $ be a domain on $\C$, $\AA \subset \Omega$ a subset of zero capacity, and $\varphi$ be a  $C^1$  real-valued function in $\Omega\setminus \AA$. We say that a signed measure $\mu\in \MM_\R(\Omega)$ is a continuous \emph{$(\mathcal A, \varphi)$-critical} if for any $h$ smooth in $\Omega\setminus \AA$ such that $h\big|_\AA \equiv 0$,
\begin{equation}\label{derivativeEnergy}
    \frac{d}{dt}\, E_\varphi(\mu^t)\big|_{t=0} = \lim_{t\to 0} \frac{E_\varphi(\mu^t)- E_\varphi(\mu)}{t}=0.
\end{equation}
Furthermore, if $\varphi=\Re \Phi$, where $\Phi$ is an analytic (in general, multivalued) function in $\Omega$ with a single-valued derivative $\Phi'$, then this definition does not need any modification.
\end{definition}
In what follows we will always mean by an $(\mathcal A, \varphi)$-critical measure the continuous one, satisfying Definition \ref{def:AcriticalCont}. Furthermore, in order to simplify notation, we speak about an \emph{$\mathcal A$-critical measure} meaning a continuous $(\mathcal A, \varphi)$-critical measure with the external field $\varphi \equiv 0$. Observe that if $\AA \neq \emptyset$, this notion is nontrivial.

A particularly interesting case is treated in the following Lemma:
\begin{lemma}\label{lemma:condCrit}
If  $\varphi=\Re \Phi$, and $\Phi$  is analytic in a simply connected domain $\Omega$, then condition \eqref{derivativeEnergy} is equivalent to
\begin{equation}\label{condicionCrit}
 f_\varphi(\mu; h) =0\,,
\end{equation}
with
\begin{equation}\label{defFvar}
f_\varphi(\mu; h)\isdef   \iint  \frac{h(x)-h(y) }{x -y }   \, d\mu (x ) d\mu
(y ) - 2 \, \int    \Phi'(x)\, h(x) \, d\mu(x)\,.
\end{equation}
\end{lemma}
\begin{proof}
It is sufficient to show that
$$
E_\varphi(\mu^t)-E_\varphi(\mu) =-\Re \left\{t\, f(\mu; h)
+\mathcal O (t^2) \right\}\,.
$$

We have
\begin{align*}
E(\mu^t) &=   \iint \log \frac{1}{|x^t-y^t |}\, d\mu^t(x^t) d\mu^t(y^t)   
= \iint \log \frac{1}{|(x -y )+t\, (h(x)-h(y)) |}\, d\mu (x
) d\mu (y )\,,
\end{align*}
so that
\begin{align*}
E(\mu^t)   -E(\mu)   & = -
\iint  \log \left| 1+ t\, \frac{h(x)-h(y)
}{x -y }\right| \, d\mu (x ) d\mu (y ) \\
& = -  \Re  \iint
\left\{ \log \left( 1+ t\, \frac{h(x)-h(y) }{x -y }\right)\right\}
\, d\mu (x ) d\mu (y ) \,.
\end{align*}
On the other hand,
\begin{align*}
\int_{\Omega^t} \varphi(x^t) \, d\mu^t(x^t)   - \int_{\Omega}
\varphi(x) \, d\mu(x) &= \int_{\Omega } \varphi(x+ t h(x)) \, d\mu
(x )- \int_{\Omega} \varphi(x) \, d\mu(x)  \\ &= \Re \int \left(
\Phi(x+ t h(x)) - \Phi(x)\right)\, d\mu(x)\,.
\end{align*}
Taking into account the behavior of $\log(1+x)$ for small $x$, we conclude that as $t\to 0$,
\begin{align*}
E_\varphi(\mu^t)- E_\varphi(\mu)& =-  \Re  \iint
\left( t\, \frac{h(x)-h(y) }{x -y }+\mathcal O (t^2) \right) \,
d\mu (x ) d\mu (y ) \\ & + 2\, \Re \int \left( t \, \Phi'(x)\,
h(x) + \mathcal O(t^2)\right)\, d\mu(x) ,
\end{align*}
and the statement follows.
\end{proof}

\begin{rmk}\label{remarkdiscrete}
For a finite set $\AA$ and the external field given by \eqref{admisibleField}, the discrete $(\AA , \varphi)$-critical measures fit into the same variational definition as their continuous counterparts, as long as we replace in \eqref{derivativeEnergy} the continuous energy $E_\varphi(\mu)$ by $\EE_\varphi(\mu)$.

Indeed, arguments similar to those used in the proof of Lemma \ref{lemma:condCrit} show that for $\mu$ in \eqref{muDiscrete}, the condition \begin{equation}\label{discreteLocalVar}
    \frac{d}{dt}\, \EE_\varphi(\mu^t)\big|_{t=0} =0  ,
\end{equation}
written for
$$
h(\zeta )=\frac{A(\zeta )}{\zeta -z}\,, \quad z\notin \AA,
$$
yields
\begin{equation*}
    \sum_{i \neq j} \frac{1}{(\zeta _i-z)(\zeta _j-z)}   + \frac{B(z)}{A(z)}\,
 \sum_{i=1}^n \frac{1}{\zeta _i-z}    =  \frac{D(z)}{ A(z)}
 \,,
\end{equation*}
where $D$ is a polynomial. In particular, the residue of the left hand side (as a function of
$z$) is $0$ at $w=\zeta _k$, $k=1, \dots, n$; setting $ y(z)\isdef \prod_{i=1}^n (z-\zeta _i)$, we arrive again at the system \eqref{almostLame}.

And viceversa, using the chain rule it is easy to show that the condition \eqref{gradient} implies \eqref{discreteLocalVar}.
\end{rmk}

Critical measures constitute an important object; for a finite set $\AA$ the natural description of their structure is in terms of the trajectories of quadratic differentials.
In the next section we give an abridged introduction to quadratic differentials on the Riemann sphere in the form needed for our purposes. For a comprehensive account on this theory see for instance \cite{MR0096806}, \cite{Pommerenke75}, \cite{Strebel:84}, \cite{MR1929066}.

\section{Rational quadratic differentials on the Riemann sphere in a nutshell} \label{sec:qd}

Let $A$ and $V$ be monic polynomials of degree $p+1$ and $p-1$, respectively, with $A$ given by \eqref{defAandB} with all $a_k$'s pairwise distinct. The rational function $V/A$ defines on the Riemann sphere $\overline \C$ the quadratic differential
\begin{equation}\label{quadDiff}
\varpi  (z)=- \frac{V(z)}{  A(z) }\, (dz)^2.
\end{equation}
The only singular points of $\varpi$ (assuming that the zeros of $V$ and $A$ are disjoint) are:
\begin{itemize}
\item the points $a_k\in \AA$, where $\varpi$ has simple poles (critical points of order $-1$);
\item the zeros of $V$ of order $k\geq 1$, where $\varpi$ has zeros of the same order;
\item the infinity, where $\varpi$ has a double pole (critical point of order $-2$) with the residue $-1$.
\end{itemize}
The rest of the points in $\C$ are the regular points of $\varpi$, and their order is $0$. All singular points of order $\geq -1$ are called \emph{finite} critical points of $\varpi$.

In a neighborhood of any regular point $z_0 $ we can introduce a local parameter
\begin{equation}\label{transf1Xi}
    \xi =\xi (z)=\int^z  \sqrt{\varpi}  =\int^z  \sqrt{-\frac{V(t)}{  A(t) }}\, dt\,,
\end{equation}
in terms of which the representation of $\varpi$ is identically equal to one. This parameter is not uniquely determined: any other parameter $\widetilde \xi$ with this property satisfies $\widetilde \xi =\pm \xi + \const$. Function $\xi $ is called the \emph{distinguished} or \emph{natural parameter} near $z_0$.

Following \cite{Pommerenke75} and \cite{Strebel:84}, a smooth curve $\gamma$ along which
$$
-V (z)/A(z)\, (dz)^2>0 \quad \Leftrightarrow \quad \Im \xi(z)=\const
$$
is a \emph{horizontal arc} of the quadratic differential $\varpi $. More precisely, if $\gamma$ given by a parametrization $z(t)$, $t\in (\alpha ,\beta )$, then
$$
-\frac{V}{A}\, (z(t))\, \left(\frac{dz}{dt}\right)^2>0, \quad t
\in (\alpha ,\beta )\,.
$$
A maximal horizontal arc is called a \emph{horizontal trajectory} (or simply a \emph{trajectory}) of $\varpi $.
Analogously, trajectories of $-\varpi$ are called \emph{orthogonal} or  \emph{vertical trajectories} of  $\varpi$; along these curves
$$
V(z)/A(z)\, (dz)^2>0 \quad \Leftrightarrow \quad \Re \xi(z)=\const\,.
$$

Any simply connected domain $D$ not containing singular points of $\varpi$ and bounded by two vertical and two horizontal arcs is called a \emph{$\varpi$-rectangle}. In other words, if $\xi$ is any distinguished parameter in $D$, then $\xi(D)$ is a (euclidean) rectangle, and $D\mapsto \xi(D)$ is a one-to-one conformal mapping. Obviously, this definition is consistent with the freedom in the selection of the natural parameter $\xi$.

We can define a conformal invariant metric associated with the quadratic differential $\varpi$, given by the length element $|d \xi|=|\sqrt{V/A}|(z) |dz|$; the $\varpi$-length of a curve $\gamma$ is
$$
\|\gamma\|_{\varpi }= \frac{1}{\pi}\, \int_{\gamma}
\sqrt{\left|\frac{V}{A}\right| } \,(z)\, |dz|\,;
$$
(observe that this definition differs by a normalization constant from the definition
5.3 in \cite{Strebel:84}). Furthermore, if $D$ is a simply connected domain not containing singular points of $\varpi$, we can introduce the $\varpi$-distance by
$$
\dist(z_1, z_2; \varpi, D)=\inf \{\|\gamma\|_{\varpi  }:
\, z_1, z_2 \in \bar \gamma, \; \gamma \subset D \}\,.
$$
Trajectories and orthogonal trajectories are in fact geodesics (in the $\varpi$-metric) connecting any two of its
points. Indeed, according to \cite[Thm. 8.4]{Pommerenke75},  in any  simply connected domain $D$ not containing singular points of $\varpi$, a trajectory arc $\gamma$  joining $z_1$ with $z_2$ is the shortest: if $L_1$, $L_2$ are the orthogonal trajectories through $z_1$ and $z_2$, respectively, then any rectifiable curve  $\widetilde \gamma$  that connects $L_1$ with $L_2$ in $D$ satisfies
$$
\| \gamma\|_{\varpi}\leq \| \widetilde \gamma\|_{\varpi}\,.
$$

\begin{figure}[htb]
\centering \begin{tabular}{ll} \hspace{0cm}\mbox{\begin{overpic}[scale=0.55]%
{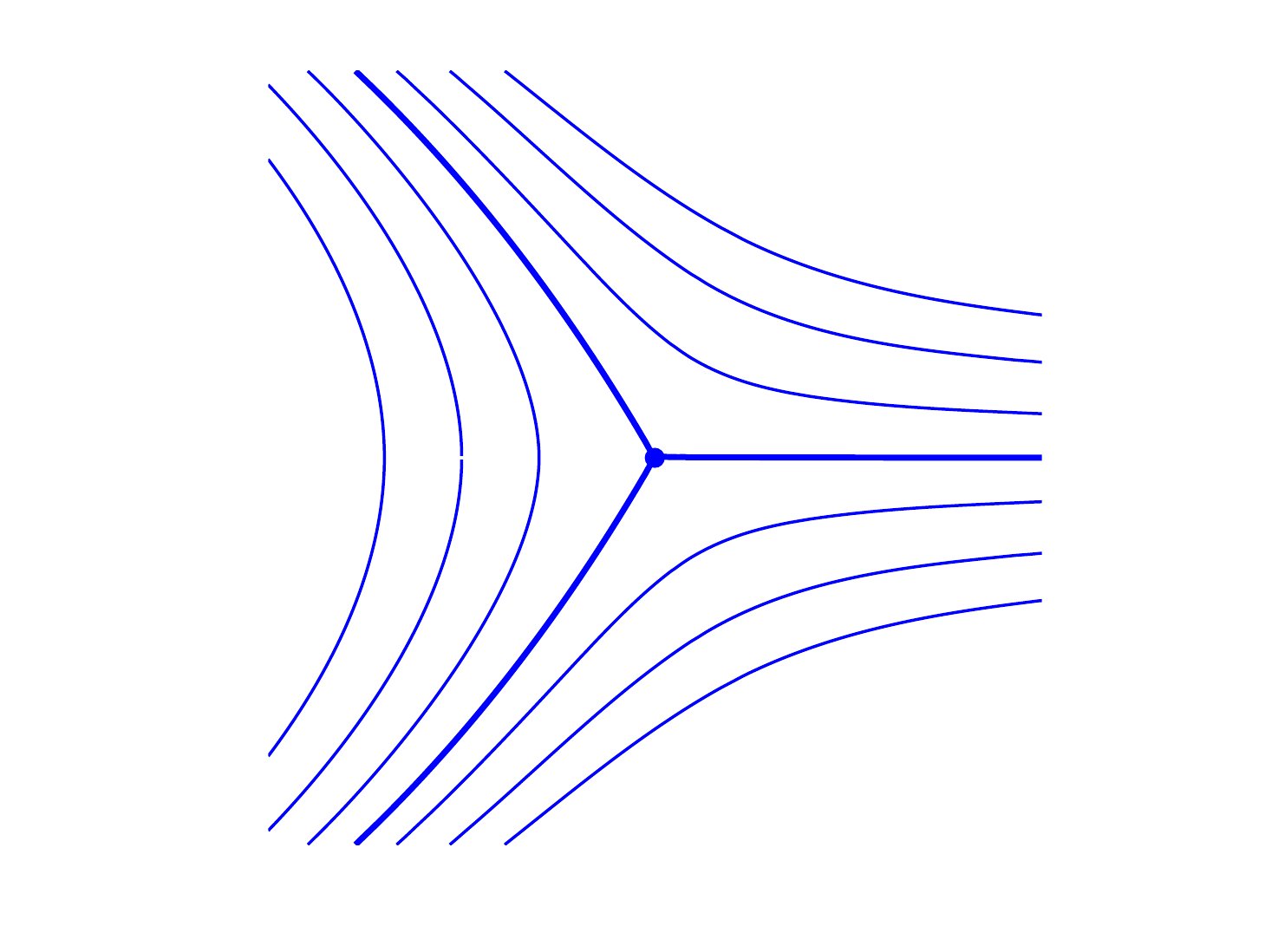}%
\end{overpic}} &
\mbox{\begin{overpic}[scale=0.55]%
{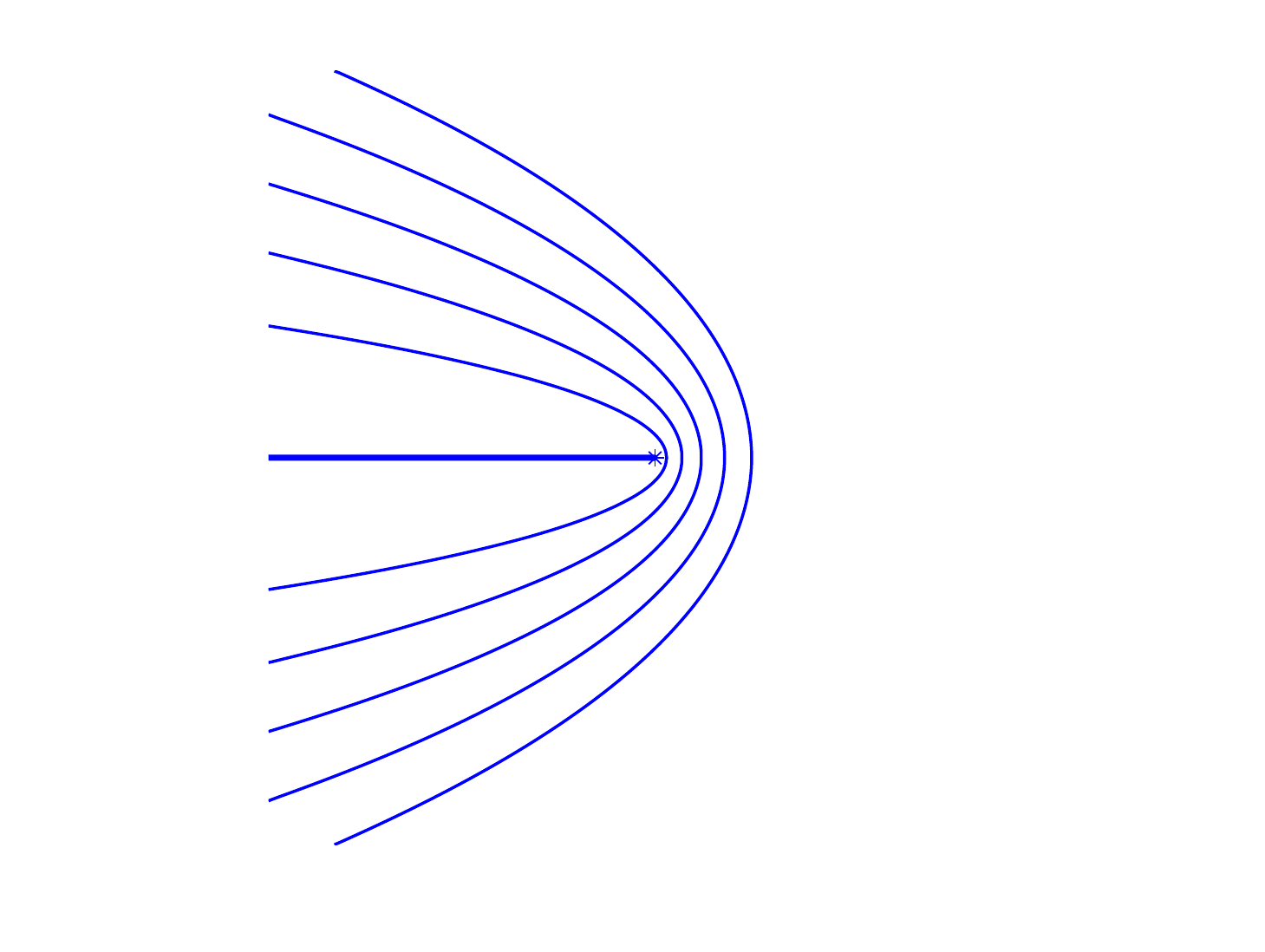}%
\end{overpic}}
\end{tabular}
\caption{The local trajectory structure near a simple zero (left) or a simple pole.}
\label{fig:localstructure}
\end{figure}

The local structure of the trajectories is well known (see the references cited at the end of the previous Section). For instance, at any regular point trajectories look locally as simple analytic arcs passing through this point, and through every regular point of $\varpi$ passes a uniquely determined horizontal and uniquely determined vertical trajectory of $\varpi$, that are locally orthogonal at this point \cite[Theorem 5.5]{Strebel:84}.
 If $z$ is a finite critical point of $\varpi$ of order $k\geq -1$, then from $z$ emanate $k+2$ trajectories under equal angles $2\pi/(k+2)$ (see Figure \ref{fig:localstructure}).

\begin{figure}[htb]
\centering \begin{tabular}{lll} \hspace{-1.5cm}\mbox{\begin{overpic}[scale=0.47]%
{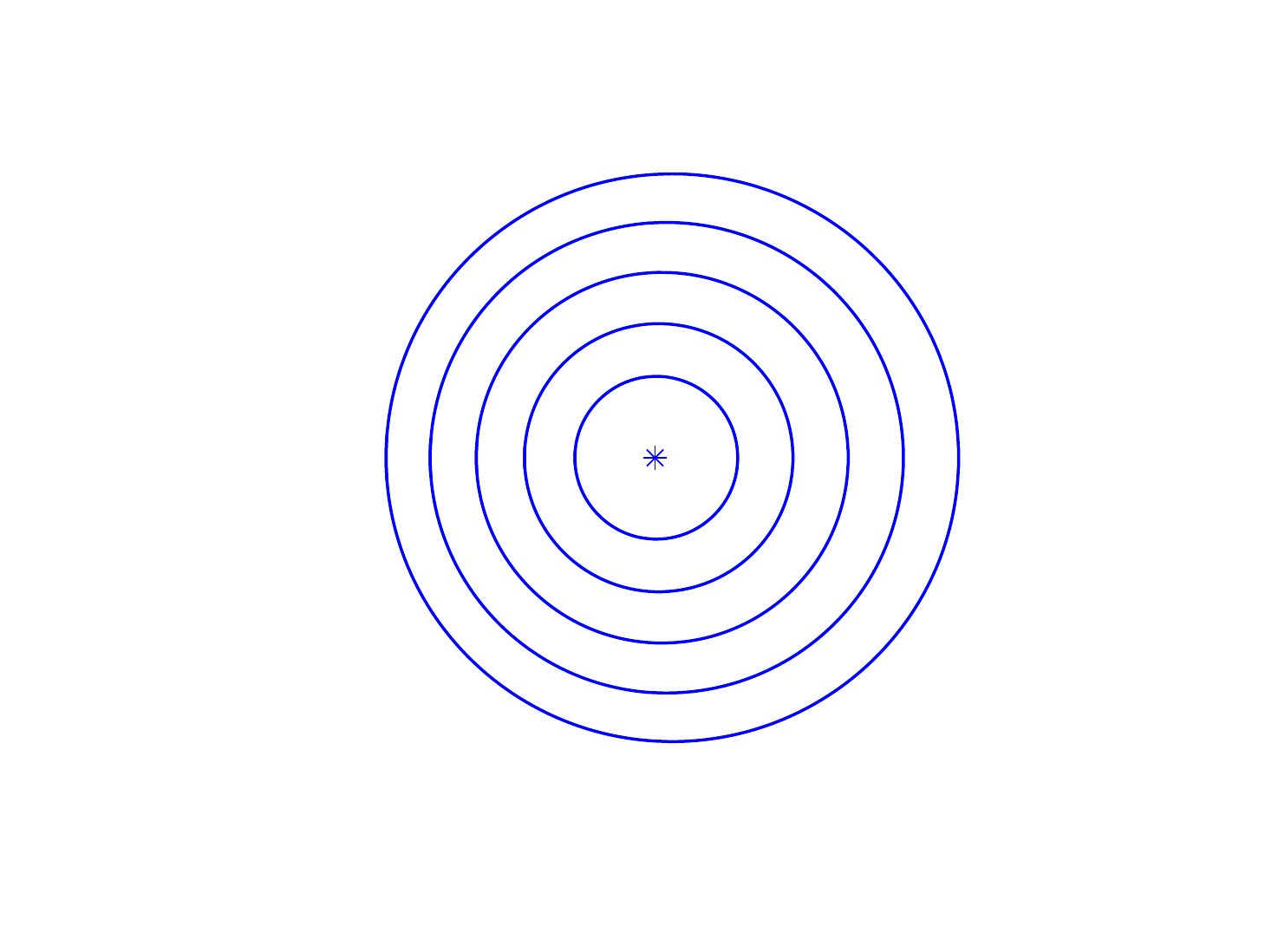}%
\end{overpic}} &
\hspace{-2.0cm}
\mbox{\begin{overpic}[scale=0.45]%
{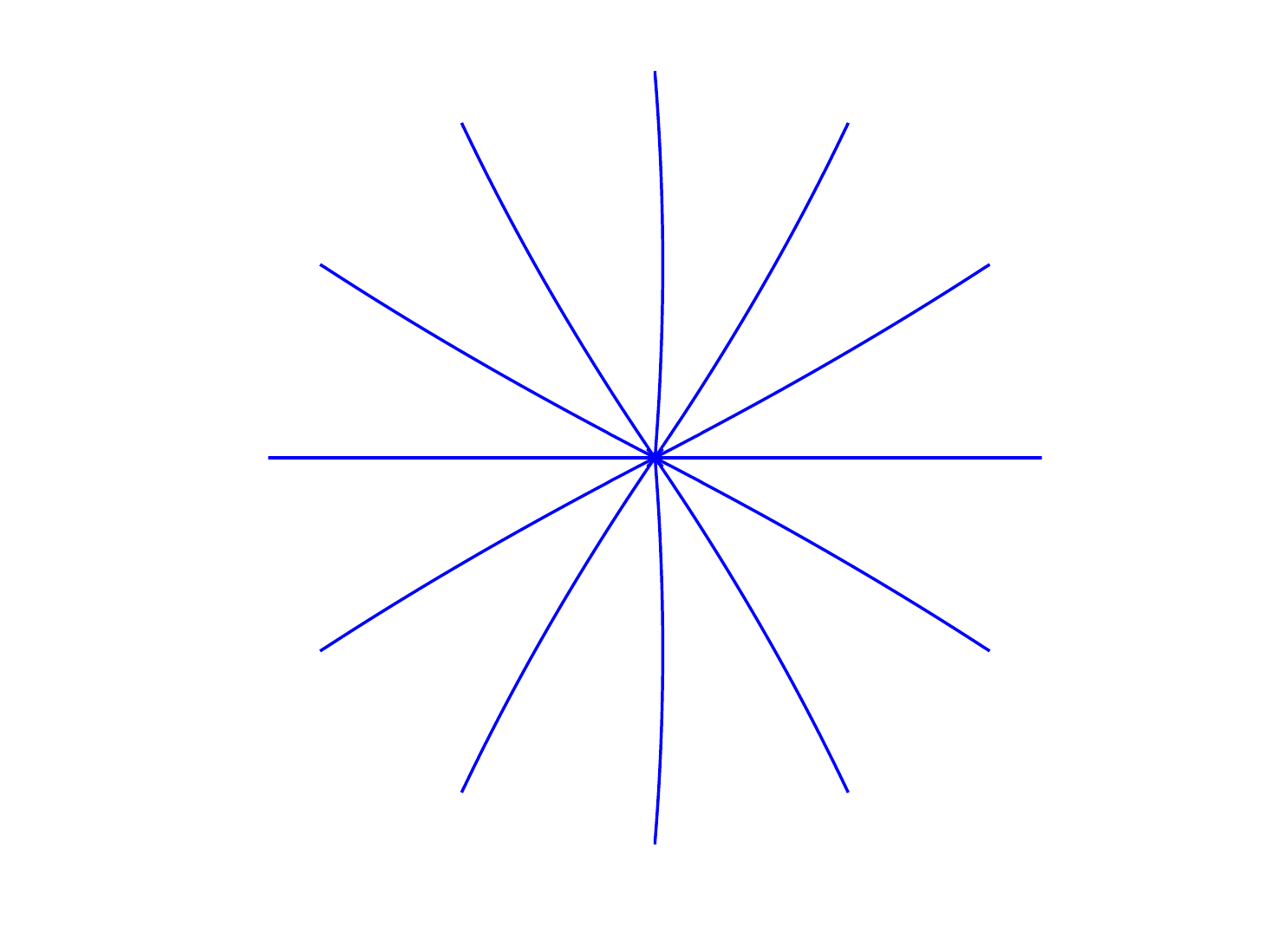}%
\end{overpic}}&
\hspace{-2.0cm}
\mbox{\begin{overpic}[scale=0.45]%
{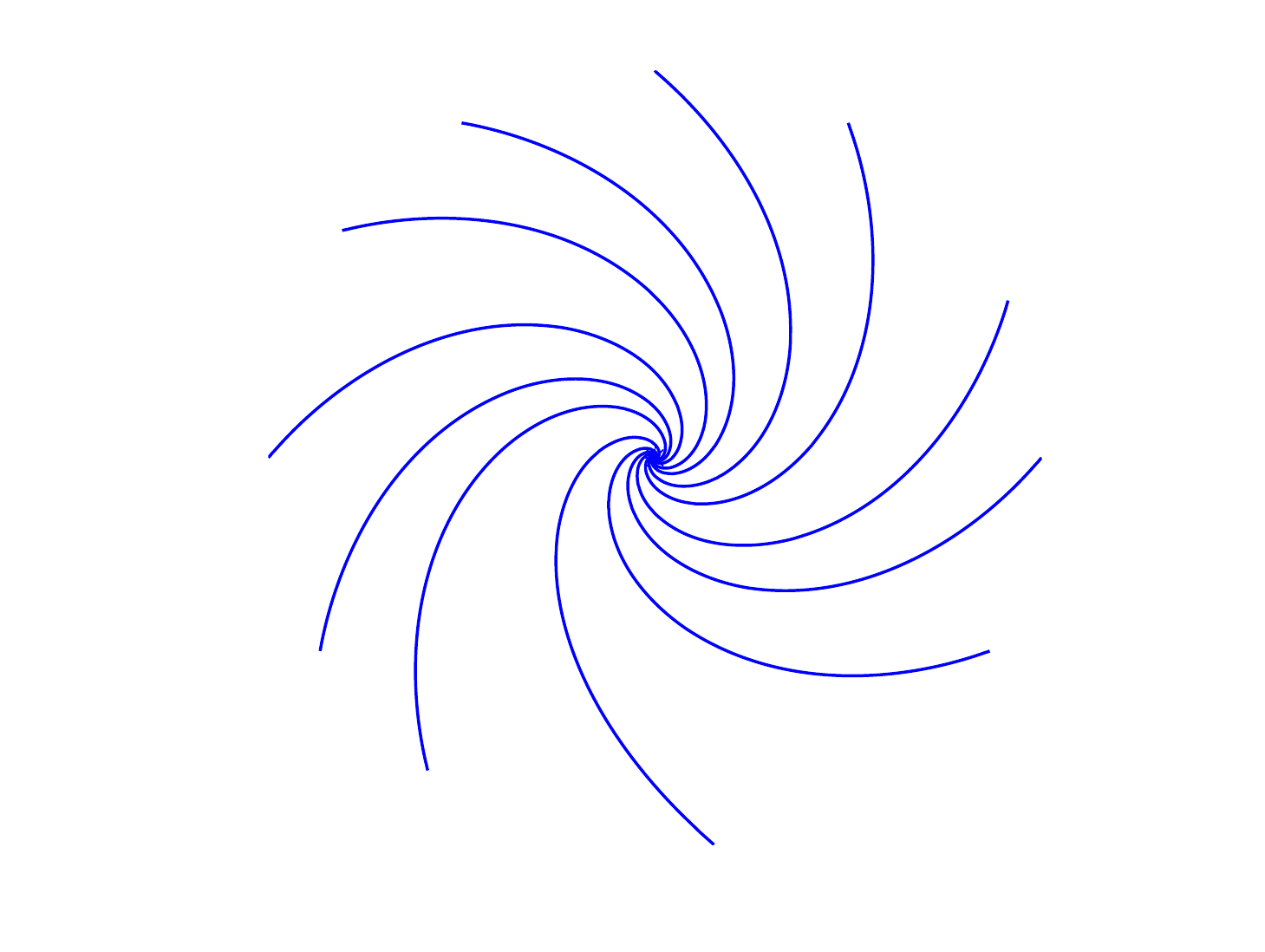}%
\end{overpic}}
\end{tabular}
\caption{The local trajectory structure near a double pole with a negative (left), positive (center) or non-real residue.}
\label{fig:localstructure2}
\end{figure}

In the case of a double pole, the trajectories have either the radial, the circular or the spiral form, depending whether the residue at this point is negative, positive or non-real, see Figure \ref{fig:localstructure2}. In particular, with the assumptions on $A$ and $V$ above all trajectories of the quadratic differential \eqref{quadDiff} in a neighborhood of infinity are topologically identical to circles.

The global structure of the trajectories is much less clear. The trajectories and orthogonal trajectories of a given differential  $\varpi$ produce a transversal foliation of the Riemann sphere $\overline \C$. The main source of troubles is the existence of the so-called recurrent trajectories, whose closure may have a non-zero plane Lebesgue measure.
We refer the reader to \cite{Strebel:84} for further details.

A trajectory $\gamma$ is \emph{critical} or \emph{short} if it joints two (not necessarily different) finite critical points of $\varpi$. The set of critical trajectories of $\varpi$ together with their endpoints (critical points of $\varpi$) is the \emph{critical graph} of $\varpi$. Critical and closed trajectories are the only trajectories of $\varpi$ with finite $\varpi$-length. The quadratic differential $\varpi$ is called \emph{closed} if all its trajectories are either critical  or closed (i.e.~all its trajectories have a finite $\varpi$-length). In this case the trajectories of $\varpi$ that constitute closed Jordan curves cover the whole plane, except a set of critical trajectories of a plane Lebesgue measure zero; see e.g.~Figure \ref{fig:intersectionNegative} for a typical structure of such trajectories.

If the quadratic differential \eqref{quadDiff} with $A$ given by \eqref{defAandB} is closed, there exists a set $\Gamma$ of at most $p$ critical trajectories of $\varpi$ such that the complement to $\Gamma $ is connected, and $\sqrt{V/A}$ has a single-valued branch in $\C\setminus \Gamma$.

\section{Critical measures in the field of a finite system of fixed charges} \label{sec:qdandcriticalmeasures}

In what follows we fix the set of $p+1$ distinct points $\AA=\{a_0, \dots, a_p \} \subset \C$ and consider the basic domain $\Omega=\C\setminus \AA$, $\AA=\{a_0, a_1, \dots, a_p\}$, and an external field $\varphi$ of the form
\begin{equation}\label{Phiforcharges}
 \varphi=\Re \Phi, \quad   \Phi(z)=-\sum_{k=0}^p \frac{\rho_k}{2} \, \log (z-a_k )\,, \quad \Phi'(z)= -\sum_{k=1}^p   \frac{ \rho_k/2}{ z-a_k  }=-
\frac{B(z)}{2 A(z)}\,,
\end{equation}
where we have used notation from \eqref{BoverA}. If $\{\rho_0, \dots, \rho_p\}\subset \R$, then this external field corresponds to the potential of a discrete signed measure supported on $\AA$:
\begin{equation}\label{potentialDiscrete}
    \varphi(z) = U^\sigma(z), \qquad \sigma =  \sum_{k=0}^p \frac{\rho_k}{2}\, \delta_{a_k}\in \mathfrak M_{p+1}.
\end{equation}
However, if any $\rho_k\in \C\setminus \R$, then $\varphi$ is not single-valued in $\C\setminus \AA$; nevertheless, the notion of an $(\AA, \varphi)$-critical measures for this case has been discussed in Definition \ref{def:AcriticalCont}. In particular, Lemma \ref{lemma:condCrit} applies.

In this section we state and prove the main structural theorem for $(\AA, \varphi)$-critical measures, which asserts that the support of any such a measure is a union of analytic curves made of trajectories of a rational quadratic differential. On each arc of its support the measure has an analytic density with respect to the arc-length measure. Finally, we describe the Cauchy transform and the logarithmic potential of an $(\AA, \varphi)$-critical measure.

\subsection{The main theorem}  

According to \eqref{Phiforcharges}, $\AA$ is exactly the set of singularities of the external field $\varphi$, except for the case when $\rho_k=0$ for some $k\in\{0, \dots, p\}$. In such a case we do not drop the corresponding $a_k$ from the set $\AA$; it remains as a fixed point of the class of variations (Definition \ref{def:AcriticalCont}). However, the status of the point $a_k\in \AA$ with $\rho_k=0$ is different from the case $\rho_k\neq 0$, see next theorem.

\begin{theorem}
\label{main_thm}
Let $\AA=\{a_0, a_1, \dots, a_p\}$ and $\varphi $ given by \eqref{Phiforcharges}. Then for any continuous $(\AA, \varphi)$-critical measure $\mu$ there exists a rational function $R$ with poles at $\AA$ and normalized by
\begin{equation}\label{normalizationR}
    R(z)=\left(\frac{ \kappa}{z}\right)^2 +\mathcal O\left(\frac{1}{z^3}\right), \quad z\to \infty, \qquad \kappa\isdef \mu(\C) +\frac{1}{2}\, \sum_{j=0}^p \rho_j,
\end{equation}
such that the support $\supp (\mu)$ consists of a union of trajectories of the quadratic differential $\varpi(z)=-R(z) dz^2$. If all $\rho_j\in \R$, then $\varpi$ is closed, and $\supp (\mu)$ is made of a finite number of trajectories of $\varpi$.

If in the representation \eqref{Phiforcharges}, $\rho_j= 0$, $j\in \{0, 1, \dots, p\}$, then $a_j$ is either a simple pole or a regular point of $R$; otherwise $R$ has a double pole at $a_j$.
\end{theorem}

The proof of this theorem reduces to two lemmas below. The first of them deals with the principal value of the Cauchy transform
\begin{equation}\label{principalvalue}
C^\mu(z)\isdef \lim_{\epsilon \to 0+} \int_{|z-x|> \epsilon }  \frac{1}{ x-z }\, d\mu(x)
\end{equation}
of the $(\AA, \varphi)$-critical measure $\mu$.
\begin{lemma}
\label{thm:rakhmanov}
For any $(\AA, \varphi)$-critical measure $\mu$ there exists a rational function $R$ with properties listed in Theorem \ref{main_thm} such that
\begin{equation}\label{charactMeasure}
\left(C^\mu(z) + \Phi'(z)
\right)^2=R(z) \qquad \lebesgue_2-\text{a.e.},
\end{equation}
where $\lebesgue_2$ is the plane Lebesgue measure on $\C$.
\end{lemma}
\begin{remark}
Formula \eqref{charactMeasure} and its variations for equilibriums measures of compact sets of minimal capacity (see Section \ref{sec:criticalmeasuresandextremal}) are well-known, although occasionally written in terms of quadratic differentials, see e.g.~the work of Nuttall~\cite{MR769985}, Stahl~\cite{Stahl2008}, \cite{Stahl:86}, Gonchar and Rakhmanov  \cite{Gonchar:87}, \cite{Rakhmanov94}, Deift and collaborators \cite{MR2000j:31003}. Notice that in the situation considered here the support of the critical measure is not known a-priori. 
\end{remark}
\begin{remark} \label{rem:sufficientcondition}
Formula \eqref{charactMeasure} is also sufficient for $\mu$ being $(\AA, \varphi)$-critical, so that it in fact characterizes these critical measures. The proof of this statement lies beyond the scope of this already lengthy paper, and we do not go into further details.
\end{remark} 
\begin{proof}
Assume that $\mu$ is an   $(\AA, \varphi)$-critical measure for $\varphi$ like in \eqref{Phiforcharges}. We will actually show that \eqref{charactMeasure} is valid at any point $z\in \C$ where the integral defining $C^\mu$ is absolutely convergent. It is well known that at such a $z$,
\begin{equation}\label{localconditionCauchy}
    \lim_{r\to 0^+} \int_{|x-z|<r} \frac{d\mu(x)}{|x-z|} =0\,,
\end{equation}
and this property holds a.e.\ with respect to $\lebesgue_2$.

At this point we would like to emphasize that the variational arguments we present next would be significantly simpler for points $z\notin \supp (\mu)$; in this case one can use the standard Schiffer variations $h(\zeta)=A(\zeta )/(\zeta -z)$ (see e.g.~the Appendix of \cite{Courant1950}), as it was done in the original paper \cite{Rakhmanov94} (for the logarithmic potentials), and subsequently in \cite[Chapter 8]{Kamvissis2003} and \cite{Kamvissis2005} (for the Green potentials). In the present situation we do not have any a priori information about $\supp(\mu )$. In order to address the problem of a possible pole on the support of $\mu$, we modify the variation by making $h(\zeta )\equiv 0$ in a neighborhood of $z$ (see \eqref{proprtiesH}). The function $\theta(\zeta)$ introduced below is meant to preserve the smoothness of $h$.

Fix $z\in \C$ satisfying \eqref{localconditionCauchy}, and for $r>0$ denote $\mathcal D_r\isdef \{\zeta\in \C:\, |\zeta -z|<r \}$. Function
$ m(r) \isdef \mu \left( \mathcal D_r\right) $
is continuous from the left and monotonically increasing, so that the subset
$$
\Delta \isdef \left\{r\in (0,1):\, m'(r) = \lim_{\varepsilon \to 0} \frac{m(r+\varepsilon)-m(r-\varepsilon)}{2 \varepsilon} \text{ exists} \right\}
$$
has the linear Lebesgue measure 1.

For  $r\in \Delta$ and  $\varepsilon \in (0,1)$ define the ``smooth step'' function
$$
\Lambda (x,\varepsilon)\isdef \begin{cases}
0, & \text{if } 0\leq x <1-\varepsilon, \\
\dfrac{(x-1-\varepsilon)^2 (x-1+2\varepsilon)}{4 \varepsilon^3}, & \text{if } 1-\varepsilon \leq  x < 1+\varepsilon, \\
1, & \text{if }   x \geq 1+\varepsilon.
\end{cases}
$$
It is easy to see that $\Lambda (\cdot, \varepsilon)\in C^1(\R_+)$ and that $|\frac{d}{dx} \, \Lambda (x, \varepsilon)| < 1/\varepsilon$ for all $\varepsilon>0$. Using this function we define on $\C$ the $C^1$ function
$$
\theta (\zeta) =\theta (\zeta, r, \varepsilon) \isdef \Lambda \left( \frac{|\zeta -z|}{r}, \varepsilon\right),
$$
and consider the condition \eqref{condicionCrit} with the following particular choice of $h$:
\begin{equation}\label{h_for_variation}
h(\zeta )= h_\varepsilon(\zeta ;r)=\frac{A(\zeta )}{\zeta -z}\, \theta (\zeta, r, \varepsilon)\,.
\end{equation}
For the sake of brevity we use the notation
$$
\mathcal K_{r,\varepsilon}\isdef   \overline {\mathcal D_{r(1+\varepsilon)}} \setminus \mathcal D_{r(1-\varepsilon)} , \quad \mathcal F_{r,\varepsilon}\isdef \C\setminus \mathcal D_{r(1+\varepsilon)},
$$
so that $\mathcal D_{r(1-\varepsilon)}$, $\mathcal K_{r,\varepsilon}$ and $\mathcal F_{r,\varepsilon}$ provide a partition of $\C$. Furthermore, by construction
\begin{equation}\label{proprtiesH}
    h(\zeta )=\begin{cases}
    0, &\text{if } \zeta \in \mathcal D_{r(1-\varepsilon)}, \\
    \dfrac{A(\zeta )}{\zeta -z}, &\text{if } \zeta \in \mathcal F_{r,\varepsilon}.
    \end{cases}
\end{equation}
Consider first
\begin{align*}
 \iint    \frac{h_\varepsilon(x ;r)-h_\varepsilon(y ;r) }{x -y }   \, d\mu (x ) d\mu (y ) & = I(\mathcal D_{r(1-\varepsilon)} \times \mathcal D_{r(1-\varepsilon)}) +
I(\mathcal K_{r,\varepsilon} \times \mathcal K_{r,\varepsilon}) + I(\mathcal F_{r,\varepsilon} \times \mathcal F_{r,\varepsilon}) \\
 & 2 I(\mathcal D_{r(1-\varepsilon)} \times \mathcal K_{r,\varepsilon}) + 2 I(\mathcal D_{r(1-\varepsilon)} \times \mathcal F_{r,\varepsilon}) + 2 I(\mathcal K_{r,\varepsilon}\times \mathcal F_{r,\varepsilon}),
\end{align*}
where $I(\Omega)$ means the integral in the l.h.s.\ taken over the set $\Omega$. Observe that by \eqref{proprtiesH}, $I(\mathcal D_{r(1-\varepsilon)} \times \mathcal D_{r(1-\varepsilon)}) =0$.

Let $\zeta \in \mathcal K_{r,\varepsilon}$; since
$$
\frac{\partial }{\partial \overline \zeta} \, \theta(\zeta) = \frac{1}{r} \, \Lambda '\left( \frac{|\zeta -z|}{r}, \varepsilon\right)\frac{\partial }{\partial \overline \zeta}  \, |\zeta -z|= \frac{1}{r} \, \Lambda '\left( \frac{|\zeta -z|}{r}, \varepsilon\right)\frac{\zeta -z }{|\zeta -z|}     ,
$$
we have
$$
\frac{1}{2}\, \left\|\grad  \theta(\zeta) \right\|= \left|\frac{\partial }{\partial \overline \zeta} \, \theta(\zeta) \right| \leq \frac{1}{ r \varepsilon }.
$$
In consequence, for $x, y \in \mathcal K_{r,\varepsilon}$,
\begin{equation}\label{boundH}
    \left| \frac{h_\varepsilon(x ;r)-h_\varepsilon(y ;r)  }{x -y }\right| \leq \frac{\const}{r \varepsilon},
\end{equation}
where the constant in the right hand side is independent of  $\varepsilon$. Obviously, by definition of $h$ we have that this inequality is valid (with a different constant) if $x  \in \mathcal K_{r,\varepsilon}$ and $y$ lies on a compact subset of $\C$.

From \eqref{boundH} we conclude that
\begin{equation}\label{intermediatebound1}
    \left| I( \mathcal K_{r,\varepsilon} \times \mathcal K_{r,\varepsilon}) \right|= \left|   \iint _{ \mathcal K_{r,\varepsilon} \times \mathcal K_{r,\varepsilon}  }   \frac{h_\varepsilon(x ;r)-h_\varepsilon(y ;r)  }{x -y }   \, d\mu (x ) d\mu (y )\right|\leq \frac{\const}{r \varepsilon}\, \left(\mu\left( \mathcal K_{r,\varepsilon}\right) \right)^2.
\end{equation}
Taking into account that $r \in \Delta$, we have that
\begin{equation}\label{existenseSimDerivatives}
    \lim_{\varepsilon \to 0+} \frac{\mu\left( \mathcal K_{r,\varepsilon}\right)}{\varepsilon} = 2 r\,  m'(r),
\end{equation}
so by \eqref{intermediatebound1}, $I( \mathcal K_{r,\varepsilon} \times \mathcal K_{r,\varepsilon})=o(1)$ as $\varepsilon\to 0+$.

Consider now $x\in \mathcal K_{r,\varepsilon}$ and $ y\in \mathcal D_{r,\varepsilon}$. Then
$$
\frac{h_\varepsilon(x ;r)-h_\varepsilon(y ;r)  }{x -y } =\frac{h_\varepsilon(x ;r)  }{x -y } = \frac{A(x)  }{(x-z)(x -y) }\, \theta(x) .
$$
Consider two cases. If $|y-z|<r(1-2\varepsilon)$, then
$$
\left|\frac{h_\varepsilon(x ;r)-h_\varepsilon(y ;r)  }{x -y } \right|\leq  \frac{\const  }{r(1-\varepsilon)|x -y| } \leq \frac{\const  }{r(1-\varepsilon)(|x| -|y|) }  .
$$
Hence, with a different constant,
$$
\left| \iint _{ x\in \mathcal K_{r,\varepsilon}, \; |y-z| <r(1-2\varepsilon)  }   \frac{h_\varepsilon(x ;r)-h_\varepsilon(y ;r)  }{x -y }   \, d\mu (x ) d\mu (y )\right|\leq \frac{\const  }{r  } \int_{r(1-\varepsilon)  }^{r(1+\varepsilon) } \int_0^{r(1-2\varepsilon) } \frac{s t}{t-s}\, ds dt;
$$
the double integral in the r.h.s.~is explicit, and it is straightforward to verify that it is $o(1)$ a $\varepsilon\to 0+$.

If on the contrary $r(1-2\varepsilon)<|y-z|<r(1-\varepsilon)$, with $x\in \mathcal K_{r,\varepsilon}$ we can use the estimate \eqref{existenseSimDerivatives}, which yields
$$
\left| \iint _{ x\in \mathcal K_{r,\varepsilon}, \; r(1-2\varepsilon)<|y-z|<r(1-\varepsilon)  }   \frac{h_\varepsilon(x ;r)-h_\varepsilon(y ;r)   }{x -y }   \, d\mu (x ) d\mu (y )\right|\leq  \frac{\const}{r \varepsilon}\,  \mu\left( \mathcal K_{r,\varepsilon}\right) \mu\left( \mathcal K_{r,2 \varepsilon}\right),
$$
and again by \eqref{existenseSimDerivatives}, the r.h.s~is $o(1)$ a $\varepsilon\to 0+$. Gathering the last two estimates we conclude that
$$
I(\mathcal D_{r(1-\varepsilon)} \times \mathcal K_{r,\varepsilon})= o(1), \quad \text{as } \varepsilon\to 0+.
$$
Similar considerations are obviously valid for $I(\mathcal K_{r,\varepsilon}\times \mathcal F_{r,\varepsilon})$.

Summarizing,
\begin{equation}\label{toZero}
    \lim_{\varepsilon \to 0+} \left( I(\mathcal K_{r,\varepsilon} \times \mathcal K_{r,\varepsilon}) + 2 I(\mathcal D_{r(1-\varepsilon)} \times \mathcal K_{r,\varepsilon}) + 2 I(\mathcal K_{r,\varepsilon}\times \mathcal F_{r,\varepsilon}) \right)=0,
\end{equation}
and we conclude that
\begin{equation}\label{firstIntegral}
\begin{split}
 \lim_{\varepsilon \to 0+}    \iint    \frac{h_\varepsilon(x ;r)-h_\varepsilon(y ;r)  }{x -y }   \, d\mu (x ) d\mu (y ) & = \iint _{ |x-z|\geq r, \, |y-z|\geq r   }   \frac{h_0(x ;r)-h_0(y ;r)  }{x -y }   \, d\mu (x ) d\mu (y ) \\ & + 2 \iint _{ |x-z|\geq  r, \, |y-z|< r   }   \frac{h_0(x ;r)-h_0(y ;r)  }{x -y }   \, d\mu (x ) d\mu (y )\\ & =I_1(r) + 2 I_2(r),
\end{split}
\end{equation}
where
$$
h_0(\zeta ;r)=\begin{cases}
    0, &\text{if } |\zeta -z|<r , \\
    \dfrac{A(\zeta )}{\zeta -z}, &\text{if } |\zeta -z|>r.
    \end{cases}
$$
Let us analyze the behavior of each integral as $r\to 0^+$ separately. First,
$$
I_2(r) = \iint _{ |x-z|\geq  r, \, |y-z|< r   }   \frac{A(x  )   }{(x-z) (x -y)}   \, d\mu (x ) d\mu (y ).
$$
Observe that by \eqref{localconditionCauchy},
$$
 \int _{   |x-z|\geq  r  } \left|\frac{A(x  )   }{ x-z  } \right|\left(  \int _{   |y-z|< r   }   \frac{ d\mu (y )  }{ |x -y| } \right)  \, d\mu (x )<+\infty,
$$
so that applying Fubini's theorem we conclude that
$$
I_2(r)=\int _{   |x-z|\geq  r  } \frac{A(x  )   }{ x-z  } \left(  \int _{   |y-z|< r   }   \frac{ d\mu (y )  }{ x -y } \right)  \, d\mu (x )\,.
$$
Using again \eqref{localconditionCauchy} we obtain that
$$
\lim_{r\to 0+, \, r\in \Delta} I_2(r) =0.
$$
On the other hand, by \eqref{proprtiesH},
\begin{equation}\label{secondIntegral}
I_1(r) = \iint _{ |x-z|\geq r, \, |y-z|\geq r   } \left(\frac{A(x)}{(x-z)(x-y)}-
\frac{A(y)}{(y-z)(x-y)}\right) \, d\mu (x ) d\mu (y ).
\end{equation}
The identity
\begin{equation}\label{identityForA}
    A(x)(y-z)-A(y)(x-z)+A(z)(x-y)=(x-y)(x-z)(y-z)D(x,y,z)
\end{equation}
is immediate, where
$$
D(x,y,z)=\alpha _0(x,y)+ \alpha _1(x,y)z +\dots+\alpha _{p-2}(x,y)z^{p-2}+z^{p-1}
$$
is a polynomial of degree $\leq p-1$ in each variable.
Hence,
\begin{align} \label{expression_h}
\frac{A(x)}{(x-z)(x-y)}= \frac{A(y)}{(y-z)(x-y)}-\frac{A(z)}{(x-z)(y-z)}+ D(x,y,z)\,.
\end{align}
Using it in \eqref{secondIntegral} we get that
\begin{align*}
\lim_{r\to 0+, \, r\in \Delta} I_1(r) = D_1(z) - A(z) \, \left( C^\mu(z) \right)^2,
\end{align*}
where
\begin{equation}\label{D1}
D_1(z) = \iint    D(x,y,z) \, d\mu (x ) d\mu (y )\,.
\end{equation}
Thus, by \eqref{firstIntegral},
\begin{equation}\label{part1}
\lim_{r\to 0+, \, r\in \Delta}  \lim_{\varepsilon \to 0+}    \iint    \frac{h_\varepsilon (x; r)-h_\varepsilon (y; r) }{x -y }   \, d\mu (x ) d\mu (y )   = D_1(z) - A(z) \, \left( C^\mu(z) \right)^2.
\end{equation}

In a similar fashion we can analyze
$$
  \int  \Phi'(x)h_\varepsilon (x;r) \, d\mu(x) = \left(   \int_{\mathcal K_{r,\varepsilon}} + \int_{\mathcal F_{r,\varepsilon}}\right) \Phi'(x)h_\varepsilon (x;r) \, d\mu(x).
$$
Again estimates on $\mathcal K_{r,\varepsilon}$ and \eqref{existenseSimDerivatives}  show that
$$
 \lim_{\varepsilon \to 0+} \int  \Phi'(x)h_\varepsilon (x;r) \, d\mu(x) =    \int_{|x-z|\geq r}  \Phi'(x)\,  \frac{A(x)}{x-z} \, d\mu(x).
$$
Taking into account \eqref{Phiforcharges} we can rewrite the right hand side as
$$
-\frac{1}{2}\, \int    \frac{B(x)}{x-z} \, d\mu(x)= -\frac{1}{2}\,\int_{|x-z|\geq r}    \frac{B(x)-B(z)}{x-z}  \, d\mu(x)- \frac{1}{2}\,  B(z)\int_{|x-z|\geq r}  \frac{1}{x-z}  \, d\mu(x).
$$
Thus,
$$
\lim_{r\to 0} \lim_{\varepsilon \to 0+} \int  \Phi'(x)h_\varepsilon (x;r) \, d\mu(x) = -\frac{1}{2}\,\left(  D_2(z)+  B(z)C^\mu(z) \right),
$$
where
 \begin{equation}\label{D1yD2}
D_2(z)= \int \frac{B(x)-B(z)}{x-z}  \, d\mu(x)
\end{equation}
is a polynomial of degree $\leq p-1$; it is $\equiv 0$ if $\varphi \equiv 0$.

Combining this last identity with \eqref{part1} and using \eqref{defFvar} we get that
\begin{equation}\label{limitInVariation}
\lim_{r\to 0} \lim_{\varepsilon \to 0+} f_\varphi(\mu;h_\varepsilon (\cdot ;r)) = D_1(z) - A(z) \, \left( C^\mu(z) \right)^2 +  D_2(z)+  B(z)C^\mu(z).
\end{equation}

Since \eqref{condicionCrit} is valid for each $\varepsilon >0$ and $r\in \Delta$, we obtain that the right hand side in \eqref{limitInVariation} is $0$. In consequence,
$$
\left( A (z) C^\mu(z) \right)^2 - A(z) B(z)\, C^\mu(z) +
B^2(z)/4=A(z)\, \left(D_1(z) + D_2(z)\right)+B^2/4(z)\,.
$$
Taking into account \eqref{Phiforcharges} we rewrite this condition as \eqref{charactMeasure}, with
$$
R(z) \isdef \frac{D_1(z) + D_2(z)}{A(z)}+\left( \Phi'(z)\right)^2\,,
$$
and clearly $R$ has a double pole at $z=a_j$ if and only if $\rho_j\neq 0$. Finally, the normalization condition \eqref{normalizationR} follows from considering \eqref{charactMeasure} as $z\to \infty$.
This establishes the assertion of the theorem.
\end{proof}

The next proposition is a slightly modified version of \cite[Lemma 4]{MR1909635} by T.~Bergkvist and H.~Rullg{\aa}rd. The original lemma considers only positive measures $\mu$ for which $(C^\mu)^k$, for certain $k\in \N$, is a reciprocal of a polynomial.

\begin{lemma}\label{rulgard} 
Assume that $\mu\in \MM_\R$ is a finite signed Borel measure on the plane whose Cauchy transform $C^\mu$ is such that that there exist rational functions $r$ and $R$ with possible poles at $\AA$ satisfying
\begin{equation}\label{squareofCauchyRulgards}
    \left( C^\mu + r  \right)^2(z) = R(z) \qquad \lebesgue_2-\text{a.e.}
\end{equation}
Then $\mu$ is supported on a union of analytic arcs, that are trajectories of the  quadratic differential $\varpi(z) = -R(z) dz^2$, with possible mass points at $\AA$.

If all poles of $r$ are simple and have real residues, then additionally $\varpi$ is a closed differential, and the number of connected components of $\supp(\mu)$ is finite.

Finally, if $\mu$ is a positive Borel measure, then the intersection of any  $\varpi$-rectangle, not containing the zeros or poles of $R$, with $\supp(\mu)$ is connected.
\end{lemma}
\begin{proof}
For the quadratic differential $\varpi$ consider a  $\varpi$-rectangle $D$ (see the definition in Section \ref{sec:qd}), disjoint with $\AA$ and not containing the zeros of $R$. We select in this rectangle a holomorphic branch of $\sqrt{R}$ and a distinguished parameter
\begin{equation}\label{definitionofxi}
     \xi =\xi (z)  =\int^z  \sqrt{R(t)}\, dt\,,
\end{equation}
which is a conformal mapping of $D$ onto $\widehat D \isdef \xi(D)$.

Let us define in $\widehat D$ the following function:
$$
\chi ( \xi) \isdef \sgn \left( \frac{C^\mu + r}{\sqrt{R}}(z(\xi))\right).
$$
Hence, $\chi$ takes only two values, $\pm 1$, and for $z\in D$,
\begin{equation}\label{rulgard1}
\left( C^\mu + r\right)(z) = \chi(\xi(z)) \sqrt{R}(z).
\end{equation}
We have that in the sense of generalized derivatives, for $z\in D$,
$$
\frac{\partial}{\partial \overline z}\, \left( C^\mu + r\right)(z) =\pi \mu(z), \qquad \frac{\partial}{\partial \overline z} = \frac{1}{2}\, \left(\frac{\partial}{\partial x} + i \frac{\partial}{\partial y}\right).
$$
Differentiating in \eqref{rulgard1} and using the chain rule, we get
$$
-\pi \mu(z) = \frac{\partial}{\partial \overline z}\, \left(\chi(\xi(z)) \sqrt{R}(z) \right) = \frac{\partial}{\partial \overline z}\, \left(\chi(\xi(z))\right) \sqrt{R}(z) = \frac{\partial \chi(\xi)}{\partial \overline \xi}\, \overline{\left(\frac{\partial  \xi (z)}{\partial z} \right)}   \sqrt{R}(z).
$$
Taking into account the definition of $\chi$ in \eqref{definitionofxi} we conclude that if $z\in D$ and $\xi = \xi(z)$, then
\begin{equation}\label{mainIdentityRulgard}
    \frac{\partial \chi(\xi)}{\partial \overline \xi} =-\frac{\pi \mu(z)}{|R(z)|}\,.
\end{equation}
In particular, the (generalized) partial derivative of $\chi(\xi)$ along the vertical axis is zero; if $\xi =x+iy$, this implies that $\chi(\xi)$ is equivalent to a function $g(x)$ which takes only values $+1$ and $-1$. 
Thus, the set
$$
\left\{\xi:\,  \frac{\partial \chi(\xi)}{\partial \overline \xi}\neq 0 \right\}
$$
is a union of vertical arcs in the $\xi$-plane. From \eqref{mainIdentityRulgard} it follows that the number of these arcs is finite. This means that the image of the support of $\mu$ in $D$ by \eqref{definitionofxi} is made of vertical lines, that is, $\supp (\mu) \cap D$ is a union of horizontal trajectories of $\varpi$. Moreover, if $\mu$ is positive, we get by \eqref{mainIdentityRulgard} that
$$
\frac{\partial \chi(\xi)}{\partial \overline \xi} \geq 0 \text{ in } \widehat D,
$$
so that $\chi(\xi)$ changes sign at most once in $\widehat D$. In other words, $\supp (\mu) \cap D$ contains at most one single analytic arc, which is a horizontal trajectory of $\varpi$.

Finally, if all residues of $r$ are real, then
$$
\Re \int^z \left( C^\mu + r\right)(t) dt =U^\mu(z)+ \Re \int^z   r (t) dt +\const
$$
is harmonic and single-valued in $\C\setminus \supp (\mu) $, which means that the trajectories of $\varpi$ are either critical or level curves of a harmonic function. Thus, $\varpi$ is closed.
\end{proof}
\begin{rmk}
See Figure \ref{fig:intersectionNegative} for an illustration of the statement about the connectedness of the intersection of a  $\varpi$-rectangle with $\supp(\mu)$. Obviously, in this assertion we can replace the $\varpi$-rectangle by any simply-connected domain that is mapped one-to-one by $\xi$ onto a convex set.
\end{rmk}
\begin{figure}[htb]
\centering \begin{tabular}{ll} \hspace{-1.8cm}\mbox{\begin{overpic}[scale=0.7]%
{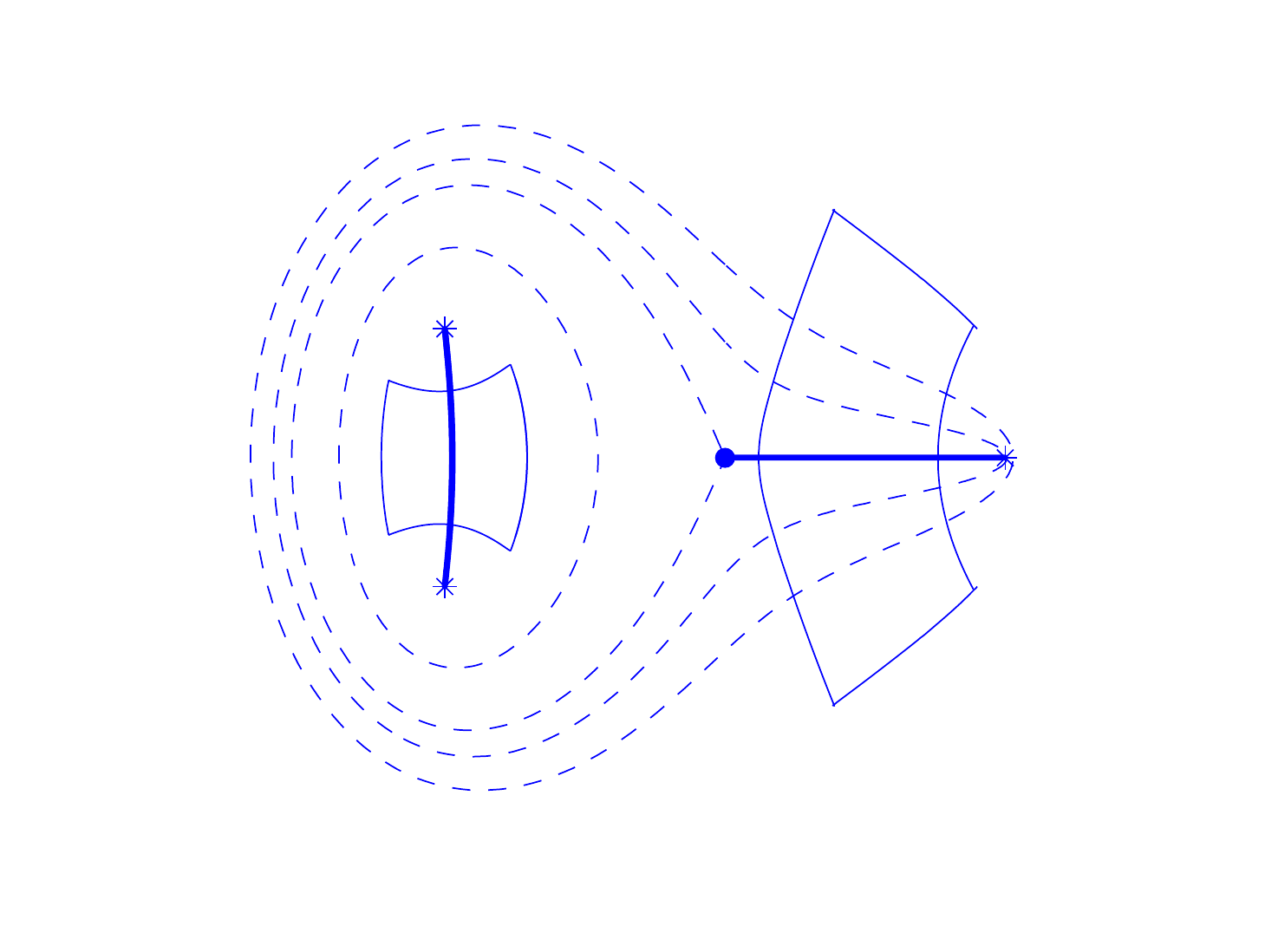}%
\end{overpic}} &
\hspace{-1.5cm}
\mbox{\begin{overpic}[scale=0.7]%
{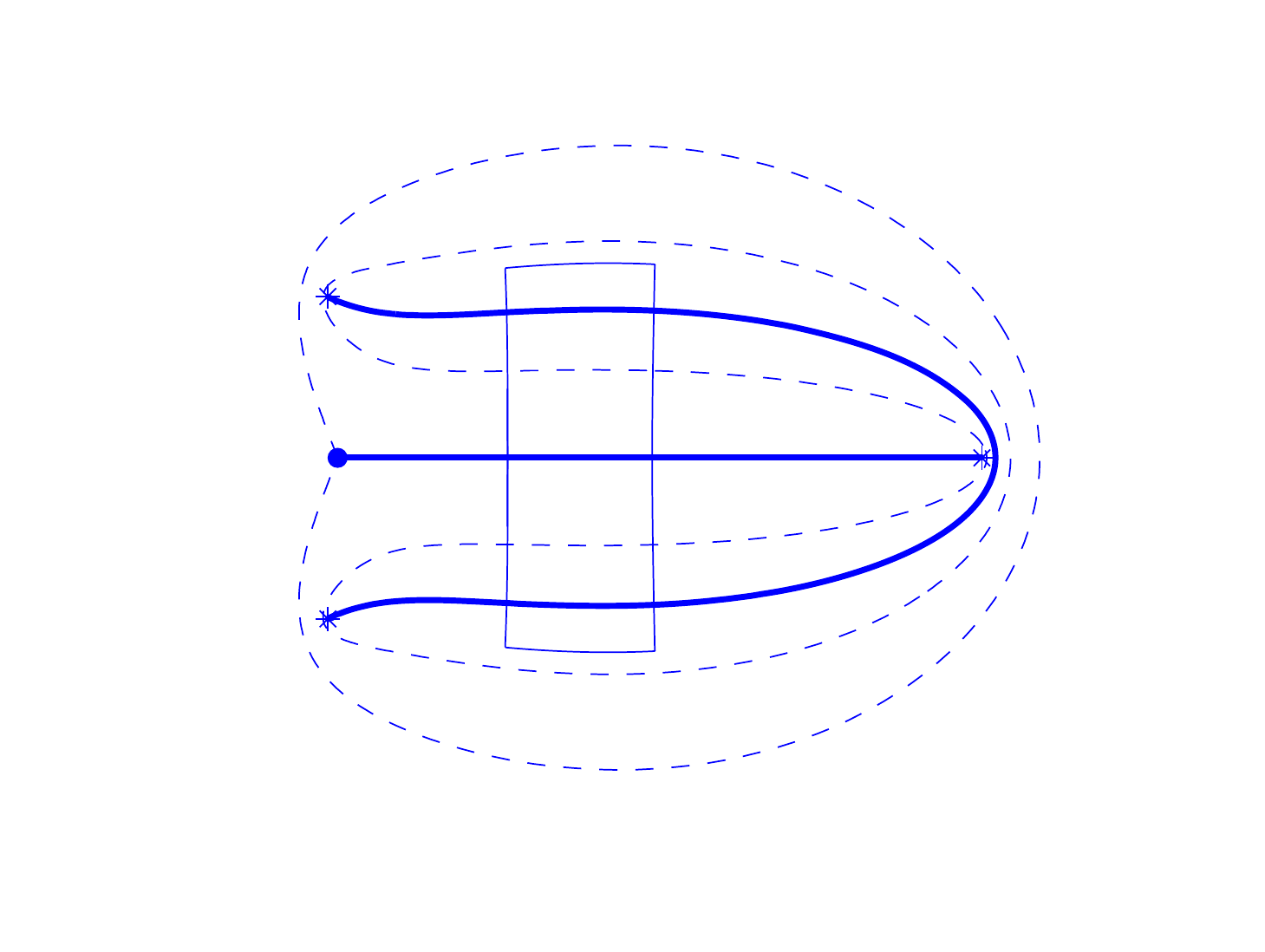}%
\end{overpic}}
\end{tabular}
\caption{$\varpi$-rectangles intersecting the support of a positive (left) and sign-changing measure (right); for further details, see Section \ref{sec:positiveAcritical}.}
\label{fig:intersectionNegative}
\end{figure}

\begin{rmk}
Observe that in Lemma \ref{rulgard} we do not assume a priori that $E_\varphi(\mu)<\infty$, so that mass points of $\mu$ are allowed.
\end{rmk}

An immediate consequence of Lemma \ref{thm:rakhmanov} and Lemma \ref{rulgard} is
\begin{corollary}\label{cor:mainthm}
If in \eqref{Phiforcharges}, $\{\rho_0, \dots, \rho_p\}\subset \R$, then in any connected component of $\C\setminus \supp(\mu )$  we can select a single-valued branch of $\sqrt{R }$ such that there the formula
\begin{equation}\label{potentialMeasure}
U^\mu(z)+\varphi(z)=-\Re \int^z \sqrt{R(t)} \, dt
\end{equation}
holds.
\end{corollary}

\subsection{Critical and reflectionless measures}

By Theorem \ref{main_thm}, $\supp(\mu)$ is a union of analytic arcs.
\begin{definition}\label{def:regularpoint}
 We call a point $z\in \supp( \mu  )$ \emph{regular} if there exists a simply connected open neighborhood $B$ of $z$ such that $B\cap \supp(\mu )$ is a Jordan arc.
\end{definition}

\begin{lemma}\label{lemma:onthesup}
Let $\mu$ be an $(\AA, \varphi)$-critical measure. Then
the principal value of the Cauchy transform of $\mu$ \eqref{principalvalue} satisfies
\begin{equation}\label{zeroOnSupp}
    C^\mu(z)+\Phi'(z)=0, \quad z\in \supp(\mu )\setminus \AA.
\end{equation}
On any simple subarc of $\supp(\mu)$ measure $\mu$ is absolutely continuous with respect to the arc-length measure, and its density is given by
\begin{equation}\label{density}
d\mu(z)= \frac{1}{\pi }\, \left|   \sqrt{R(z)} \, dz \right| .
\end{equation}
\end{lemma}
\begin{rmk}
We can reformulate \eqref{zeroOnSupp} as
$$
 C^\mu +\Phi' =0 \quad \text{$\mu$-a.e.\ on } \C.
$$
\end{rmk}
\begin{proof}
Assume first that $z\in \supp( \mu  )$ is regular, and let $B$ be a simply connected open neighborhood $B$ of $z$ such that $B\cap \supp(\mu )$ is an open analytic arc not containing $\AA$. By Lemma \ref{thm:rakhmanov}, 
$$
\left(C^\mu(\zeta ) + \Phi'(\zeta )
\right)^2=R(\zeta ), \quad \zeta \in \C\setminus \supp(\mu ),
$$
where the Cauchy transform is understood in the strong sense (ordinary integral),
so that
\begin{equation}\label{formulaforCmu}
    C^\mu(\zeta )=-\Phi'(\zeta ) +\sqrt{R(\zeta )}
\end{equation}
in each connected component of $B\setminus \supp(\mu)$, with an appropriate selection of the branch of the square root.

At a regular point $z$ the boundary values of $C^\mu$ from both sides of $\supp(\mu)$,
$$
C^\mu_\pm (z)\isdef \lim_{\zeta \to z^\pm, \, \zeta \in \C\setminus \supp(\mu)} C^\mu(\zeta ),
$$
are well defined and satisfy the Sokhotsky-Plemelj relations,
\begin{equation}\label{sokhotsky}
    C^\mu_+(z)-C^\mu_-(z)=2\pi i \mu'(z), \qquad C^\mu_+(z)+C^\mu_-(z)=2C^\mu(z).
\end{equation}
By \eqref{formulaforCmu}, if the branch of $\sqrt{R}$ coincides on both sides of $B\cap \supp(\mu)$, then $\mu\equiv 0$ there, which is impossible. Hence, with an appropriate selection of the branch of $\sqrt{R}$ in $B$,
$$
  C^\mu(z )=-\Phi'(z ) \pm \sqrt{R(z )}, \quad z\in B\cap \supp(\mu),
$$
and both \eqref{zeroOnSupp} and \eqref{density} follow from \eqref{sokhotsky}.

Finally, if $z\in \supp(\mu)\setminus \AA$ is not regular, it must coincide with a zero of the rational function $R$ in \eqref{charactMeasure}. Taking into account the expression for the density of the measure we see that it vanishes at $z$ at least as the square root of $(\zeta -z)$. Then, at this point \eqref{localconditionCauchy} holds, as well as formula \eqref{charactMeasure}. This concludes the proof of \eqref{zeroOnSupp}.
\end{proof}

Formula  \eqref{zeroOnSupp} is a direct continuous analogue of property \eqref{almostLame1} for the discrete critical measures, and it may be proved directly (independently from Theorem \ref{main_thm}) using local variations ($h\equiv 0$ outside of a small neighborhood of the singularity at $z$). The original proof of Theorem \ref{main_thm} by the second author (unpublished) was based on a combination of Lemmas \ref{thm:rakhmanov} and \ref{lemma:onthesup}. The technique from \cite{MR1909635} used in Lemma \ref{rulgard} above streamlines the arguments.

Observe that for $\varphi\equiv 0$ we obtain from Lemma \ref{lemma:onthesup} that for any $\AA$-critical measure $\mu$, and for any regular point $z\in \supp(\mu)$, $C^\mu(z)=0$, so that $C^\mu=0$ $\mu$-a.e. Measures with this property are called \emph{reflectionless}; see \cite{Melnikov08}, where they are treated in the context of the geometric function theory. 

\begin{rmk} 
\label{rmk:reflectionlessCMV}
For measures supported on the unit circle $\mathbb T\isdef \{z \in \C:\, |z|=1 \}$ we can find in literature an alternative definition of the ``reflectionlessness'' (see \cite{Gesztesy2006a}, \cite{Gesztesy2006}, \cite{Gesztesy2009}), characterized by vanishing of the sum of the boundary values of the Carath\'{e}odory function of $\mu $ and not of its Cauchy transform. For instance, the Lebesgue measure $\lebesgue _1$ on $\mathbb T$ is reflectionless in the sense of \cite{Gesztesy2006a}, but does not satisfy $\mathrm{v.p.}\, C^{\lebesgue _1}=0$ on $\mathbb T$. In this paper we give the ``reflectionless measure'' the meaning specified above.
\end{rmk}

Reflectionless measures on $\R$ have their origin in the spectral theory; they are in fact spectral measures of Schr\"{o}dinger self-adjoint operators with reflectionless potentials (see e.g.~\cite{Craig1989}) and of reflectionless Jacobi operators \cite{Teschl99}.

It is immediate to show that there are reflectionless measures with infinite energy. So, the following conjecture seems natural:
\begin{conjecture}
For any positive reflectionless measure $\mu$ with a finite energy we can find a polar set $\AA$ on $\C$ such that $\mu$ is $\AA$-critical.
\end{conjecture}
However, a weaker statement follows from Proposition \ref{rulgardAdd} below: assume that $\mu$ is a reflectionless measure supported on a finite set of analytic arcs $\Gamma=\Gamma_1\cup \dots\cup \Gamma_k$ of $\C$, and that $\mu$ is absolutely continuous with respect to the arc-length measure. Then there exists a discrete finite set $\AA$ such that this measure is $\AA$-critical.

\subsection{Equilibrium conditions in terms of potentials, and the $S$-property} \label{sec:ChebotarevStahl}

The variational requirements defining a continuous $(\AA, \varphi)$-critical measure $\mu$ impose equilibrium conditions that we discuss next. We will see that the gradient of the total potential (that is, force) vanishes at any regular point of the support of the measure located in the conducting part of the plane. However, $\grad (U^\mu + \varphi)$ is not continuous across any arc in $\supp(\mu )$, and we have to consider separately the force acting on an element of charge from either side of $\supp(\mu)$. This leads to equality of the normal derivatives, the so called \emph{$S$-property}, see \eqref{Sprop} below.

\begin{lemma}
\label{lemma:Sconditions}
The total potential of an $(\AA, \varphi)$-critical measure $\mu$ satisfies the following properties:
\begin{enumerate}
\item[(i)] if $\supp(\mu )=\Gamma_1\cup \dots\cup \Gamma_k$, where $\Gamma_j$ are the connected components of $\supp(\mu )$, then
    $$
    U^\mu(z) + \varphi(z) = w_j=\const, \quad z\in \Gamma_j, \quad j=1, \dots, k.
    $$
\item[(ii)] at any regular point $z\in \supp(\mu )$,
\begin{equation}\label{Sprop}
    \frac{\partial   }{\partial n_+}\, \left( U^\mu +\varphi\right)(z)= \frac{\partial   }{\partial n_-}\, \left( U^\mu +\varphi\right)(z),
\end{equation}
where $n_\pm$ are the normal vectors to $\supp(\mu )$ at $z$ pointing in the opposite directions.

Additionally, if $z\in \supp(\mu )\setminus \AA$ is not regular, then
\begin{equation}\label{zeroatnotreg}
    \grad \left( U^\mu(z) + \varphi(z)\right)=0.
\end{equation}
\end{enumerate}

Furthermore, assume that a finite real measure $\mu$, whose support $\supp(\mu)$ consists of a union of a finite set of analytic arcs, $\supp(\mu)=\Gamma_1\cup \dots\cup \Gamma_k$, satisfies conditions (i) and (ii) above. Then $\mu$ 
satisfies an equation of the form \eqref{charactMeasure}, where $R$ is a rational function with possible poles at $\AA$ of order $\leq 2$.
\end{lemma}
\begin{remark}
As it follows from the last statement of this Lemma and Remark \ref{rem:sufficientcondition}, conditions \emph{(i)} and \emph{(ii)} are sufficient for $\mu$ being $(\AA,\varphi)$-critical.
\end{remark}
\begin{proof}
Let $\mu$ be an $(\AA, \varphi)$-critical measure. From Theorem \ref{main_thm} and Lemma \ref{lemma:onthesup} it follows that $\mu$ has an analytic density on the regular points of its support, made of analytic curves. Hence, $U^\mu$ is continuous up to the boundary, and \emph{(i)} is a direct consequence of Corollary \ref{cor:mainthm} and the fact that $\mu $ lives on trajectories of the quadratic differential $-R(z)(dz)^2$.

For any $z\in \C\setminus \supp(\mu)$ we have that
$$
\frac{\partial }{\partial \overline{z}}\, \left( U^\mu(z)+\varphi(z)\right)=\frac{1}{2}\, \overline{\left( C^\mu(z)+\Phi'(z)\right)},
$$
and this relation is inherited by the limit values on the Carath\'{e}odory boundary of $\C\setminus \supp(\mu)$. Using \eqref{zeroOnSupp} we conclude that on regular points of $\supp(\mu)\setminus \AA$,
$$
\frac{\partial }{\partial \overline{z}}  \left( U^\mu(z)+\varphi(z)\right)_+ + \frac{\partial }{\partial \overline{z}}  \left( U^\mu(z)+\varphi(z)\right)_- =0  .
$$
Observe that  for a real valued function $u$, $\partial u(z)/(\partial \overline{z})$ coincides, up to a factor $1/2$, with the (complex) gradient of $u$. From \emph{(i)} it follows that $\grad(U^\mu +\varphi)$ is normal to $\supp(\mu)$, so that
$$
\left\| \grad\left( U^\mu(z) +\varphi(z)\right)\right\|=\left| \frac{\partial }{\partial n_\pm}  \left( U^\mu(z) +\varphi(z)\right) \right|, \quad z\in \supp(\mu)\setminus \AA,
$$
and \eqref{Sprop} follows from the last two equalities.

Finally, the only possible non-regular points of $\supp(\mu)\setminus \AA$ are the zeros of $R$, where the density vanishes, and we conclude \eqref{zeroatnotreg}.

Let us prove the reciprocal. Assume that $\mu\in \mathcal M_\R$ has a bounded support comprised of a finite number of smooth linear connected components $\Gamma_j$. Fix a regular point $\zeta\in \supp(\mu)$ (without loss of generality, $\zeta \in \Gamma_j$), and let again $B$ be a simply connected open neighborhood  of $\zeta $ such that $B\cap \supp(\mu )$ is a Jordan arc. It splits $B$ into two disjoint domains, that we denote by $B^\pm$, so that $B\setminus \supp(\mu)=B^+ \cup B^-$. It follows from \emph{(i)} and \eqref{Sprop} that
$$
U(z)\isdef \begin{cases}
U^\mu(z)+\varphi(z)-w_j, & \text{if } z\in B^+, \\
-\left( U^\mu(z)+\varphi(z)-w_j\right), & \text{if } z\in B^-, \\
0, & \text{if } z\in B\cap \supp(\mu),
\end{cases}
$$
is harmonic in $B$. Equivalently, $C(z)\isdef  \partial U(z)/( \partial z)$ is holomorphic in $B$. But
\begin{equation}\label{defCaux}
    C(z)=\begin{cases}
C^\mu(z)+\Phi'(z) , & \text{if } z\in B^+, \\
-\left( C^\mu(z)+\Phi'(z) \right), & \text{if } z\in B^-,
\end{cases}
\end{equation}
is continuous in $B$. It implies that $R(z)\isdef (C^\mu+ \Phi')^2(z)$ is holomorphic in $B$, and in consequence, $R$ is analytic at any regular point of $\supp(\mu)$. Since $R$ is obviously analytic also in $\C\setminus \supp(\mu)$, we conclude that it is in fact holomorphic in $\C$, except for the set of irregular points of $\supp(\mu)$, that is, the endpoints of the arcs comprising $\supp(\mu)$. By \eqref{zeroatnotreg} and \eqref{defCaux}, $R$ vanishes also at the irregular points of $\supp(\mu)\setminus \AA$ and at infinity. Since $\mu$ is a finite measure with finite energy, $C^\mu$ has a sub-polar growth at any point of $\C$. Altogether it means that $R$ is a rational function with possible poles at $\AA$ (of order $\leq 2$). 
\end{proof}

\subsection{Correspondence of critical measures with closed quadratic differentials} \label{sec:correspondence}

We begin with some general remarks on $(\AA, \varphi)$-critical measures for the case \eqref{potentialDiscrete}, when all $\rho_k \in \R$ and the external field corresponds to the potential of a discrete signed \emph{real} measure supported on $\AA$. According to Theorem \ref{main_thm}, for any $(\AA, \varphi)$-critical measure $\mu$ there exists a closed quadratic differential $\varpi=-R(z)dz^2$ such that $\supp(\mu)$ consists of a finite union of its trajectories.

This is not a one-to-one correspondence: even in the class $\mathcal M$ the same quadratic differential may correspond to a whole family of critical measures.
\begin{example}
\label{example:1to1_1}
Let $p=0$, $a_0=0$, and $\varphi(z)= \frac{1}{2}\, \log|z|$ (generated by a charge $-1/2$ at the origin). Then for any $r>0$ the normalized angular (Lebesgue) measure $\lebesgue _1$ living on the circle $|z|=r  $ is $(\AA, \varphi)$-critical. Each such a measure is supported on a trajectory of the same quadratic differential $-(dz)^2 /z^2$ (for a discrete analogue of this statement, see Remark \ref{rmk:discreteunbounded}).
\end{example}
\begin{example}
\label{example:1to1_2}
In a more general situation we can consider an $\AA$-critical measure $\mu$ for an arbitrary configuration $\AA$ without external field ($\varphi\equiv 0$); the trajectories of the associated quadratic differential $\varpi$ near infinity are closed Jordan curves. Select any such a trajectory $\beta$ containing in the bounded component of its complement both $\AA$ and $\supp(\mu)$, and denote by $\widehat \mu$ the balayage of $\mu$ onto $\beta$ (see the definition in \cite[\S II.4]{Saff:97}). Then $\mu_1 \isdef 2\widehat \mu-\mu$ is another $\AA$-critical measure with the same total mass than $\mu$, and such that $\beta\subset \supp(\mu_1)$; observe that $\mu_1$ also corresponds to the same quadratic differential. The verification of this assertion is a simple exercise.
\end{example}
What these very basic examples have in common is that in each case when we were able to construct more than one $\AA$-critical measure associated with the same quadratic differential, closed trajectories in the support of measures were present. Moreover, we have seen that an infinite family of critical measures may correspond to the same quadratic differential. This is not the case if we restrict ourselves to critical measures with a connected complement to the support. However, even in this situation the quadratic differential can give rise to more than one (signed) critical measure, as the following example shows.
\begin{example}
\label{example:1to1_4}
Consider the quadratic diffferential $\varpi$ with 4 simple poles $a_{0}, \dots, a_{3}$ at the vertices of a rectangle, and two simple zeros $v_{1}, v_{2}$, situated symmetrically at the midpoints of the longest sides of this rectangle. The trajectories of such a differential are depicted in Figure~\ref{fig:several}. 
\begin{figure}[htb]
\centering \begin{overpic}[scale=0.55]%
{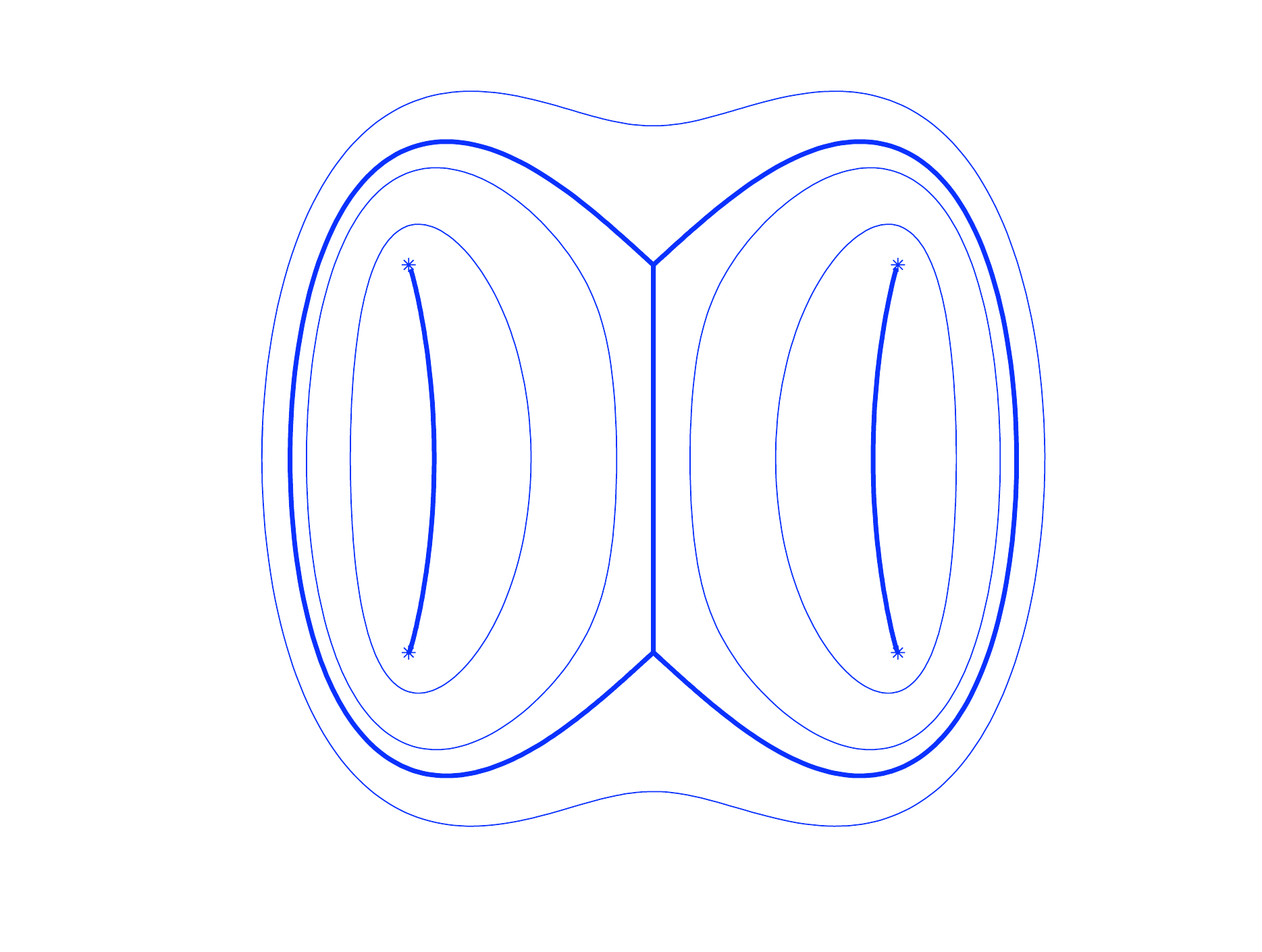}%
   \put(33,23){\small $a_0 $}
 \put(33,54.5){\small $a_1$}
    \put(67,23){\small $a_3 $}
 \put(67,54.5){\small $a_2$}
  \put(50.5,20.5){\small $v_2$}
   \put(50.3,57){\small $v_1$}
\end{overpic}
\caption{Trajectories of the quadratic differential $\varpi$ described in Example \ref{example:1to1_4}. Bold lines represent the critical graph of $\varpi$.}
\label{fig:several}
\end{figure}
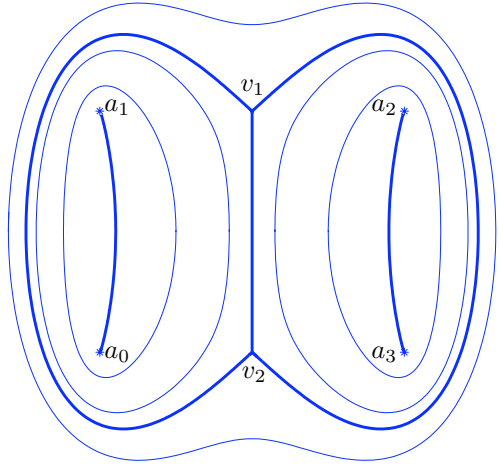
We can associate to $\varpi$ three different $\AA$-critical measures with a connected complement of their supports. Indeed, although the critical trajectories joining poles will always belong to the support of any such a critical measure, for the third component of the support we can choose any of critical trajectories connecting both zeros $v_{j}$.
\end{example}

All these examples illustrate the general difficulty of the analysis of the correspondence between quadratic differentials and critical measures. Nevertheless, we will show below that in the class of \emph{positive} $(\AA, \varphi)$-critical measures $\mu$ corresponding to $\varphi$ generated by a \emph{positive} measure,  the mapping associating to such a $\mu$ the quadratic differential $\varpi$ described in Theorem \ref{main_thm} is an injection. Moreover, $\C \setminus \supp(\mu)$ is connected. The assumption of positivity of the mass giving rise to $\varphi$ is necessary, as the following example (first considered by Teichm\"{u}ller in his ``Habilitationsschrift'' \cite{Teichmuller1938}) shows.
\begin{example}
 \label{example:1to1_3}
 Let $\AA=\{0, 1\}$, $\sigma=-\alpha \delta_0$, $0<\alpha <1$ (negative ``attracting'' charge). Then there exists a unique positive critical measure $\mu$ with $\mu(\C)=1+\alpha $. The corresponding quadratic differential is
\begin{equation}\label{teichqd}
    -\frac{(z-c) \, (dz)^2}{z^2 (z-1)}, \quad c = c(\alpha )\in (0,1).
\end{equation}
The $\supp(\mu)$ is the whole critical graph of this differential, which consists of the segment $[c,1]$ and a closing loop passing through $c$ and enclosing the origin (see Figure \ref{fig:teich}).
\begin{figure}[htb]
\centering \begin{overpic}[scale=0.5]%
{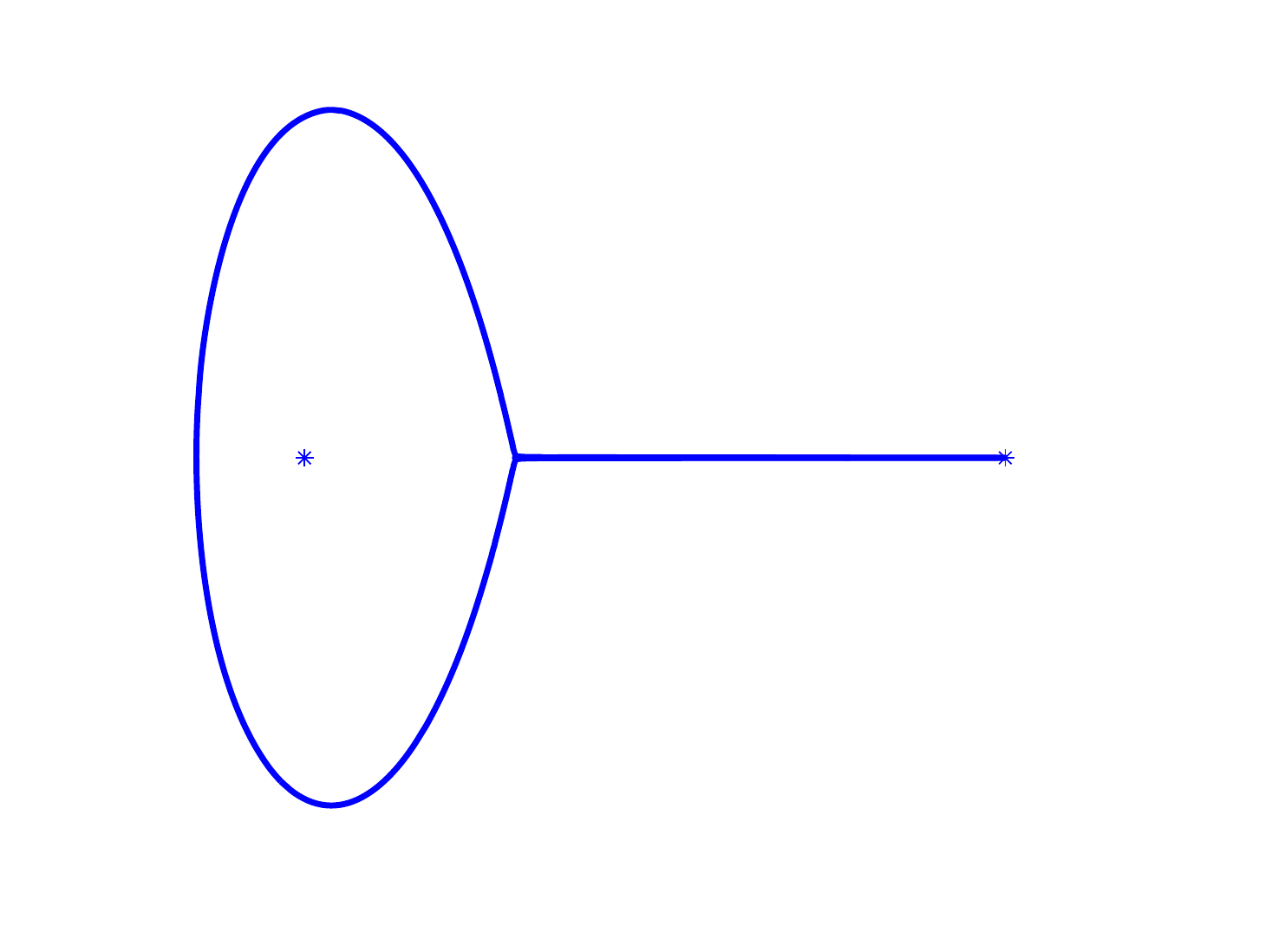}%
\put(26,37.5){$0 $}
       \put(37,37.5){$c $}
   \put(82,37.5){$1 $}
\end{overpic}
\caption{Critical graph of the quadratic differential given in \eqref{teichqd}.}
\label{fig:teich}
\end{figure}
Thus, we do not have critical measures with a connected complement to $\supp(\mu)$. Moreover, if we now consider an  external field of an opposite sign, $\widetilde \varphi=U^{- \sigma}$ (positive ``repelling'' charge), then the corresponding $(\AA, \widetilde \varphi)$- critical measure $\widetilde \mu$ is associated with the same quadratic differential \eqref{teichqd}, but now $\supp(\widetilde \mu)=[c,1]$, and $\widetilde \mu(\C)=1-\alpha $.
We point out that this is not a mere artificial example: the whole variety of the critical measures in a similar situation appears in the asymptotic analysis of the Jacobi polynomials with varying non-standard parameters, see e.g.~\cite{MR2124460} and \cite{MR2142296}.
\end{example}

We present here a lemma that will allow us to isolate the cases that are of our interest.
\begin{lemma}
\label{lemma:closedsupp}
Let $\mu\in \mathcal M_\R$ be an $(\AA, \varphi)$-critical measure for $\varphi=U^\sigma$, $\sigma\in \mathcal M_\R$, and assume that $\beta$ is a   
closed contour contained entirely in $\supp(\mu)$, delimiting the bounded domain $\Omega$. Then
$$
-\mu(\beta) = 2(\mu + \sigma) (\Omega).
$$
\end{lemma}
\begin{proof}
Let $n_-$ be the unit outer normal vector to $\beta$, and $n_+=-n_-$. By Gauss theorem (see e.g.~Theorem 1.1, \S II.1 of \cite{Saff:97}),  
\begin{align*}
 \frac{1}{2\pi}\, \oint_\beta \frac{\partial}{\partial n_-}\, \left( U^\mu + \varphi \right)(z) |dz| & =\frac{1}{2\pi}\, \oint_\beta \frac{\partial}{\partial n_-}\,   U^{\mu + \sigma}  (z) |dz|= (\mu + \sigma) (\beta \cup \Omega), \\
  \frac{1}{2\pi}\, \oint_\beta \frac{\partial}{\partial n_+}\, \left( U^\mu + \varphi \right)(z) |dz| & =\frac{1}{2\pi}\, \oint_\beta \frac{\partial}{\partial n_+}\,   U^{\mu + \sigma}  (z) |dz|=- (\mu + \sigma) (  \Omega).
\end{align*}
By the $S$-property \eqref{Sprop}, both integrals in the left hand side are equal, and the Lemma follows since $\sigma(\beta)=0$.
\end{proof}
\begin{proposition}
\label{uniquenessmeasure}
If $\varphi=U^\sigma$, with $\sigma\geq 0$, then
\begin{enumerate}
\item[(i)] the support of any positive $(\AA, \varphi)$-critical measure has a connected complement; 
\item[(ii)] the correspondence between the positive $(\AA, \varphi)$-critical measures and the associated quadratic differentials is injective.
\end{enumerate}
\end{proposition}
\begin{proof}
Statement \emph{(i)} is an obvious consequence of Lemma~\ref{lemma:closedsupp}.

Let now $\mu$ be a positive $(\AA, \varphi)$-critical measure and let $\varpi=-R(z)dz^{2}$ be the quadratic differential whose critical trajectories support $\mu$ (see Theorem \ref{main_thm}). We have to prove that $\varpi$ determines $\mu$ uniquely.

Let $\Gamma$ be the critical graph of $\varpi$, and denote by $Crit(\varpi)$ the class of all (signed) $(\AA, \varphi)$-critical measures $\widetilde \mu$ corresponding to $\varpi$ and such that $\C\setminus \supp(\widetilde \mu)$ is connected; by \emph{(i)}, $\mu\in Crit(\varpi)$. If both sides of a critical trajectory $\gamma \subset \Gamma$ belong to the boundary of the same connected component of $\C\setminus \Gamma$, then either one of these disjoint possibilities holds:
\begin{itemize}
\item $\gamma$ belongs to the support of every measure from $  Crit(\varpi)$;
\item $\gamma$ is not contained in the support of any measure  from $ Crit(\varpi)$.
\end{itemize}
Indeed, by \eqref{squareofCauchyRulgards}, any measure $\widetilde \mu\in Crit(\varpi)$ is recovered from its support using the Sokhotsky-Plemelj formulas. Since both sides of $\gamma$ belong to a connected complement of  $\C\setminus \Gamma$, either possibility above is determined by the analytic continuation of $\sqrt{R}$: $\gamma\in \supp(\widetilde \mu)$ if and only if $\sqrt{R}$ will have opposite signs on both sides of $\gamma$. This is obviously the case if any closed curve, contained in this connected component and joining both sides of $\gamma$, encloses an odd number of singular points of $\varpi$. 

As a corollary, we conclude that any critical trajectory emanating from a simple pole of $\varpi$, must belong to the support of any $\widetilde \mu\in Crit(\varpi)$.

Let now $\gamma$ be a critical trajectory joining two different zeros of $\varpi$. Each side of $\gamma$ is a boundary of a ring domain filled with closed trajectories of $\varpi$. If both sides of $\gamma$ are in the boundary of the same ring domain, then considerations above apply. So,   it remains to consider the case when $\gamma$ is in the outer boundary of a ring domain  $\Omega$ (see e.g.~the middle arc joining $v_{1}$ and $v_{2}$ in Figure~\ref{fig:several}). Let $\widetilde \gamma$ be a closed trajectory from $\Omega$, and denote by $\lambda$ the restriction of $\mu+\sigma$ contained inside $\widetilde \gamma$. Since $\lambda$ is by assumption a positive measure, the gradient (flux) of the potential $U^{\lambda}$ on $\widetilde \gamma$ is directed inwards. By continuity, this also happens on $\gamma$, so the restriction of $\mu+\sigma$ to $\gamma$ cannot be positive. In conclusion, $\gamma$ does not belong to $\supp(\mu)$, and it finishes the proof of \emph{(ii)}.
\end{proof}

\begin{rmk}
In fact, the following stronger uniqueness property holds (formulated in the notation introduced in the proof above): if for a quadratic differential $\varpi$, the set $Crit(\varpi)$ contains more than one measure, then none of the measures in $Crit(\varpi)$ is positive. This is the case, for instance, of Example \ref{example:1to1_4}.
\end{rmk}

Combining Lemma~\ref{lemma:Sconditions} and Proposition \ref{uniquenessmeasure} and following the logic of Remark \ref{rem:sufficientcondition} we get an addendum to Lemma~\ref{rulgard}, which has an independent interest:
\begin{proposition}\label{rulgardAdd} 
Assume that $\mu\in \MM$ is a finite positive Borel measure on the plane whose Cauchy transform $C^\mu$ is such that that 
$$
    \left( C^\mu    \right)^2(z) = R(z) \qquad \lebesgue_2-\text{a.e.} 
$$
for a rational function $R$. Denote by $\AA$ the set of poles of $R$.  Then $\mu$ is an $\AA$-critical measure, and it is uniquely determined by $ R$.
\end{proposition}

\section{Critical measures and extremal problems} \label{sec:criticalmeasuresandextremal}

Critical measures are connected in an essential way with a class of extremal problems that lies on a crossroad of the geometric function theory, approximation theory, potential theory and some other topics that will be mentioned below. We start with one of the oldest problems of that kind.

\subsection{Chebotarev's continuum} \label{sec:Chebotarevcont}

For a finite set $\AA=\{a_0, \dots, a_p \}$ in $\C$ we are interested in the continuum of minimal capacity containing $\AA $. More precisely, if we denote by $\mathcal F$ the family of all continua $F\subset \C$ with $\AA\subset F$, we seek $\Gamma\in \mathcal F$ such that
\begin{equation}\label{minCapacity}
    \cp(\Gamma) =\min_{F\in \mathcal F} \cp(F).
\end{equation} 
This problem was raised by Chebotarev (alternative spelling, Tchebotar\"{o}v) in a letter to P\'{o}lya (see \cite{Polya:1929}). Gr\"otzsch \cite{Grotzsch1930a} and Lavrentiev (or Lavrentieff) \cite{Lavrentieff1930}, \cite{Lavrentieff1934} proved that there exists a unique $\Gamma^*=\Gamma^*(\AA)$ satisfying \eqref{minCapacity}, and that $\Gamma^*$ is a union of critical trajectories of a rational quadratic differential $-R(z)(dz)^2$, where $R=V^*/A$, with $A$ given by \eqref{defAandB}, and $V^*(z)=\prod_{j=1}^{p-1} (z-v_j^*)$. We call this $\Gamma ^*$ the \emph{Chebotarev's compact} or \emph{Chebotarev's continuum}.

The most recent account on the background of this problem and some of its applications can be found in \cite{Ortega-Cerda2008}. Here we describe briefly (and inductively) some basic facts about the geometry of $\Gamma^*$ that will be useful in the sequel.

The case $p=1$ (that is, when $\AA$ contains 2 points) is trivial: in this situation $\Gamma^*$ is the segment $[a_0, a_1]$ joining both points.  

For $p=2$ (the first non-trivial case), $V^*(z)=z-v^*$. If $a_0$, $a_1$ and $a_2$ are collinear, say $a_2\in [a_0, a_1]$, then $v^*=a_2$, and we are in the situation of $p=1$ considered above. Otherwise $\Gamma^*(\AA)$ is made of three analytic arcs, each emanating from the respective pole $a_j$ ($j=0, 1, 2$) of $R$, and all merging at $v^*$ (see Figure \ref{fig:Chebotarev}).  Point $v^*=v^*(\AA)$, as function of $\AA$, is uniquely defined, but its analytic representation is not known (not to speak of the zeros of $V^*$ for $p\geq 3$), see \cite{Kuzmina1980}.  In the totally symmetric case when $a_j$ lie at the vertices of the equilateral triangle, point $v^*$ is just its center, and the three curves coincide with the bisectors joining it with the vertices. For further geometric properties of $\Gamma^*$ and $v^*$ see \cite[Ch.~1]{Kuzmina1980}. 

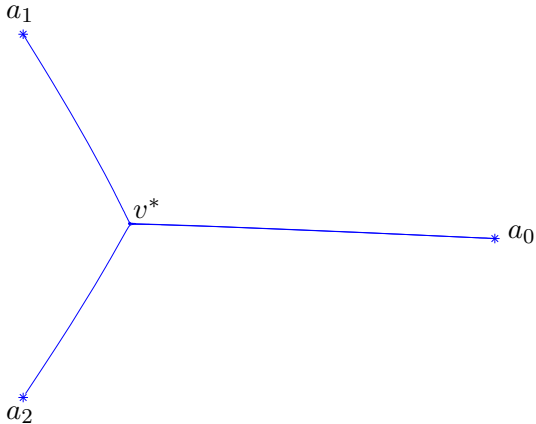
\begin{figure}[htb]
\centering \begin{overpic}[scale=0.6]%
{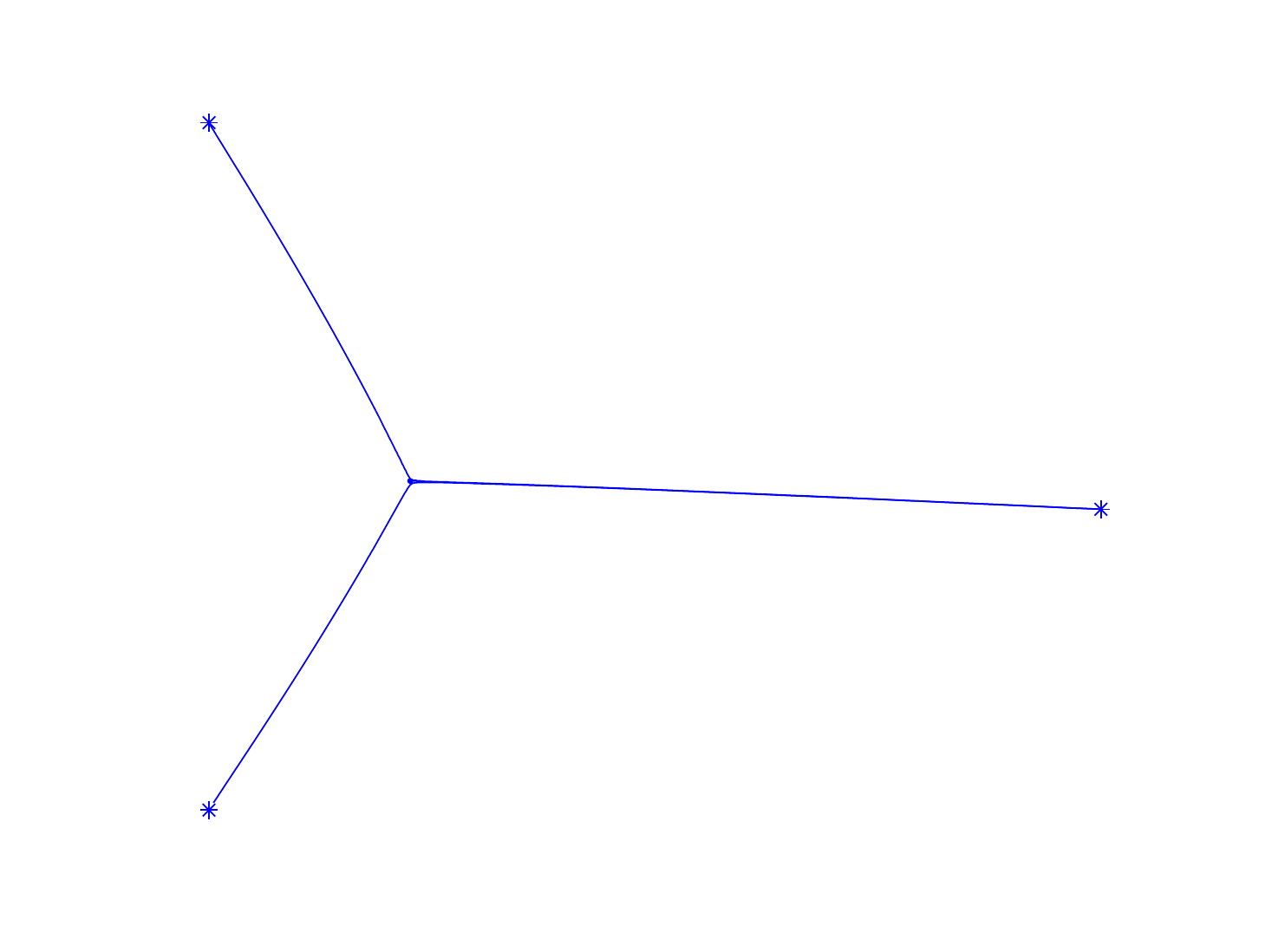}%
       \put(33,38){$v^* $}
   \put(89,35){$a_0 $}
   \put(14,8){$a_2 $}
   \put(14,68){$a_1 $}
\end{overpic}
\caption{Chebotarev's compact for three points.}
\label{fig:Chebotarev}
\end{figure}

For $p=3$ (four poles), in a generic situation the zeros of $V^*$ are simple, and the Chebotarev's compact consists of 5 arcs; if these zeros coalesce forming a double zero, then $\Gamma^*(\AA)$ is made of 4 arcs (see Figure~\ref{fig:cheb4points}). When one or both zeros $v_j^*$ coincide with a pole from $\AA$ we are left in one of the cases previously considered.

The number of ``degenerate'' situations growth fast with $p$, so in the sequel we restrict our attention to a generic one, when the points from $\AA$ are in a general position. This notion of genericity (as opposed to some more special or coincidental cases that are possible) means in our case that all zeros of $V^*$ are simple and disjoint with $\AA$.

\begin{figure}[htb]
\centering \begin{tabular}{ll} \hspace{-1.8cm}\mbox{\begin{overpic}[scale=0.65]%
{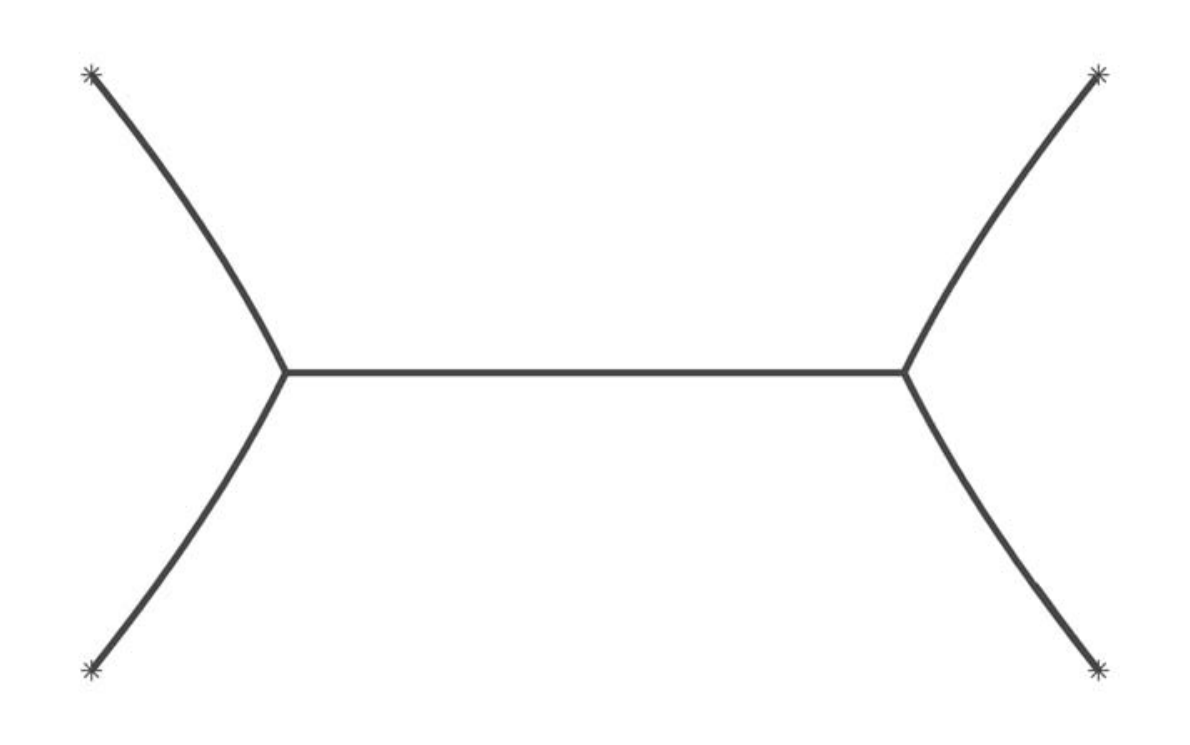}%
       \put(24,34){$v^*_1 $}
   \put(70,34){$v^*_2 $}
\put(2,55){$a_1 $}
\put(2,4){$a_0 $}
\put(94,55){$a_2 $}
\put(94,4){$a_3 $}
\end{overpic}} &
\mbox{\begin{overpic}[scale=0.45]%
{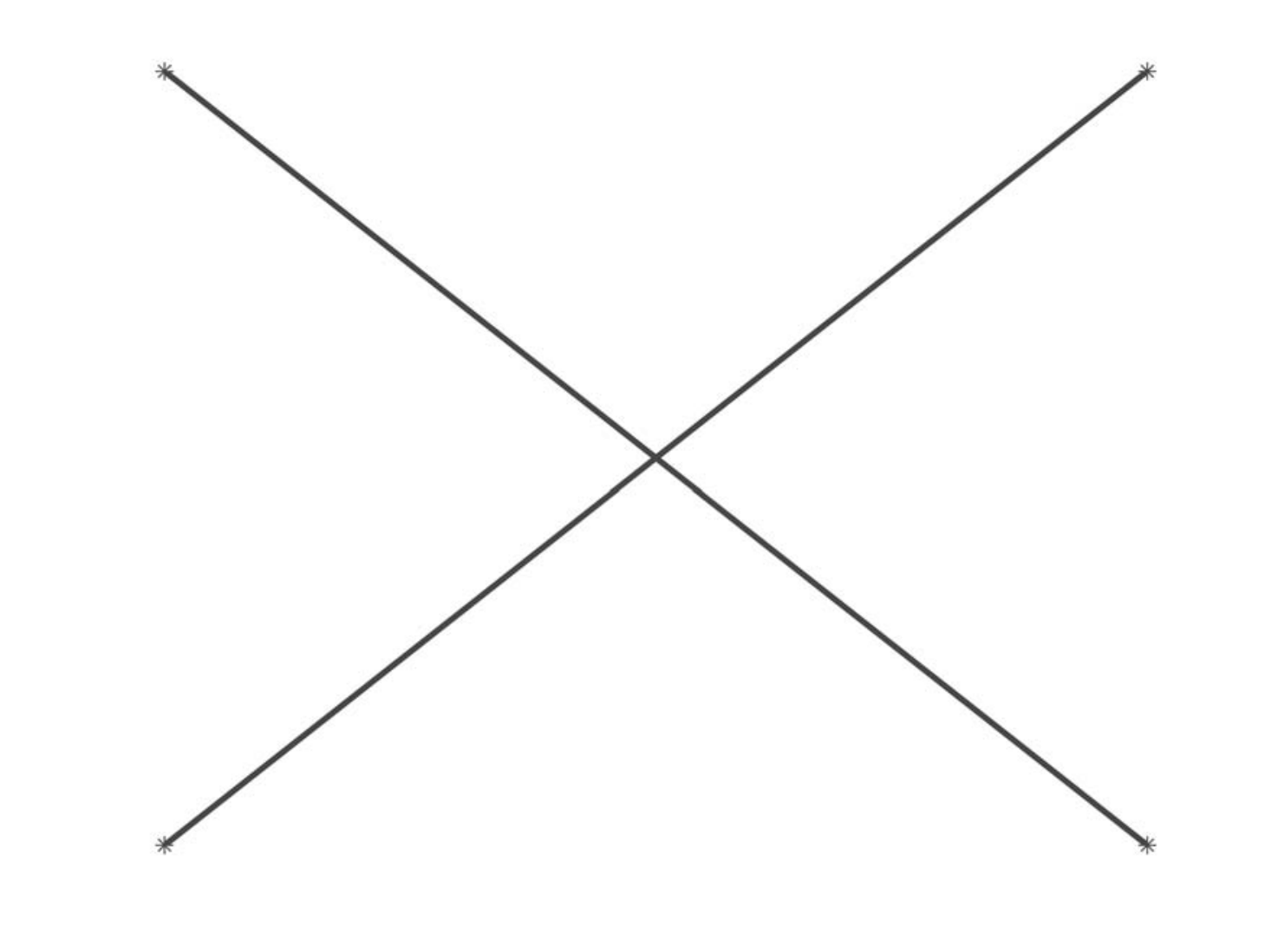}%
     \put(55,38){$v^*_1=v^*_2 $}
\put(6,68){$a_1 $}
\put(6,6){$a_0 $}
\put(93,68){$a_2 $}
\put(93,6){$a_3 $}
\end{overpic}}
\end{tabular}
\caption{Chebotarev's compact for 4 points, when the zeros of $V^*$ are simple (left) or double.}
\label{fig:cheb4points}
\end{figure}

For $p=5$ (6 points) we can find essentially two different configurations: the linearly ordered set (Figure~\ref{fig:cheb6points}, left) and the branched tree (Figure~\ref{fig:cheb6points}, right). In general, a non-degenerate $\Gamma(\AA)$ consists of a set of linear branches, like in Figure~\ref{fig:cheb4points}, left, and Figure~\ref{fig:cheb6points}, left, and it may have a branch point at a zero of $V^*$ where three branches merge, like in Figure~\ref{fig:cheb6points}, right.

Thus, $\Gamma^*(\AA)$ is an analytic tree. It will be crucial for our study of families of positive $\AA$-critical measures below. In fact, $\Gamma^*(\AA)$ plays a role of the ``origin'' of a coordinate system on the parameter plane of the above mentioned families of measures.

\begin{figure}[htb]
\centering \begin{tabular}{ll} \hspace{0cm}\mbox{\begin{overpic}[scale=0.9]%
{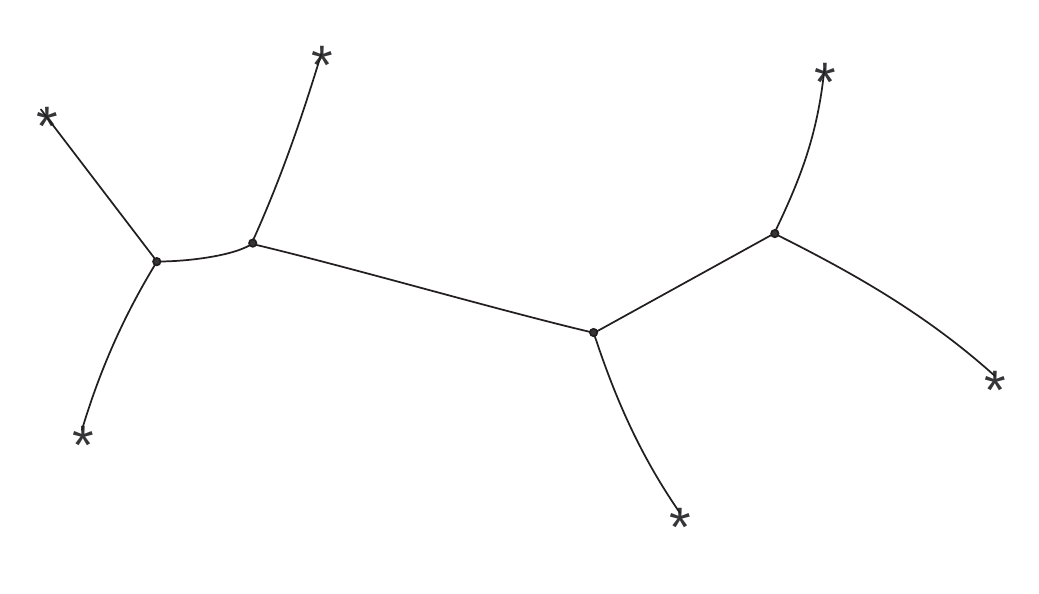}%
       \put(15,28){\small $v^*_1 $}
       \put(23,30){\small $v^*_2 $}
       \put(54,21){\small $v^*_3 $}
   \put(70,36){\small $v^*_4 $}
\end{overpic}} &
\hspace{0.5cm}
\mbox{\begin{overpic}[scale=0.9]%
{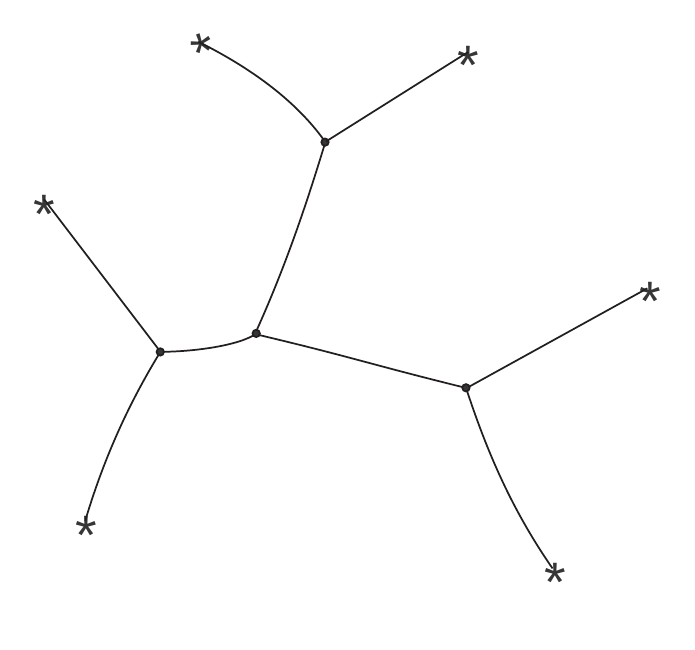}%
 \put(22,37){\small $v^*_1 $}
 \put(34,39){\small $v^*_2 $}
  \put(46,67){\small $v^*_3 $}
   \put(61,33){\small $v^*_4 $}
\end{overpic}}
\end{tabular}
\caption{Chebotarev's compacts for 6 points.}
\label{fig:cheb6points}
\end{figure}

Observe finally that by Lemma \ref{lemma:Sconditions}, Proposition \ref{uniquenessmeasure} and the result of Lavrentiev, the Robin (or equilibrium) measure $\lambda _\Gamma $ of $\Gamma=\Gamma(\AA)$ is a (unit) positive $\AA$-critical measure with a connected support, and that $\lambda_\Gamma$ is the only (unit) positive $\AA$-critical measure with this property.

\subsection{Cuts of minimal capacity  and convergence of Pad\'{e} approximants}

For $\AA=\{a_0, \dots, a_p \}\subset \C$ we denote by $\mathcal U(\AA)$ the class of analytic germs $f(z)=\sum_{n\geq 0} f_n z^{-n}$ at infinity that admit an analytic continuation to $\overline{\C}\setminus \AA$. For a fixed germ  $f\in \mathcal U(\AA)$ consider the family $\mathfrak F(f)$ of cuts $F\subset \C$ which make $f$ single-valued in their complement:
$$
\mathfrak F(f) \isdef \{ F\subset \C:\, f \text{ holomorphic in } \overline{\C}\setminus F\}.
$$
We are interested in the set $\Gamma=\Gamma(f)\subset \mathfrak F(f) $ such that
\begin{equation}\label{minCapacityF}
    \cp(\Gamma) =\min_{F\in \mathfrak F(f)} \cp(F).
\end{equation}
Observe that this problem is a generalization of that considered in the previous section.

Approximately 40 years ago one of the hottest topics in approximation theory of analytic functions was the problem of convergence of diagonal Pad\'{e} approximants $[n/n]_f=P/Q$ for functions $f\in \mathcal U(\AA)$. In this connection J.~Nuttall made a basic conjecture that the sequence $[n/n]_f$ converges (in capacity) to $f\in \mathcal U(\AA)$ in $\overline{\C}\setminus \Gamma(f)$; he also proved his conjecture in some special cases (see review \cite{MR769985}). The problem was completely solved in 1986 by H.~Stahl, who proved Nuttall's conjecture in a striking generality: for closed sets $\AA$ with $\cp(\AA)=0$. In particular, he established that for any $f\in \mathcal U(\AA)$ there exists an essentially unique set $\Gamma(f)$ satisfying \eqref{minCapacityF}; he also characterized it in terms of the equilibrium conditions and the $S$-property (see Section \ref{sec:ChebotarevStahl}). We refer the interested reader to \cite{Stahl2008}, \cite{MR88d:30004a}, \cite{Stahl:86}, \cite{MR90i:30063}. 

We note that our concept of an $\AA$-critical measure can also be extended to compact sets $\AA$ of capacity zero, but in this paper we deal with finite sets $\AA$ only, assumption that we keep in the sequel.

For a finite set $\AA$ there obviously exists only a finite number of different possible solutions (sets of minimal capacity)  of \eqref{minCapacityF}, that we denote by $\Gamma_0, \dots, \Gamma_N$, in such a way that $\Gamma_0$ is the Chebotarev's continuum for $\AA$. In other words, for any $f\in \mathcal U(\AA)$, $\Gamma(f)\in \{ \Gamma_0, \dots, \Gamma_N\}$.

For any $k=0, \dots, N$, the equilibrium measure $\lambda=\lambda_k$ of $\Gamma=\Gamma_k$ is $\AA$-critical and satisfies
\begin{align} \label{eq1}
    & U^{\lambda } (z)    = \rho _k=\const, \quad z\in \Gamma =\supp(\lambda  ),  \\
  \label{eq2}
   & \frac{\partial   }{\partial n_+}\,   U^{\lambda }  (z) = \frac{\partial   }{\partial n_-}\,  U^{\lambda } (z), \quad z \text{ at regular points of } \Gamma,
\end{align}
where $n_\pm$ are the normal vectors to $\supp(\mu )$ at $z$ pointing in the opposite directions. Moreover, it follows from \cite{MR88d:30004a} that these conditions define $\lambda$ and $\Gamma$ uniquely in the given homotopic class (we spare details here).

A comparison with Lemma \ref{lemma:Sconditions} (for $\varphi\equiv 0$) shows that among all $\AA$-critical measures the Robin (or equilibrium) measures $\lambda_k$ are distinguished by the equality of the equilibrium constants: if $\Gamma_{k,1}, \dots, \Gamma _{k,M}$ are the connected components of $\Gamma =\Gamma _k$, then
$$
U\big|_{\Gamma_{k,1}}=\dots = U\big|_{\Gamma_{k,M}}.
$$
Chebotarev's continuum $\Gamma_0$ corresponds to $M=1$; the other $\Gamma_k$'s may be combinatorially characterized using the structure of $\Gamma_0$. We present an example explaining this statement.
\begin{example}
\label{example:mincuts}
Assume $p=3$, $a_0=-x-iy$ (with $x>y>0$), $a_1=\overline{a_0}$, $a_2=-a_0$, $a_3=-a_1$, like in Figure~\ref{fig:cheb4points}, left. Then there are exactly two different compact sets of minimal capacity. One is the Chebotarev's set $\Gamma_0$ (as the one depicted in Fig.~\ref{fig:cheb4points}, left), corresponding for instance to
$$
f(z)=\left( \frac{(z-a_0)(z-a_1)}{(z-a_2)(z-a_3)}\right)^{1/4} \in \mathcal U(\AA),
$$
and the other, $\Gamma_1$, with two components (as in Fig.~\ref{fig:Cheb4pts2parts}), corresponding to function
$$
f(z)=\left( \frac{(z-a_0)(z-a_1)}{(z-a_2)(z-a_3)}\right)^{1/2} \in \mathcal U(\AA).
$$
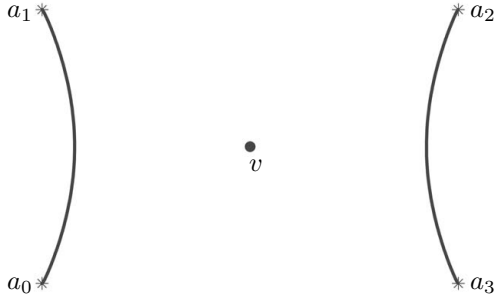
\begin{figure}[htb]
\centering \begin{overpic}[scale=0.6]%
{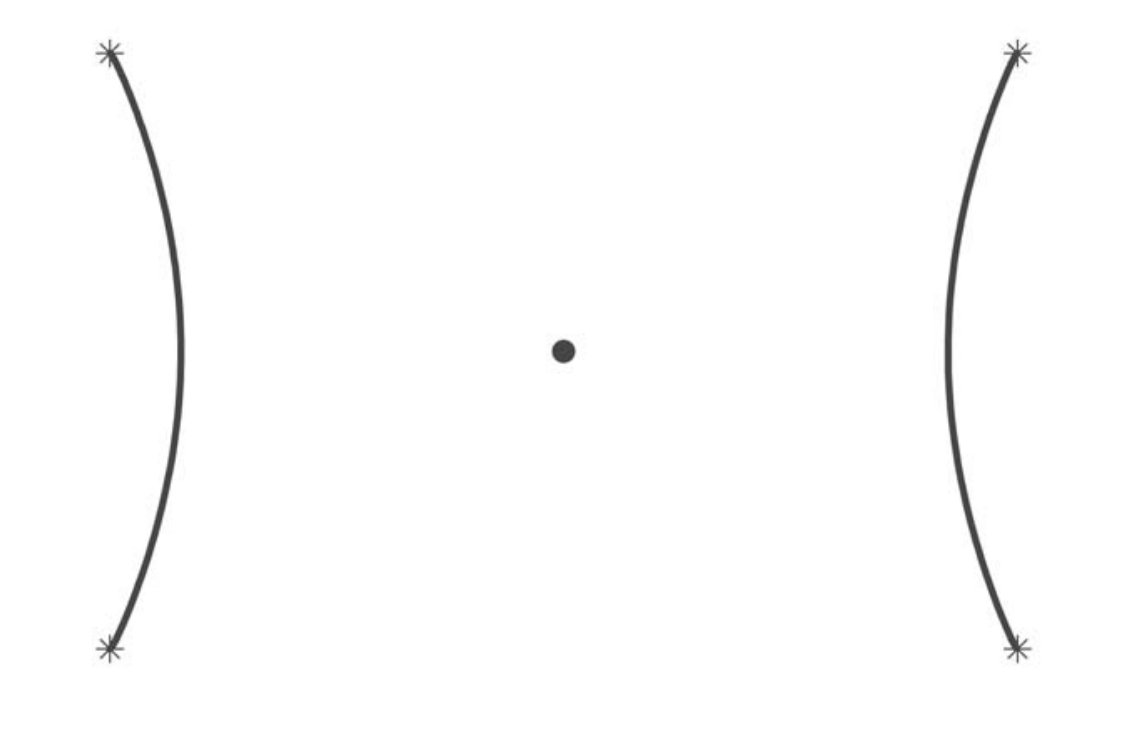}%
     \put(3,6){\small $a_0 $}
\put(3,59){\small $a_1 $}
\put(93,59){\small $a_2 $}
\put(93,6){\small $a_3 $}
\put(50,29){\small $v $}
\end{overpic}
\caption{A minimal capacity set for 4 points.}
\label{fig:Cheb4pts2parts}
\end{figure}
Observe that $\Gamma _1$ is topologically equivalent to $\Gamma _0$ with the arc connecting $v_1^*$ and $v_2^*$ removed. The quadratic differential associated with $\lambda =\lambda_1$ has a double zero $v$ at the center of the rectangle, so that \eqref{potentialMeasure} takes the form
$$
w-U^\lambda(z)=\Re \int_{a_0}^z \frac{t-v}{\sqrt{A(t)}}\, dt,
$$
which is the Green function of the Riemann surface of the function $y^2=A(x)$ with logarithmic poles at $\infty_1$ and $\infty_2$. This is also an elliptic integral of the third kind.
\end{example}

For a general set $\AA$ with an arbitrary number of points consider the Chebotarev's continuum $\Gamma_0$ and select any arc $\beta $ of $\Gamma_{0}$ connecting two zeros of the corresponding quadratic differential (a ``zero--zero connection''). Then there exists a compact $\Gamma_1$ of minimal capacity with two connected components, which is topologically equivalent to $\Gamma_0\setminus \beta $, whose quadratic differential exhibits a double zero instead of the pair of zeros we have selected. 
We can repeat this operation with any other remaining zero--zero connections, until all these connections are gone. If $p=2m-1$, the maximal number of connected components of a set of minimal capacity is $m$. The unique $\Gamma$ with $m$ components is again the zero level of he Green function for the Riemann surface of $y^2=A(x)$.

The compact sets $\Gamma_k$ and their Robin measures $\lambda_k$ associated with $\AA$ play a central role in any investigation of the strong asymptotics for complex orthogonal polynomials -- denominators of the Pad\'{e} approximant of functions from $\mathcal U(\AA)$, see \cite{MR891770}.

\subsection{Further connections}

The central part in the Stahl's solution of the convergence problem for Pad\'{e} approximants was a new method of investigation (based directly on the $S$-property) of the $n$-th root asymptotics of polynomials satisfying complex non-hermitian orthogonality conditions, like those verified by the denominators of the Pad\'{e} approximants to functions from $\mathcal U(\AA)$. This method was further developed in \cite{Gonchar:87} in relation with the best rational approximations; the $n$-th root asymptotics was obtained for complex orthogonal polynomials with respect to varying weights (i.e.~depending on the degree of the polynomial). The existence of a varying weight motivates the appearance of an external field in the associated equilibrium problem. Accordingly, the $S$-property in the related existence problem should be modified to include the external field too (we omit here the non-essential details).

Let $\varphi$ be a harmonic function in a domain $\Omega\subset \C$. For a curve $\Gamma\subset \Omega$ let $\lambda=\lambda_{\Gamma,\varphi}$ be the unit equilibrium measure on $\Gamma$ in the external field $\varphi$, so that conditions \eqref{equilibriumSection6} hold. Recall that additionally $\Gamma$ has the $S$-property if at the regular points of $\supp(\lambda )$,
\begin{equation}\label{S-propertySect6}
 \frac{\partial   }{\partial n_+}\,  \left( U^\lambda+\varphi \right)(z)= \frac{\partial   }{\partial n_-}\, \left( U^\lambda+\varphi \right)(z),
\end{equation}
where $n_\pm$ are the normal vectors to $\supp(\mu )$ at $z$ pointing in the opposite directions (it is assumed that the set of irregular points of $\supp(\lambda )$ has capacity zero). We call the support of such a $\lambda$ an \emph{$S$-curve}. The cornerstone to any application is the problem of \emph{existence} of such a curve, that we discuss here briefly.

Given a harmonic function $\varphi$ in a domain $\Omega$ and a homotopic class $\mathfrak F$ of curves $F\subset \Omega$, we should find a curve $\Gamma\in \mathfrak F$ with the $S$-property.

A direct and constructive approach to the solution of this problem is based on the observation that if such a curve $\Gamma\in \mathfrak F$ exists, its equilibrium measure $\lambda =\lambda _{F,\varphi}$ is $(\AA, \varphi)$-critical for some set $\AA$ of fixed points, which depends on the definition of $\mathfrak F$ and on the singularities of $\varphi$. Then (see Theorem \ref{main_thm} and Lemma \ref{thm:rakhmanov}) we conclude that $(C^\lambda +\Phi')^2=R$, where $R$ is some function, meromorphic in $\Omega$ with (usually, known) poles at $\AA$ and (usually, unknown) zeros. These zeros are the main parameters of the problem; they must be found using a system of equations, typically in terms of periods of $\int^z \sqrt{R}\, dt$), that reflect all the given information (including the geometry of the class $\mathfrak F$). 

For $\varphi\equiv 0$ this is basically a classical method, which goes back to Abel and Riemann (abelian integrals, see the discussion below). The existence of an external field does not change the nature of the problem, but it posses additional technical difficulties. If $\varphi \not \equiv 0$, we generally have $\supp(\lambda ) \varsubsetneq \Gamma$ (compare with \eqref{eq1} and \eqref{eq2}), and finding the support of the equilibrium measure might turn out to be a formidable task even for a fixed $\Gamma$. We refer to \cite{Gonchar:87} for further details. Se also \cite{MR1203311}, as well as \cite{Bertola2007}, where the $S$-problem for $\Omega=\C$, and $\varphi=\Re (P)$, where $P$ is a polynomial, was considered.

Another way to prove the existence of the $S$-curve independently of its construction is based on the electrostatic interpretation of the critical measures, which yields the following extremal problem. Consider the equilibrium energy $E_\varphi(\cdot )$ (see \eqref{defEnergyContinuous}--\eqref{defWeightedEnergyCont}) of a curve $F\in \mathfrak F$ as a functional on $\mathfrak F$:
$$
E_\varphi[F] = E_\varphi(\lambda_{F, \varphi} ):\, \mathfrak F \mapsto \R.
$$
Under rather general assumptions it is possible to prove that if a curve $\Gamma \in \mathfrak F$, satisfying
\begin{equation}\label{minmax}
   E_\varphi[\Gamma]=\max_{F\in \mathfrak F} E_\varphi[F]
\end{equation}
exists, then it has the $S$-property. For $\varphi\equiv 0$ we have $E_\varphi[F]=-\log (\cp(F))$, and \eqref{minmax} is equivalent to the minimal capacity problem considered above. For $\varphi \not \equiv 0$, this is the weighted capacity \eqref{def_capacity} minimization (see e.g.~\cite{Saff:97}), and the method was outlined in \cite{Gonchar:87} in connection with the best rational approximation of $\exp(-x)$ on $[0, +\infty)$. The discrete analogue of problem \eqref{minmax} and its connection with Jacobi polynomials was discussed in \cite{Marcellan07}.

In a rather surprising twist, a completely different problem was reduced to the existence of an $S$-curve in \cite{Kamvissis2003}, where the semiclassical solution of the focusing nonlinear Schr\"{o}dinger equation was constructed using methods of the inverse scattering theory. The problem itself, as well as the methods of its solution, did not have a priori any visible connection with those examined in \cite{Gonchar:87}.
A partial explanation of the mystery is suggested by the connection of the orthogonal polynomials with the inverse scattering via the matrix Riemann-Hilbert (RH) problem (see e.g.~\cite{MR2000g:47048} and \cite{Fokas92}). Apparently, the corresponding RH problems considered in \cite{Gonchar:87} and \cite{Kamvissis2003} are similar. In connection with the theory of nonlinear partial differential equations we should mention also the seminal work of Lax and Levermore \cite{Lax:1983fk}, \cite{Lax:1983vr}, \cite{Lax:1983bw}, where the connection between a singular limit of the KdV equation and the energy problem with upper constraint was established.

The recently developed tools of asymptotic analysis of the RH problems of a certain class, such as the non-linear steepest descent method of Deift and Zhou (see \cite{MR2001f:42037}, \cite{MR2001g:42050},  \cite{MR2000g:47048}, as well as \cite{Bleher1999}) combined with the $\overline{\partial}$-problem (\cite{MR2219316}, \cite{McLaughlin/Miller:2008}), have become powerful weapons in the study of the strong asymptotics of polynomials of complex orthogonality. We can find multiple examples in a series of works of Aptekarev, Baik,  Deift, Kuijlaars, McLaughlin, and Miller, to mention a few (see e.g.~\cite{Aptekarev:02}, \cite{MR2000g:47048}, \cite{MR2124460}, \cite{Kuijlaars/Mclaughlin:04}, \cite{MR2087231}, \cite{Mclaughlin/Vartanian/Zhou06c}; this is necessarily a very partial list).   

 A very active topic of research that benefited greatly from these ideas is the random matrix theory (RMT), see \cite{MR2000g:47048}, as well as a closely connected field of random particle ensembles with non-intersecting paths and other determinantal point processes \cite{Borodin98}, \cite{Soshnikov00}. The detailed behavior (as the degree goes to infinity) of the related orthogonal polynomials allows to establish fine asymptotic results for the random matrices or non-intersecting paths ensembles, see e.g.~\cite{Aptekarev/Bleher/Kuijlaars}, \cite{Bertola:2002kx}, \cite{Bertola:2009uq}, \cite{Bleher2010},  \cite{MR2038771}, \cite{Bleher/Kuijlaars1},  \cite{Bleher/Kuijlaars3}, \cite{Duits2010}, \cite{Duits:2009zr}, \cite{Ercolani:2001vn},
 \cite{MR2470930} for again a very partial list.

One of the main ingredients of the solution of such kind of asymptotic problems (independently of the approach we follow) is the analysis of the $S$-property related to the concrete situation. The total potential of the corresponding critical (or equilibrium) measure is called the $g$-function (see \cite{MR2000g:47048}). This function typically accounts for the leading term of the asymptotics, and the support of the measure is the set where the essential oscillatory behavior takes place. Hence, the construction of the $g$-function or of the $S$-curve, or even establishing the existence of the latter without finding the parameters explicitly, is an important problem. In the presence of a significant external field it has been solved so far for some particular situations. The use of the critical measures in this context may present a new approach to the problem at large.

\section{Weak limit of zeros of Heine-Stieltjes polynomials} \label{sec:weakLimits}

We return to our original motivation, armed now with the tools developed so far, in order to analyze the possible weak limits of the polynomial solutions of \eqref{DifEq}. We formulate first a statement slightly more general than necessary for our problem.

We are given $\AA=\{a_0, a_1, \dots, a_p\}\subset \C$, and a sequence of external fields of the form
\begin{equation}\label{extFieldMore}
\varphi_n=\Re \Phi_n, \quad   \Phi_n(z)=-\sum_{k=0}^p \frac{\rho_k(n)}{2} \, \log (z-a_k )\,,
\end{equation}
where $\rho_k(n)\in \C$.

\begin{theorem}
\label{the:weakAsymptoticsWithField}
Let $\mu_n\in \mathfrak M_n$, $n\in \N$, be a discrete $(\AA, \varphi_n)$-critical measure corresponding to an external field \eqref{extFieldMore}. If for a subsequence $\mathcal N\subset \N$, limits
$$
\lim_{n\in \mathcal N } \frac{\rho _k(n)}{n} =\rho_k, \quad k=0, 1, \dots, p,
$$
exist, then any weak-* limit point $\mu$ of the normalized measures $\{\mu_n/n\}$, $n\in \mathcal N$, is a continuous $(\AA,\varphi)$-critical measure with respect to the external field $\varphi$ given by \eqref{Phiforcharges}.

In particular, if
\begin{equation}\label{conditionUnboundedCont}
    \Re \sum_{k=0}^p  \rho_k  > -\frac{1}{2},
\end{equation}
then $\mu$ is a unit continuous $(\AA,\varphi)$-critical measure.
\end{theorem}
\begin{proof}
Without loss of generality, we assume that $\nu_n\isdef \mu_n/n  \weakto \mu$, $n\in \mathcal N$, where $\weakto$ means convergence in the weak-* sense. By Proposition \ref{prop:boundedness}, if \eqref{conditionUnboundedCont} holds then $\supp(\mu_n)$, $n\in \mathcal N$, are uniformly bounded, so that the set of normalized measures $\nu_n$ is weakly compact, and $\mu$ is a probability measure on $\C$.

By Remark \ref{remarkdiscrete}, it is sufficient to show that for smooth functions $h$, condition
\begin{equation}\label{variationdiscreteenergy}
    \frac{d}{dt}\, \EE_{\varphi_n}(\nu_n^t)\big|_{t=0} =0  , \qquad n\in \N,
\end{equation}
implies
$$
\frac{d}{dt}\, E_\varphi(\mu^t)\big|_{t=0} = 0.
$$
Let
$$
\mu_n =\sum_{k=1}^n \delta_{\zeta_k^{(n)}};
$$
reasoning as in the proof of Lemma \ref{lemma:condCrit} we conclude that \eqref{variationdiscreteenergy} is equivalent to the condition
\begin{equation}\label{conditionLocalVarDiscrete}
    \frac{1}{n^2}\, \sum_{i\neq j}  \frac{h(\zeta_i^{(n)})-h(\zeta_j^{(n)}) }{\zeta_i^{(n)} -\zeta_j^{(n)} }  - \frac{2}{n^2}\, \sum_{k=1}^n    \Phi'_n(\zeta_k^{(n)})\, h(\zeta_k^{(n)})=0 , \qquad n\in \N.
\end{equation}
Observe that
$$
\frac{1}{n^2}\, \sum_{i\neq j}  \frac{h(\zeta_i^{(n)})-h(\zeta_j^{(n)}) }{\zeta_i^{(n)} -\zeta_j^{(n)} }   = \lim_{\epsilon \to 0 }\iint_{|x-y|>\epsilon }  \frac{h(x)-h(y) }{x -y }   \, d\nu_n (x ) d\nu_n (y ).
$$
Since $(h(x)-h(y))/(x-y)$ is continuous, standard arguments show that
$$
\lim_{n\in \mathcal N} \frac{1}{n^2}\, \sum_{i\neq j}  \frac{h(\zeta_i^{(n)})-h(\zeta_j^{(n)}) }{\zeta_i^{(n)} -\zeta_j^{(n)} }   = \iint  \frac{h(x)-h(y) }{x -y }   \, d\mu (x ) d\mu (y ).
$$

On the other hand, $\varphi_n/n \to \varphi$, $n\in \mathcal N$, locally uniformly in $\C\setminus \AA$, with $\varphi$ given in \eqref{Phiforcharges}. Since $h$ vanishes on $\AA$, $\Phi'_n  h$ are continuous, and
$$
\lim_{n\in \mathcal N} \frac{1}{n^2}\, \sum_{k=1}^n    \Phi'(\zeta_k^{(n)})\, h(\zeta_k^{(n)})  = \lim_{n\in \mathcal N} \frac{1}{n}\, \int    \Phi'_n(x)\, h(x) \, d\nu_n(x) = \int    \Phi'(x)\, h(x) \, d\mu(x).
$$
In consequence,
$$
\lim_{n\in \mathcal N} \frac{1}{n^2}\, \left(  \sum_{i\neq j}  \frac{h(\zeta_i^{(n)})-h(\zeta_j^{(n)}) }{\zeta_i^{(n)} -\zeta_j^{(n)} }  - 2 \, \sum_{k=1}^n    \Phi'_n(\zeta_k^{(n)})\, h(\zeta_k^{(n)})\right)=f_\varphi(\mu; h),
$$
with $f_\varphi$ defined in \eqref{defFvar}. Using \eqref{conditionLocalVarDiscrete}, we conclude that  $f_\varphi(\mu; h) =0$, and it remains to apply Lemma \ref{lemma:condCrit}.
\end{proof}

We consider next a sequence of pairs $(Q_n, V_n)$ of Heine-Stieltjes polynomials $Q_n$ of degree $n$ and their corresponding Van Vleck polynomials $V_n$. One of the central results of this paper is a description of all the possible limits of the normalized zero-counting measures of the Heine-Stieltjes polynomials $Q_n$. Since the residues in \eqref{BoverA} are independent of $n$, by applying Theorem \ref{the:weakAsymptoticsWithField} with $\varphi\equiv 0$ we get:
\begin{corollary}
\label{the:weakAsymptoticsFirst}
Any weak-* limit point of the normalized zero counting measures $\nu(Q_n)/n$ of the Heine-Stieltjes polynomials is a unit continuous $\AA$-critical measure.
\end{corollary}
\begin{rmk}
Taking into account Proposition \ref{characterizationHS}, we could restate the last result in terms of the zero-counting measures of Heine-Stieltjes polynomials corresponding to a generalized Lam\'{e} equation with coefficients depending on $n$. An analogue of Theorem \ref{the:weakAsymptoticsWithField}  has been used in \cite{MR2003j:33031} for the study of the weak-* limits of the normalized zero counting measures of the Heine-Stieltjes polynomials with varying positive residues and $\AA\subset \R$.
\end{rmk}

By \cite{Shapiro2008a}, the zeros of Van Vleck polynomials accumulate on the convex hull of $\AA$, so that the set of all Van Vleck polynomials is bounded (say, in the component-wise metrics). Hence, in our consideration of a sequence of pairs $(Q_n, V_n)$ of Heine-Stieltjes polynomials $Q_n$ of degree $n$ and their corresponding Van Vleck polynomials $V_n$ we suppose without loss of generality that there exists a monic polynomial $V$ of degree $p-1$ such that
\begin{equation}\label{convergenceC}
    \lim_{n\to \infty} V_n = V.
\end{equation}
\begin{theorem}
\label{the:weakAsymptotics}
Under assumption \eqref{convergenceC}, the normalized zero counting measure $\nu(Q_n)/n$ converges (in a weak-* sense) to an $\AA$-critical measure $\mu  \in \MM_1$; furthermore, the quadratic differential
$$
\varpi = -\frac{V}{A}(z) \, dz^2
$$
is closed, the support $\Gamma =\supp(\mu )$ consists of critical trajectories of $\varpi$, $\C\setminus \Gamma$ is connected, and we can fix the single valued branch of $\sqrt{V/A}$ there by $\lim_{z\to\infty} z \sqrt{V(z)/  A(z) }=1$. With this convention,
\begin{equation}\label{asymptNthRoot}
    \lim_n \left| Q_n(z)\right|^{1/n} =\exp\left( \Re \int^z \sqrt{\frac{V}{A}}(t) \, dt \right)
\end{equation}
locally uniformly in $\C\setminus \Gamma$, where a proper normalization of the integral in the right hand side is chosen, so that
$$
\lim_{z\to\infty} \left( \Re \int^z \sqrt{\frac{V}{A}}(t) \, dt- \log |z|\right)=0.
$$
\end{theorem}
In other words, weak-* limits of the normalized zero counting measures of $Q_n$'s are unit positive $\AA$-critical measures. 
The inverse inclusion (that any unit positive $\AA$-critical measure is a weak-* limit of the normalized zero counting measures of Heine-Stieltjes polynomials) is also valid, but it cannot be established using methods of this paper. We plan to present the proof in a subsequent publication related to the strong asymptotics of Heine-Stieltjes polynomials.
However, in the rest of the paper we identify both sets of measures.
\begin{proof}
Let $\mu$ be any weak-* accumulation point of $\{\nu(Q_n)/n\}$. By Lemma \ref{thm:rakhmanov}, the Cauchy transform of $\mu$ must satisfy an equation of the form \eqref{charactMeasure}. Rewriting \eqref{DifEq} in the Riccati form and taking limits as $n\to \infty$ (with account of \eqref{convergenceC}) we conclude that
$$
\left(C^\mu(z) \right)^2=\frac{V}{A}(z)\,,  \qquad z\notin \supp (\mu)\,,
$$
and $\Gamma=\supp (\mu)$ is a union of critical trajectories of $\varpi$. The rest of assertions about $\varpi$ follows from Proposition~\ref{uniquenessmeasure}.

Finally, \eqref{asymptNthRoot} is the straightforward consequence of \eqref{charactMeasure} and the fact that for any monic polynomial $P$, $\log 1/|P(z)|$ is the potential of its zero-counting measure.
\end{proof}

Theorem \ref{the:weakAsymptotics} provides an analytic description of the weak-* limits of the zero counting measures associated to the Heine-Stieltjes polynomials. However, this description is in a certain sense implicit, since it depends on the limit $V$ of the Van Vleck polynomials $V_n$, that constitute therefore the main parameters of the problem. We must complement this description with the study of the set of all possible limits $V$. As it follows from Theorem \ref{the:weakAsymptotics}, this can be done in two steps:
\begin{enumerate}
\item[\emph{i)}] describing the global structure of the trajectories of closed rational quadratic differentials with fixed denominators on the Riemann sphere, and the corresponding parameters (numerators). This problem has an independent interest;
\item[\emph{ii)}] extracting from this set the subset giving rise to positive unit $\AA$-critical measures.
\end{enumerate}
It was mentioned in Section \ref{sec:qd} that problem \emph{i)} is in general very difficult. In particular, we can find critical trajectories of any homotopic type. We have seen in Section \ref{sec:correspondence} that not all of them will correspond to the support of an $\AA$-critical measure, and elucidating this relation is the main step towards the complete description of the weak-* limits of the zero counting measures of Heine-Stieltjes polynomials. We start with a detailed discussion in the next Section of the simplest non-trivial case of three poles ($p=2$), corresponding to Heun's differential equation. For a general $p$ the geometry becomes so complex, that in this paper we just outline the main results (see Section \ref{sec:generalP}).

\section{Heun's differential equation ($p=2$)} \label{sec:p=2}

In this section we concentrate on the differential equation \eqref{DifEq} with $A(z)=(z-a_0)(z-a_1)(z-a_2)$ and $V(z)=z-v$. Our goal is to illustrate all previously established general results in the simplest non-trivial situation.

Observe that the Van Vleck polynomials constitute now a 1-parameter family, that makes the whole analysis much easier. So we introduce the quadratic differential
\begin{equation}\label{qdp2}
    \varpi_v= \frac{v-z}{A(z)}\, dz^2
\end{equation}
and two sets,
$$
\mathcal V \isdef \left\{ v\in \C:\, \varpi_v \text{ is closed} \right\},
$$
as well as the \emph{Van Vleck set}
$$
\mathcal V_+ \isdef \left\{ v\in \C:\, v \text{ is an accumulation point of the zeros of Van Vleck polynomials} \right\}.
$$
A direct consequence of Theorem \ref{the:weakAsymptotics} is that
$$
\mathcal V_+ \subset \mathcal V.
$$
This inclusion is proper. Obviously, our main purpose, motivated by the analysis of the Heine-Stieltjes and Van Vleck polynomials, is to study $\mathcal V_+$; along this path related questions will be dealt with, such as the structure of the closed quadratic differentials of the trivial homotopic type, and the set of positive critical measures. On the other hand, general quadratic differentials and signed critical measures have an independent interest, and some results will be presented below.

Namely, we will show that:
\begin{itemize}
\item the set of closed quadratic differentials \eqref{qdp2} is parametrized by a family of analytic arcs, dense on the plane and joining the Chebotarev's center of $\AA$ (see its definition below) with infinity;
\item each arc of this family represents an individual homotopic type of the quadratic differential, so that two values of the parameter $v$ in \eqref{qdp2} corresponding to different curves yield homotopic classes of critical trajectories non reducible to each other;
\item the trivial homotopic type of closed quadratic differentials corresponds only to the union of three subarcs joining the set $\AA$ with its Chebotarev's center. This star-shaped set is isomorphic to the set of \emph{positive} $\AA$-critical measures, and coincides with $\mathcal V_+$;
\item for each $v\in \mathcal V_+$, the zeros of the corresponding Heine-Stieltjes polynomials accumulate on the two critical trajectories of $\varpi_v$ in \eqref{qdp2}.
\end{itemize}

\subsection{Global structure of trajectories} \label{sec:globalStruc2}

Along this section we denote by $\Gamma_v$ the \emph{critical graph} of $\varpi_v$, that is, the set of critical trajectories of $\varpi_v$ together with their endpoints (critical points of $\varpi_v$).

There exists a unique $v^*=v^*(\AA)\in \C$ such that the critical graph  $\Gamma_{v^*}$ is a connected set; $\Gamma_{v^*}$ coincides with the Chebotarev's compact $\Gamma^*$ associated with $\AA$ (see Section \ref{sec:Chebotarevcont}). To simplify terminology, in the context of three poles we call $v^*$ the \emph{Chebotarev's center} for $\AA$.

Since the value $v^*$ is in many senses exceptional, along with the poles $a_i$, we will introduce the notation $\mathcal A^*\isdef \mathcal A \cup \{v^*\}$ for the ``exceptional set''.

We have mentioned in Section \ref{sec:qd} that the global structure of the trajectories of a quadratic differential can be extremely complicated. For the quadratic differential \eqref{qdp2} a certain order is imposed by the double pole at infinity with a negative residue.
\begin{proposition} \label{thm:existenceclosedtrajectoryp2}
Let $A(z)=(z-a_0)(z-a_1)(z-a_2)$ be a polynomial with simple roots in $\C$ and $v\in \C\setminus  \AA^*$. Then the quadratic differential \eqref{qdp2} has a closed critical trajectory $\beta$ containing $v$. Let $\Omega$ be the bounded domain delimited by $\beta $. Then $\Omega$ contains at least two points from $\AA$.
\end{proposition}
\begin{proof}
Due to the local structure of the trajectories, $\varpi_v$ has a closed trajectory freely homotopic to infinity (in other words, it is topologically identical to a circle  and contains all the finite critical points of $\varpi_v$ in the bounded component of its complement). According to Theorem 9.4 of \cite{Strebel:84}, this closed trajectory is embedded in a uniquely determined maximal ring domain $\mathcal R_\infty$, swept out by homotopic closed trajectories of $\varpi_v$.
We denote by $\beta^*$ the bounded connected component of the boundary of $\mathcal R_\infty$.

Obviously, $\beta^*$ contains at least one critical point of $\varpi_v$, and no finite critical points of $\varpi_v$ lie in $\mathcal R_\infty$ (that is, in the unbounded component of $\C\setminus \beta^*$).  We conclude that all finite critical points of $\varpi_v$ lie either on $\beta^*$ or in the bounded component of its complement. If this component is empty, it means that $\beta^*$ contains all the critical points, and it is the Chebotarev compact for $\AA$. Otherwise, the bounded component of $\C\setminus \beta^*$ contains an interior point, and from the local structure of the trajectories of $\varpi_v$ at simple poles we infer that $\beta^*$ cannot contain only poles of $\varpi_v$. Hence, $v\in \beta^*$, and at least two of the three trajectory arcs emanating from $v$ belong to $\beta^*$. It is easy to see that either they end up at respective poles (and then again $\beta ^*=\Gamma^*$), or they form a closed loop, that we call $\beta$. We call $\Omega$ the bounded domain delimited by $\beta $. There is only one trajectory emanating from $v$ remaining, that either is recurrent or ends at a pole from $\AA$. Hence, $\beta^*$ cannot contain more than one pole, so that $\Omega$ contains at least two poles. This concludes the proof.
\end{proof}
\begin{rmk}
An examination of the proof of Proposition~\ref{thm:existenceclosedtrajectoryp2} shows that the existence of the extremal trajectory $\beta^*$ containing a zero of $\varpi$ y such that all trajectories outside $\beta^*$ are closed and homotopic to infinity is a fact, valid for an arbitrary quadratic differential of the form \eqref{quadDiff}.
\end{rmk}

Let $v\in \C \setminus \AA^* $. By the theorem on the local structure of the trajectories of $\varpi_v$, there are three trajectories originating at $v$. From Proposition \ref{thm:existenceclosedtrajectoryp2} it follows that two of them form a single closed loop $\beta $ that splits $\C$ into two domains: the bounded component of the complement, $\Omega$, and the unbounded one, that we denote by $\mathcal D$. Let us denote by $\gamma $ the remaining trajectory emanating from $v$. Observe that $\beta$ and $\gamma $ have a single common point, $v$; according to the relative position of $\gamma $ with respect to $\beta $ we can establish the following basic classification for the critical graph $\Gamma_v$ of $\varpi_v$:
\begin{enumerate}
\item[1)] \emph{Exterior configuration:} $\gamma \setminus \{v\}$ belongs to the unbounded component $\mathcal D$ of the complement of $\beta $ (see Fig.~\ref{fig:weirdtrajectories}, left). In this case $v\in \mathcal V$ (the quadratic differential is closed).

 One of our main results, which we discuss later, is that $v\in \mathcal V_+$ if and only if $\Gamma_v$ has an exterior configuration.

\item[2)] \emph{Interior closed configuration:} $\gamma \setminus \{v\}$ belongs to the bounded component $\Omega$ of the complement of $\beta $, and $\gamma $ is finite (see Fig.~\ref{fig:weirdtrajectories}, right or Fig.~\ref{fig:recurrent}, right). In this case $\gamma $ is a critical trajectory joining $v$ with one of the poles $a_j$, and $v\in \mathcal V$ (the quadratic differential is closed).

\item[3)]  \emph{Interior recurrent configuration:} $\gamma \setminus \{v\}$ belongs to the bounded component $\Omega$ of the complement of $\beta $, and $\gamma $ is not finite (see Fig.~\ref{fig:recurrent}, left). In this case $\gamma $ (in fact, all trajectories in $\Omega$) is a recurrent trajectory, dense in $\Omega$, and $v\notin \mathcal V$.

\end{enumerate}

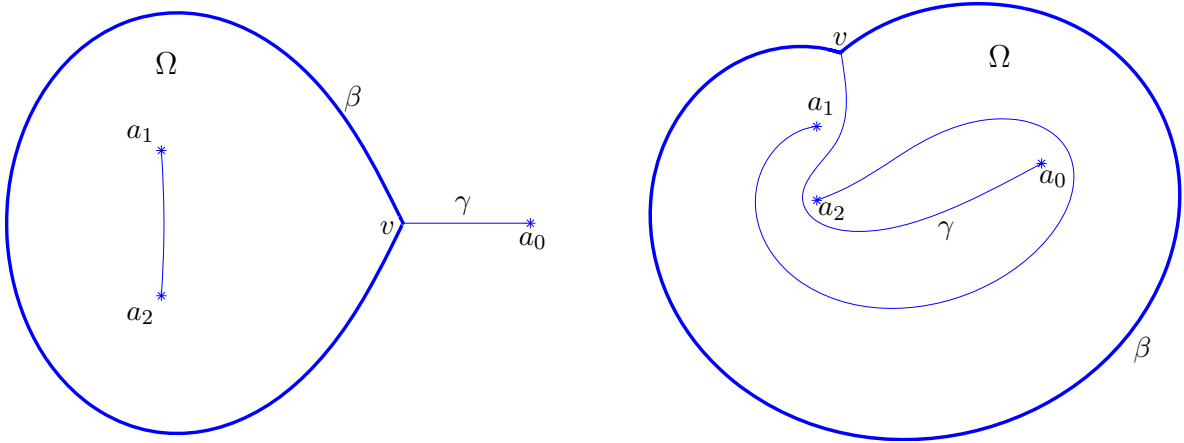
\begin{figure}[htb]
\centering \begin{tabular}{ll} \hspace{-1.3cm}\mbox{\begin{overpic}[scale=0.64]%
{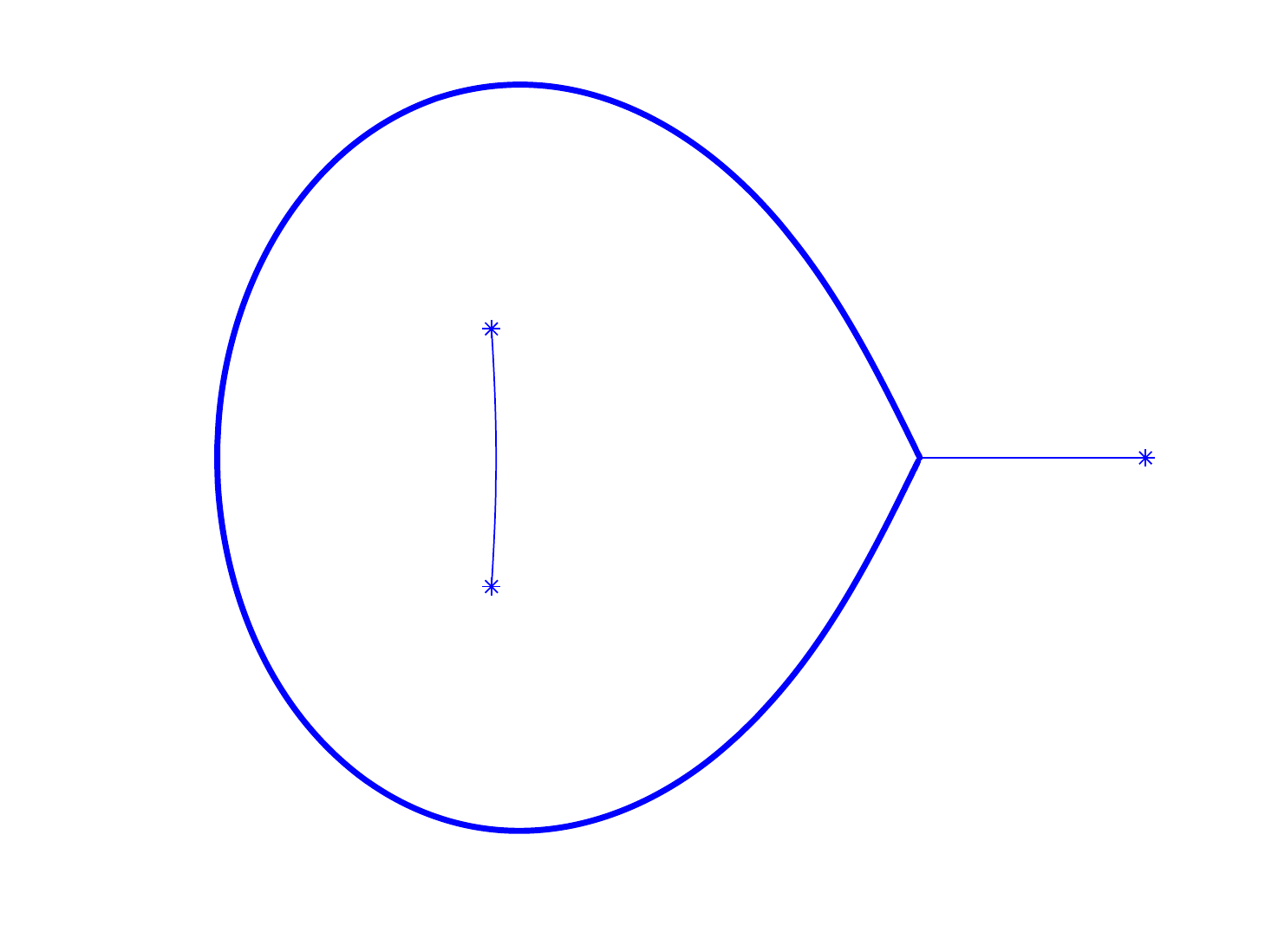}%
\put(38,60){\large $\Omega $}
     \put(69.5,37.5){$v $}
        \put(34,51){$a_1$}
           \put(64.5,55.5){$\beta $}
      \put(34,25.5){$a_2 $}
     \put(89,36){$a_0$}
 \put(80,41){$\gamma $}
\end{overpic}} &
\hspace{-1cm}\mbox{\begin{overpic}[scale=0.65]%
{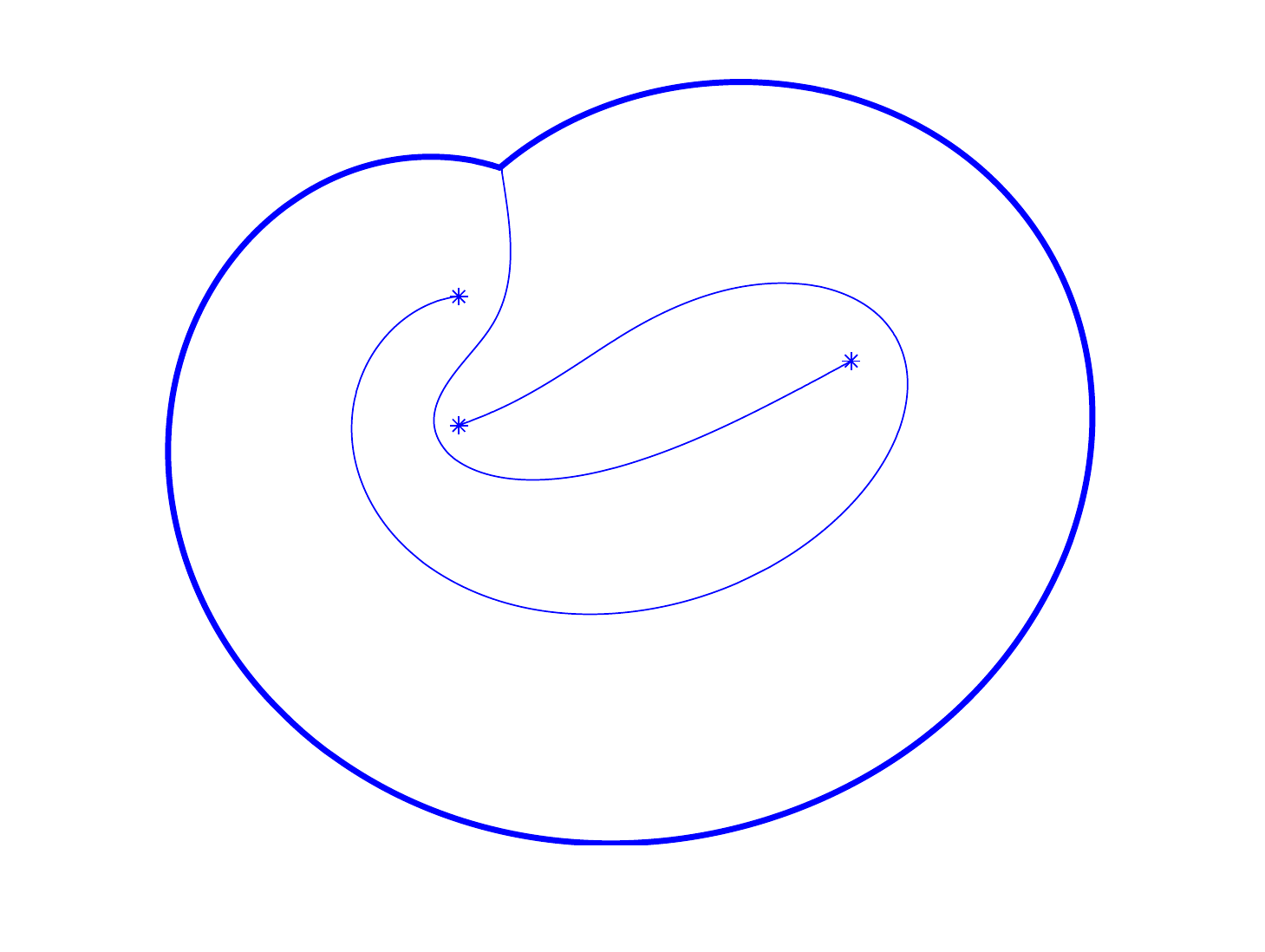}%
\put(60,60){\large $\Omega $}
     \put(38.5,63){$v $}
        \put(35,54){$a_1$}
      \put(36.5,39.5){$a_2 $}
       \put(80,20){$\beta $}
     \put(67,44){$a_0$}
 \put(53,37){$\gamma $}
\end{overpic}}
\end{tabular}
\caption{Critical graph of a quadratic differential in an exterior (left) and closed interior configurations.}
\label{fig:weirdtrajectories}
\end{figure}

In order to prove that our classification exhausts all the possibilities it is convenient to single out the following simple statement:
\begin{proposition}
\label{prop:onetrajectory}
The quadratic differential $\varpi_v$ is closed if and only if there exists one critical trajectory of $\varpi_v$.
\end{proposition}
\begin{proof}
Obviously, we only need to prove sufficiency. Assume that $\gamma$ is a critical trajectory joining for instance $v$ and $a_1$. In this case, taking into account the residue at infinity, we conclude that
$$
\Re \int_{a_2} ^{a_3} \sqrt{\frac{t-v}{A(t)}}\, dt =0
$$
if we integrate along a simple arc in $\Omega \setminus \gamma $ joining both poles. This means that
$$
\Re \int_{a_2} ^{z} \sqrt{\frac{t-v}{A(t)}}\, dt
$$
is a harmonic function on the Riemann surface $\mathcal R$ of $y^2=(x-a_2)(x-a_3)$ with two cuts along the lift of $\gamma $ to $\mathcal R$. Reasoning as in \cite{Gonchar:87} we conclude that a zero level curve of this function connects $a_2$ and $a_3$, and its projection on $\C$ constitutes the second critical trajectory of $\varpi_v$. Since the critical graph $\Gamma _v$ of $\varpi_v$ is a compact set, the differential is closed.

The remaining case is analyzed in a similar way, and this concludes the proof.
\end{proof}

Recall that by construction (see Proposition \ref{thm:existenceclosedtrajectoryp2}), $\beta $ is part of the boundary $\beta^*$ of the maximal ring domain $\mathcal R_\infty$ swept out by homotopic closed trajectories of $\varpi_v$. Hence, in the case of an exterior configuration, $\gamma\subset \beta ^*$ is a critical trajectory. By Proposition \ref{prop:onetrajectory}, $v\in \mathcal V$. An analogous conclusion is obtained if $\gamma \setminus \{v\} \subset \Omega$ is critical.

Finally, assume that $\gamma \setminus \{v\} \subset \Omega$  is recurrent. According to Corollaries (1) and (2) of Theorem 11.2 in \cite{Strebel:84}, its limit set is a domain bounded by the closure of a critical trajectory. Since in this case no critical trajectories can exist in $\Omega$, we conclude that the closure of the limit set of $\gamma$ is $\overline \Omega$, which concludes the proof of the statements above.

\begin{figure}[htb]
\centering \begin{tabular}{ll} \hspace{-1.8cm}\mbox{\begin{overpic}[scale=0.65]%
{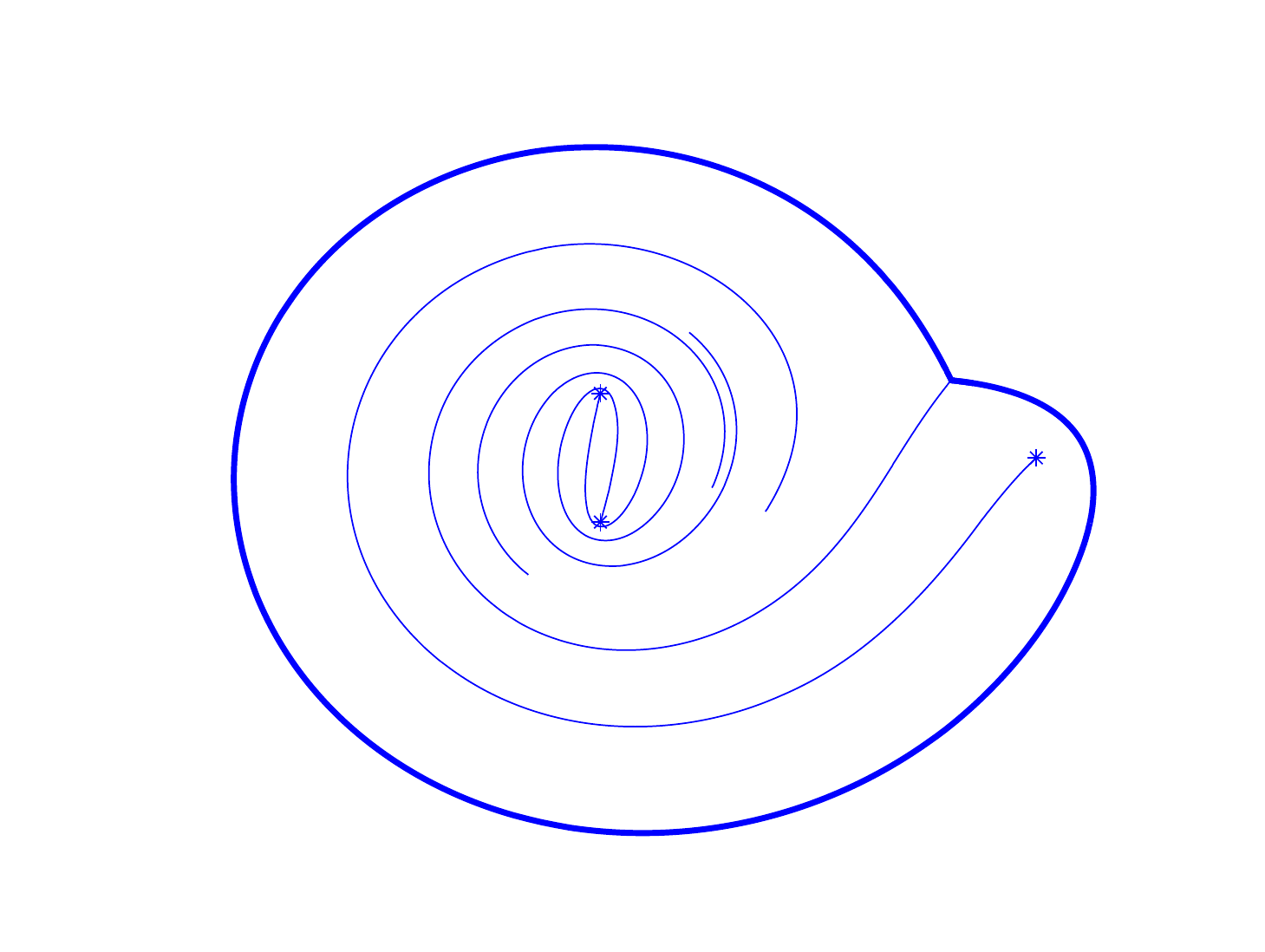}%
     \put(60,54){\large $\Omega $}
       \put(72,44){$v $}
   \put(81,36){$a_0 $}
   \put(27.5,60){$\beta $}
\end{overpic}} &
\hspace{-1.5cm}\mbox{\begin{overpic}[scale=0.65]%
{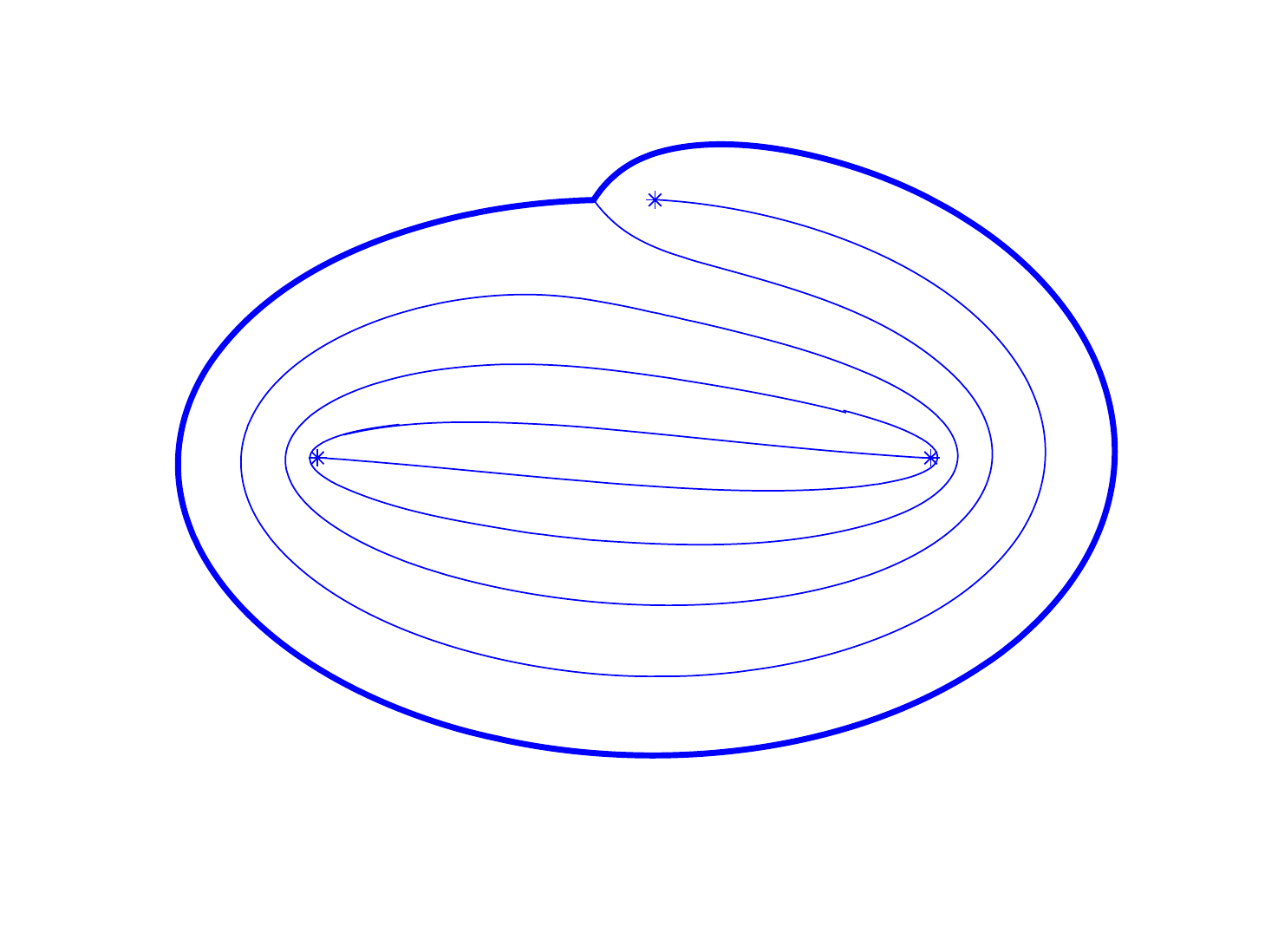}%
     \put(45,61){$v $}
        \put(52,61){$a_0$}
       \put(80,21){$\beta $}
\end{overpic}}
\end{tabular}
\caption{Trajectories of a quadratic differential in an interior recurrent (left) and  interior closed configurations.}
\label{fig:recurrent}
\end{figure}
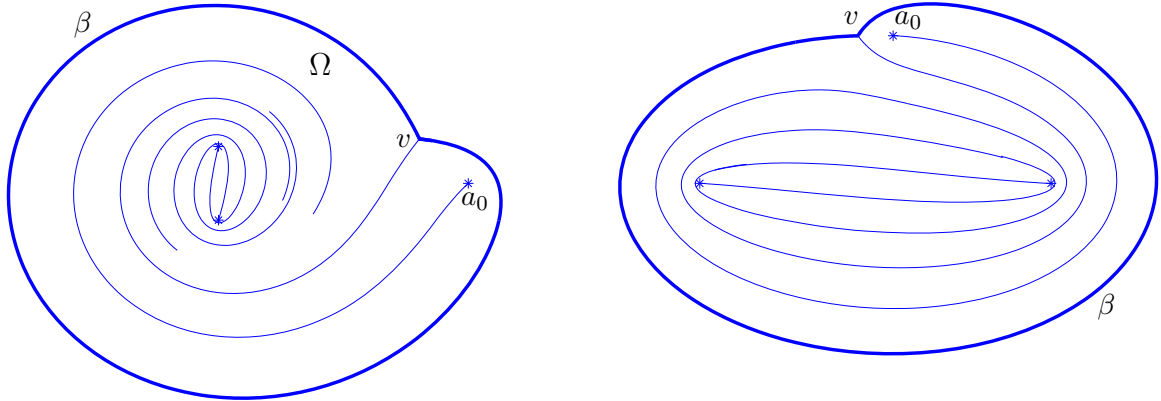

Let us summarize part of our conclusions in the following statement:
\begin{theorem}
\label{thm:structuretrajectories}
Let  $v\in \mathcal V \setminus \AA^* $. Then the critical graph $\Gamma_v$ of $\varpi_v$ is the union of three critical trajectories,
$$
\Gamma_v=\alpha \cup \beta \cup \gamma ,
$$
such that $\beta=\beta (v) $ is a closed loop that contains $v$, $\alpha=\alpha (v) $ is an arc in the bounded component $\Omega$ of $\C\setminus \beta$ joining two poles from $\AA$, and $\gamma =\gamma (v)$ connects the remaining pole with the zero $v$.

The set $\C\setminus \Gamma_v$ is the union of two disjoint domains:
\begin{itemize}
\item the bounded component $\Omega \setminus \Gamma_v$ of $\C\setminus \Gamma_v$ is a ring domain bounded either by $\alpha $ and $\beta $ (for an exterior configuration) or by $\Gamma_v$ (for an interior configuration);
\item the unbounded component $\mathcal D \setminus \Gamma_v =\mathcal R_\infty$ of $\C\setminus \Gamma_v$ is a disc domain in $\overline \C$, bounded either by $\beta $ and $\gamma $ (for an exterior configuration) or by $\beta$ (for an interior configuration).
\end{itemize}
\end{theorem}

For what follows we will fix an  orientation of both arcs $\alpha $ and $\gamma $. For instance, we can agree that $\alpha$ goes from $a_i$ to $a_j$ if $i<j$, and $\gamma $ always goes from a pole to $v$.

\subsection{Homotopic type of a closed differential} \label{sec:homotopicType2}

The partition of the plane by the critical graph $\Gamma_v$, associated to the closed quadratic differential $\varpi_v$ and described in Theorem \ref{thm:structuretrajectories}, allows us to introduce a geometric characterization of the trajectories of $\varpi$ ``at large''.

For $v\in \mathcal V$ we define the homotopic type of $\varpi_v$ as the free homotopic class in $\C\setminus \AA$ of any of its closed trajectories in $\Omega\setminus \Gamma_v$. Any such a trajectory is a closed Jordan curve in $\C\setminus \AA$ containing exactly two points from $\AA$ (say, $a_1$ and $a_2$) in its interior, and leaving in the exterior the third pole, $a_0$. We can further think that the (Carath\'{e}odory) boundary component given by the ``two-sided $\alpha $'' belongs to this homotopy class (we think of the two-sided $\alpha $  as the closed curve with $a_1$ and $a_2$ in its interior). Then, without loss of geometric information, we can identify the homotopic class of the two-sided $\alpha $ with the homotopic class of  $\alpha $ itself, considered now as an arc with fixed endpoints $a_1$ and $a_2$ in $\C\setminus \{a_0\}$.

Thus, for $v\in \mathcal V\setminus \AA^*$ the homotopic type of $\varpi_v$ has a combinatorial component (namely, which two of the three poles from $\AA$ are joined by $\alpha $) and the geometric one (given by the homotopy of $\alpha $ in the punctured plane with the third pole removed).

It follows from the general theory that there exists a one-(real) parametric family of closed differentials $\varpi_v$ with the prescribed homotopic type. One standard way to parametrize such a family is by the $\varpi_v$-lengths of the trajectories in $\Omega \setminus \Gamma_v$; another one is based on the $\varpi_v$-lengths of the conjugate (orthogonal) trajectories. More precisely, recall from Section \ref{sec:qd} that the $\varpi_v$-length of a curve $\tau$ is
$$
\|\tau \|_{\varpi_v }= \frac{1}{\pi}\, \int_{\tau}
\sqrt{\left|\frac{t-v}{A(t)}\right| } \, |dt|\,.
$$
By definition of the (horizontal) arc, given $v\in \mathcal V$, all closed trajectories of $\varpi_v$ in $\Omega \setminus \Gamma_v$ have the same $\varpi_v$-length, equal to
$$
 l_v \isdef 2 \, \|\alpha  \|_{\varpi_v }.
$$
Moreover, $l_v$ is the minimum of the $\varpi_v$-length of a closed Jordan curve separating the boundary components of $\Omega\setminus \Gamma_v$; it is called the \emph{length of the circumferences} of the cylinder associated with $\Omega \setminus \Gamma_v$ (see \cite[Chapter VI]{Strebel:84}).

The conjugate value
$$
 h_v \isdef \inf \left\{ \|\tau  \|_{\varpi_v }:\, \tau \text{ connects boundary components of } \Omega\setminus \Gamma_v \right\}
$$
is called the \emph{height of the cylinder} associated with $\Omega \setminus \Gamma_v$. Again, it coincides with the $\varpi_v$-length of any arc of orthogonal trajectory connecting the  boundary components of $\Omega\setminus \Gamma_v$.

\begin{lemma}
\label{lemma:homotopicTypes}
Let $\alpha $ be a Jordan arc lying in $\C\setminus \AA$ (except for its endpoints) and connecting two poles from $\AA$. Then for any value $h>0$ there exists a unique $v\in \mathcal V$ such that $\varpi_v$ has the homotopic type $\alpha  $, and $h_v = h$.
\end{lemma}
This Lemma may be proved by reduction to a general Theorem 21.1 from \cite[\S 21]{Strebel:84} on the existence of finite differentials. Our quadratic differentials $\varpi_v$ are not finite (due to the double pole at infinity), but they can be approximated by finite differentials in a way preserving essential characteristics.

The same family of differentials of the homotopic type $\alpha $ may be parametrized alternatively by the length $l_v$ of the circumferences  of the cylinder associated with $\Omega \setminus \Gamma_v$. In this case each homotopic type has a minimal (strictly positive) admissible length:
\begin{lemma}
\label{lemma:lengths}
Let $\alpha $ be a Jordan arc lying in $\C\setminus \AA$ (except for its endpoints) and connecting two poles from $\AA$. Define
$$
L =L(\alpha) \isdef \inf \{l_v:\, v\in \mathcal V \text{ and $\varpi_v$ has the homotopic type $\alpha  $}  \}.
$$
Then $L>0$ and for any $l>L$ there exists a unique $v\in \mathcal V$ such that $\varpi_v$ has the homotopic type $\alpha  $, and $l_v = l$.
\end{lemma}
Lemma \ref{lemma:lengths} may be reduced to Lemma \ref{lemma:homotopicTypes} or derived from the general existence theorems related to the moduli problem, see \cite[\S 21]{Strebel:84}.

\subsection{Correspondence between closed differentials and $\AA$-critical measures}

We have proved in Section \ref{sec:qdandcriticalmeasures} that for any signed $\AA$-critical measure $\mu$ with $\mu(\C)=1$ there exists a closed quadratic differential $\varpi_v$ in terms of which the measure and its potential may be analytically expressed. In general, this is not a bijection, since many critical measures correspond to the same quadratic differential. It follows from the proof of Proposition~\ref{uniquenessmeasure} that for $p=2$ a one-to-one correspondence between closed differentials and signed critical measures is restored if we consider the $\AA$-critical measures $\mu $ with the additional property that $\C\setminus \supp(\mu )$ is connected. This subclass is the most important for applications. 

Let $v\in \mathcal V\setminus \AA^*$. Obviously, $\sqrt{V/A}$, with $V(z)=z-v$, has a single-valued branch in $\C\setminus \overline{\alpha \cup \gamma }$; the critical trajectories $\alpha =\alpha (v)$ and $\gamma =\gamma (v)$ were introduced in Theorem \ref{thm:structuretrajectories}. We fix the branch by requiring that $\lim_{z\to \infty} z \sqrt{V/A}(z)=1$. Next we choose the positive (anti-clockwise) orientation in a neighborhood of infinity, that induces orientation on each side of  $\alpha$ and $\gamma$. We denote by the subindex ``$+$'' the boundary value of a function at $\alpha $ and $\gamma $ from the side where the induced orientation matches the given orientation of each arc (see the remark after Theorem \ref{thm:structuretrajectories}).

With this convention, and taking into account that  $\alpha$ and $\gamma$ are trajectories of $\varpi_v$, we conclude that
\begin{equation}\label{measureMuv}
    d\mu_v(z) \isdef \frac{1}{\pi i} \, \left(\sqrt{\frac{z-v}{A(z)}}\right)_+ \, dz,
\end{equation}
 defines a signed real measure on $ \overline{\alpha \cup \gamma }$. Moreover, taking into account the residue at infinity, we see that $\int d\mu_v=1$. Hence, we have proved the following
\begin{proposition}
\label{prop:uniquemapping}
For any $v\in \mathcal V\setminus \AA^*$ there exits a unique signed $\AA$-critical measure $\mu_v$ with $\mu_v(\C)=1$, such that $\supp (\mu_v)= \overline{\alpha \cup \gamma }$. Furthermore, $\mu_v$ is absolutely continuous on $\supp (\mu_v)$ with respect to the arc-length measure, and formula \eqref{measureMuv} holds.
\end{proposition}
\begin{rmk}
This construction can be extended in a natural way to the Chebotarev's compact ($v=v^*$) and to the degenerate cases when $v\in \AA$; in these situations, measure $\mu_{v^*}$ is positive. Hence, $v \mapsto \mu_v$ is a mapping from $\mathcal V$ into the set of signed unit measures on $\C$.
\end{rmk}

We introduce next an analytic function that will allow us to study the structure of the set $\mathcal V$ of points $v\in \C$ such that $\varpi_v $ is closed.

Let $v_0\in \mathcal V \setminus \{ v^*\}$; recall that $\alpha =\alpha(v_0)$ joins the two poles of $\varpi_{v_0}$ not connected with $v_0$ by a critical trajectory. In a simply connected neighborhood of $v_0$, disjoint with  $\alpha$,
\begin{equation}\label{def_Wnew}
    w(v)\isdef \int_{ \alpha   }  \left(  \sqrt{\frac{t-v}{A(t)}}\right)_+\, dt
\end{equation}
is analytic in $v$, single-valued, and
\begin{equation}\label{def:Derw0}
    w' (v_0 )= i \int_{\alpha  }  \frac{1}{ \sqrt{(t-v_0) A(t)}}\, dt \neq 0,
\end{equation}
since this is a period of a holomorphic differential on the elliptic Riemann surface of the algebraic function $y^2=(t-v_0) A(t)$. This construction  defines an analytic and multi-valued function  $w$ in $\C$, with $w'\neq 0$; however, formula \eqref{def_Wnew} allows to specify a single-valued branch of $w$ in a neighborhood of a point only in $\C\setminus \{ v^*\}$.

\begin{proposition}
\label{prop:localstructureV}
For every $v_0\in \mathcal V \setminus \{ v^*\}$ there exists a neighborhood $B $ of $v_0$ such that $\mathcal V\cap B $ contains an analytic arc $\ell $ passing through $v_0$, and such that the homotopic class of the trajectories of $\varpi_v$ for $v\in \ell $ is invariant.
\end{proposition}
\begin{proof}
In a small and simply-connected neighborhood $B$ if $v_0$ consider the branch of $w$ given by formula \eqref{def_Wnew}, so that $w(v_0)=\pi i (1-\mu(\gamma (v_0)))$. Since $w'(v_0)\neq 0$, the level curve
 $$
 \ell\isdef \{v:\, \Re w (v)= 0\}
 $$
is well defined in   $B$, and constitutes an analytic arc passing through $v_0$.
Clearly, if $v \in B \cap \mathcal V$ is such that  the homotopic class of the trajectories of $\varpi_v$ and $\varpi_{v_0}$ are the same, then necessarily $v\in \ell$. 
Reciprocally, assume $v\in \ell \cap B$, and consider the trajectory of $\varpi_v$ that emanates from the same pole as $\gamma (v_0)$. Due to the continuity of the level curves of
$$
\Re \int ^{z} \sqrt{\frac{t-v}{A(t)}}\, dt
$$
with respect to a variation of $v$, this trajectory is either critical (and then the proposition is proved) or recurrent. In the latter case it must intersect the orthogonal trajectory of $\varpi_v$ starting from $v$ (see \cite[\S 11]{Strebel:84}), which contradicts the hypothesis that $\Re w (v)= 0$.
\end{proof}

Now we can describe completely the structure of the set $\mathcal V$:
\begin{theorem} \label{cor:structureV}
The set $\mathcal V$ is a union of a countable number of analytic arcs $\ell_k$, $k\in \Z$, each connecting $v^*$ and $\infty$.

Two arcs from $\mathcal V$ are either identical or have $v^*$ as the only finite common point. The homotopic type of the critical trajectories of $\varpi_v$ in $\C\setminus  \AA $ remains invariant on each arc $\ell_k\setminus \AA^*$.

There are three distinguished arcs $\ell_k$, $k\in \{0, 1, 2\}$, such that
\begin{enumerate}
\item[(i)] $\ell_k$ connects $v^*$ with infinity and passes through $a_k$;
\item[(ii)] for every $v\in \ell_k$ the homotopic class of trajectories of the closed quadratic differential $\varpi_v$ is trivial;
\item[(iii)] Function $\mu_v(\gamma(v))$ is monotonically decreasing  from $\mu_{v^*}(\gamma(v^*))$ to $- \infty$ as $v$ travels $\ell_k$ from $v^*$ to $\infty$.
\end{enumerate}
\end{theorem}
\begin{proof}
Assume that $a_j$'s are not collinear (the analysis of the collinear situation is simpler). For $v_0=a_0\in \mathcal V$, the arc $\gamma (v_0)$ vanishes, and $\alpha (v_0)=[a_1, a_2]$ is the straight segment joining $a_1$ and $a_2$, so that we can fix the single-valued branch of $w$ in a neighborhood of $a_0$ by
\begin{equation}\label{branchOfW}
    w(a_0)=   \int_{ [a_1, a_2]   }   \frac{1}{\left( \sqrt{(t-a_1)(t-a_2)}\right)_+}\, dt= \pi i.
\end{equation}

Denote by $\mathcal H$ the half plane containing $a_0$ and determined by the straight line passing through $a_1$ and $a_2$.
Then
$$
 w(v)= \int_{[a_1, a_2]    }  \left(  \sqrt{\frac{t-v}{A(t)}}\right)_+\, dt,
$$
and \eqref{branchOfW} determines the single-valued branch of $w$ in $\mathcal H $. Let $\Gamma_0$ be the level curve
$$
\Gamma_0\isdef \{ z\in \overline {\mathcal H} :\, \Re w(z)=0\}.
$$
We have that
$$
w(v^*)=\pi i\,  \mu_{v^*}  \left( \Gamma^* \setminus \gamma (v^*)\right)
$$
($\Gamma^* \setminus \gamma (v^*)$ is the union of two arcs of the Chebotarev compact joining $v^*$ with $a_1$ and $a_2$). Hence, $v^*\in \Gamma_0$.

Since the rotations and translations of the plane do not affect the character of the level curves of $w$, we can assume that both $a_1, a_2 \in \R$.  Then it is immediate to see that $\Gamma _0$ can intersect  $\R$ at a single point (which belongs to the segment $[a_1, a_2]$). The other end of $\Gamma_0$ must diverge to infinity. This establishes the existence and properties of the distinguished arcs $\ell_k$, $k\in \{0, 1, 2\}$, described above.

Consider now the analytic function $w$ in the infinite sector delimited by two contiguous distinguished arcs $\ell_k$, $k\in \{0, 1, 2\}$. Fix there a $v_0\in \mathcal V$ and take  the single-valued branch determined by the condition
$$
w(v_0) =   \int_{ \alpha(v_0)   }  \left(  \sqrt{\frac{t-v_0}{A(t)}}\right)_+\, dt= \pi i \mu_{v_0}(\alpha(v_0) ).
$$
Then the level curve $\ell=\{\Re w(v)=0\}$ is an analytic curve passing through $v_0$, that can intersect the boundary of the sector only at $v^*$. Hence, $\ell$ joins $v^*$ with $\infty$. The number of different curves $\ell$ is given by the number of different homotopic types of closed trajectories, which is countable. This concludes the proof.
\end{proof}

\begin{rmk}
As $v$ approaches $ v^*$ along an arc $\ell_k \subset \mathcal V$, the support $\supp(\mu_v)=\overline{\alpha \cup \gamma }$ tends to the Chebotarev set $\Gamma^*$, but possibly covered several times, in accordance with the homotopy class of $\varpi_v$ on $\ell_k$.
\end{rmk}

Finally, it is convenient to consider another independent parametrization of measures $\mu_v$, $v\in \mathcal V$, in order to connect the characteristics of their logarithmic potential $U^{\mu_{v}}$ with the geometrically defined values of the corresponding quadratic differential $\varpi_v$. Applying Lemma \ref{lemma:Sconditions} we get
\begin{lemma}
\label{lemma:anotherparametr}
Any measure $\mu_v$, $v\in \mathcal V$, is characterized by the following property: $U^{\mu_{v}}$ is constant on each connected component of $\supp (\mu)$,
$$
U^{\mu_{v}}(z)\equiv c_\alpha \text{ for } z\in \alpha, \qquad U^{\mu_{v}}(z)\equiv c_\gamma  \text{ for } z\in \gamma ,
$$
and at any regular point of $\supp(\mu _v)$,
 \begin{equation}\label{SpropertyCritical}
\frac{\partial U^{\mu_{v}}(z)}{\partial n_+}= \frac{\partial U^{\mu_{v}}(z)}{\partial n_-},
 \end{equation}
where $n_\pm$ are the normal vectors to $\supp(\mu _v)$ pointing in the opposite directions. Moreover, we have the following relations:
$$
l_v= 2\mu_v(\alpha ) = 2\, (1- \mu_v(\gamma  ))= 2 \pi i w(v), \quad h_v= \pi |c_\alpha -c_\gamma | =\pi (c_\alpha -c_\gamma).
$$
\end{lemma}

\subsection{Positive $\AA$-critical measures} \label{sec:positiveAcritical}

Theorem \ref{cor:structureV} gives a complete description of the set $\mathcal V$ of points $v$ that make the quadratic differential $\varpi_v$ in \eqref{qdp2} closed. By Proposition \ref{prop:uniquemapping}, to every $v\in \mathcal V$ it corresponds a unique \emph{signed} $\AA$-critical measure $\mu_v$, given by formula \eqref{measureMuv}. Our next goal is to isolate the subset
\begin{equation}\label{defVhatPlus}
    \widehat {\mathcal V}_+\isdef  \{v\in \mathcal V:\, \mu_v \text{ is positive}\}.
\end{equation}

\begin{proposition}
\label{prop:positivity}
Let $v\in \mathcal V$. Measure $\mu_v$ is positive if and only if either $v=v^*$ or $v$ is in an exterior configuration.
\end{proposition}
See Section \ref{sec:globalStruc2} for the definition of the exterior configuration.
\begin{proof}
The case $v=v^*$ is trivial, so let us assume that $v\neq v^*$. Consider
$$
u(z)=\Re \int_v^z \sqrt{\frac{t-v}{A(t)}}\, dt\,, \quad \widetilde u(z)=\Im \int_v^z \sqrt{\frac{t-v}{A(t)}}\, dt\,.
$$
Since $\alpha=\alpha (v) $ and $\gamma=\gamma (v) $ are trajectories of $\varpi_v$, function $u$ is   single-valued and harmonic in $\C\setminus \Gamma$, continuous up to the boundary, and the closed trajectory $\beta $ (see Proposition \ref{thm:existenceclosedtrajectoryp2}) is its zero level curve. By the selection of the branch of the square root, $u(z)\sim \log|z|$ as $z\to \infty$, and we see that $u(z)>0$ for $z\in \C\setminus \left( \overline{\Omega} \cup \gamma \right)$, and $u(z)<0$ for $z\in \Omega  \setminus \Gamma$. In consequence,
$$
\frac{\partial }{\partial n} \, u(z)>0 \quad \text{ on }\alpha ,
$$
and on $\gamma$,
$$
\frac{\partial }{\partial n}\, u(z)\begin{cases}
>0, & \text{if $v$ is in an exterior configuration,} \\
<0, & \text{if $v$ is in an interior configuration,} \\
\end{cases}
$$
where $\partial/\partial n$ denotes the derivative in the sense of the outer normals. By the Cauchy-Riemann equations,
$$
\frac{\partial }{\partial s}\, \widetilde u(z) = \frac{\partial }{\partial n} \, u(z),
$$
where $\partial/\partial s$ is the derivative along each shore of the cuts $\alpha $ and $\gamma $ in the direction of the induced orientation. Hence, we conclude that $\mu_v\big|_{\alpha}$ is always positive, while $\mu_v\big|_{\gamma}$ is negative if and only if $v$ is in an exterior configuration.
\end{proof}
\begin{rmk} \label{rmk:positivity}
Observe that we have proved that always $\mu_v\big|_{\alpha}>0$. For $v\neq v^*$, by construction $\mu_v(\alpha )+\mu_v(\gamma )=1$, and $\mu$ does not change sign on each connected component of $\Gamma$, so that $\mu_v$ is positive if and only if $\mu_v(\alpha)\leq 1$. An equivalent condition can be stated in terms of the $\varpi_v$-length of the critical trajectories $\alpha $ and $\gamma $:
 \begin{equation}\label{muPositive}
    \mu_v\geq 0 \quad \Leftrightarrow \quad \| \alpha \|_{\varpi_v} + \| \gamma  \|_{\varpi_v} =1.
 \end{equation}
\end{rmk}
\begin{rmk}
Figure \ref{fig:intersectionNegative} illustrates that only in the exterior configuration the $\varpi_v$-rectangles intersect the support of $\mu_v$ only once (cf.~Lemma \ref{rulgard}).
\end{rmk}

Proposition \ref{prop:positivity} provides an ``implicit'' geometric description of the set $\widehat{\mathcal V}_+$. Our main result of this section describes this set completely:
\begin{theorem} \label{thm:structureVPlus2}
Let $\ell_k$, $k\in \{0, 1, 2\}$, be the distinguished arcs in $\mathcal V$ described in Theorem  \ref{cor:structureV}. The set  $\widehat{\mathcal V}_+$ is the union of the sub-arcs $\ell_k^+$ of each $\ell_k$, $k\in \{0, 1, 2\}$, connecting $a_k$ with the Chebotarev's center $v^*$ (and lying in the convex hull of $\AA$).

Furthermore, let us denote $m_k \isdef \mu_{v^*}(\Gamma_k^*)$, $k\in \{ 0, 1, 2\}$, where $\Gamma_k^*$ is the arc of the Chebotarev compact $\Gamma^*$ connecting $v^*$ with $a_k$ ($ m _0+m _1+m _2=1$). If  $v\in \ell_k \cap \widehat{\mathcal V}_+$, $k\in \{ 0, 1, 2\}$, then the trajectory $\gamma(v)$ connects $v$ with the pole $a_k$ and
\begin{equation}\label{inequalityMu}
    0\leq \mu_v(\gamma(v)) \leq m_k.
\end{equation}
In this case both trajectories $\gamma(v)$ and $\alpha (v)$ are homotopic to a segment. The bijection $\mu_v(\gamma(v)) \leftrightarrow v$ is a parametrization of the set $\ell_k \cap \widehat{\mathcal V}_+$ by points of the interval $[0, m _k]$.
\end{theorem}
\begin{proof}
Straightforward estimates show that $v\to \mu_v(\alpha (v))$ is unbounded on each arc $\ell_k\subset \mathcal V$. Furthermore,
$$
\mu_v(\alpha (v)) =1 \quad \Leftrightarrow \quad \mu_v(\gamma (v)) =0 \quad \Leftrightarrow \quad  v\in \AA.
$$
Since always $\mu_v(\alpha (v))>0$ (see Remark \ref{rmk:positivity}), we conclude that $\mu_v(\alpha (v))$ takes values in $(0,1)$ only on the portions of the distinguished arcs $\ell_k$, $k\in \{ 0, 1, 2\}$, joining the Chebotarev center $v^*$ with each pole. For any other arc $\ell_k$, $\mu_v(\alpha (v))>1$, and $\mu_v$ is not positive.
\end{proof}

\begin{figure}[h!]
\centering \begin{tabular}{l@{\hspace{15mm}}c} 
\mbox{\begin{overpic}[scale=0.29]%
{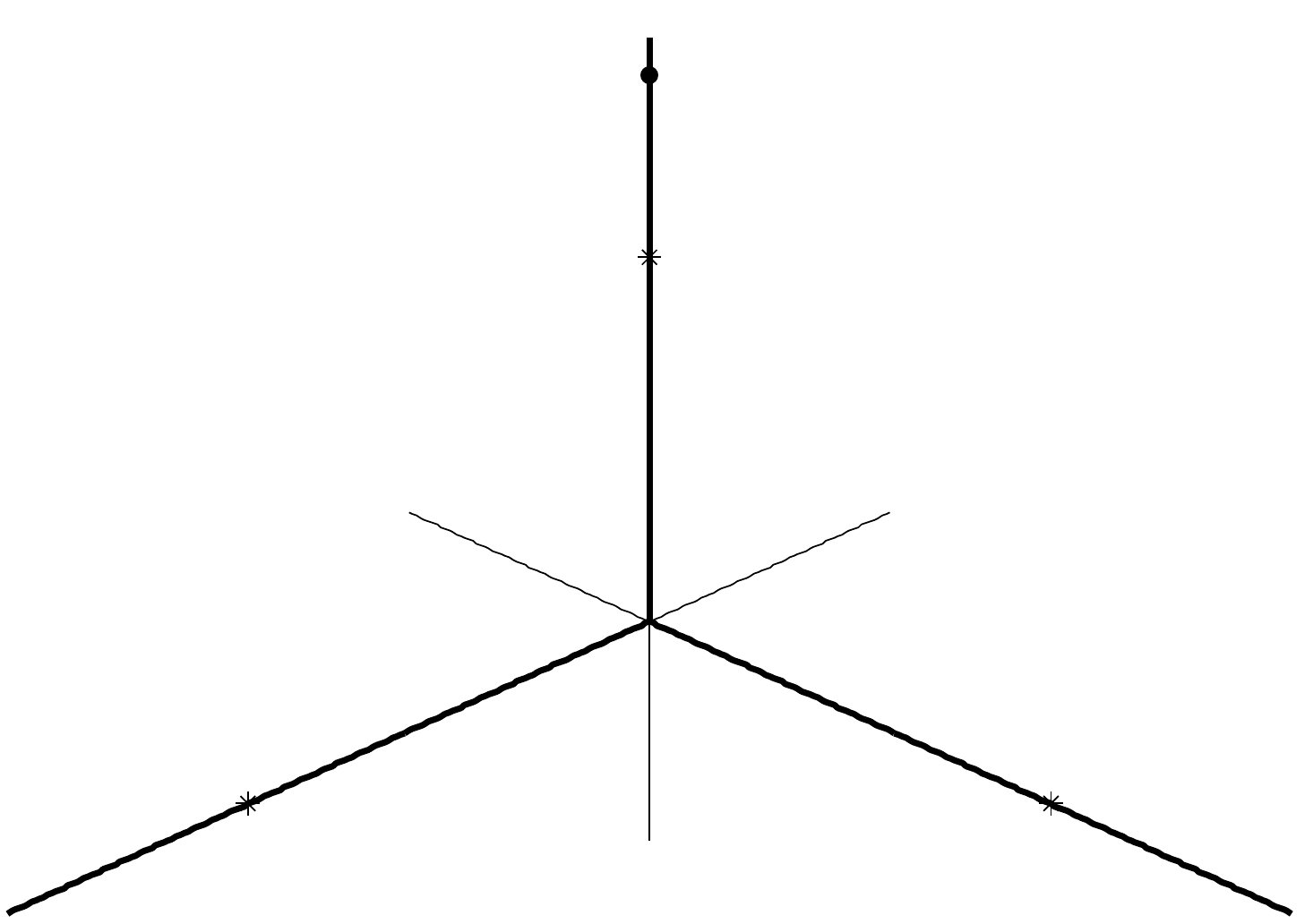}%
\put(58,45){\emph{i)} $\mu_v(\gamma(v))<0$}
      \put(40,49){\small $a_k $}
      \put(40,63){\small $v $}
\end{overpic}} &
\mbox{\begin{overpic}[scale=0.32]%
{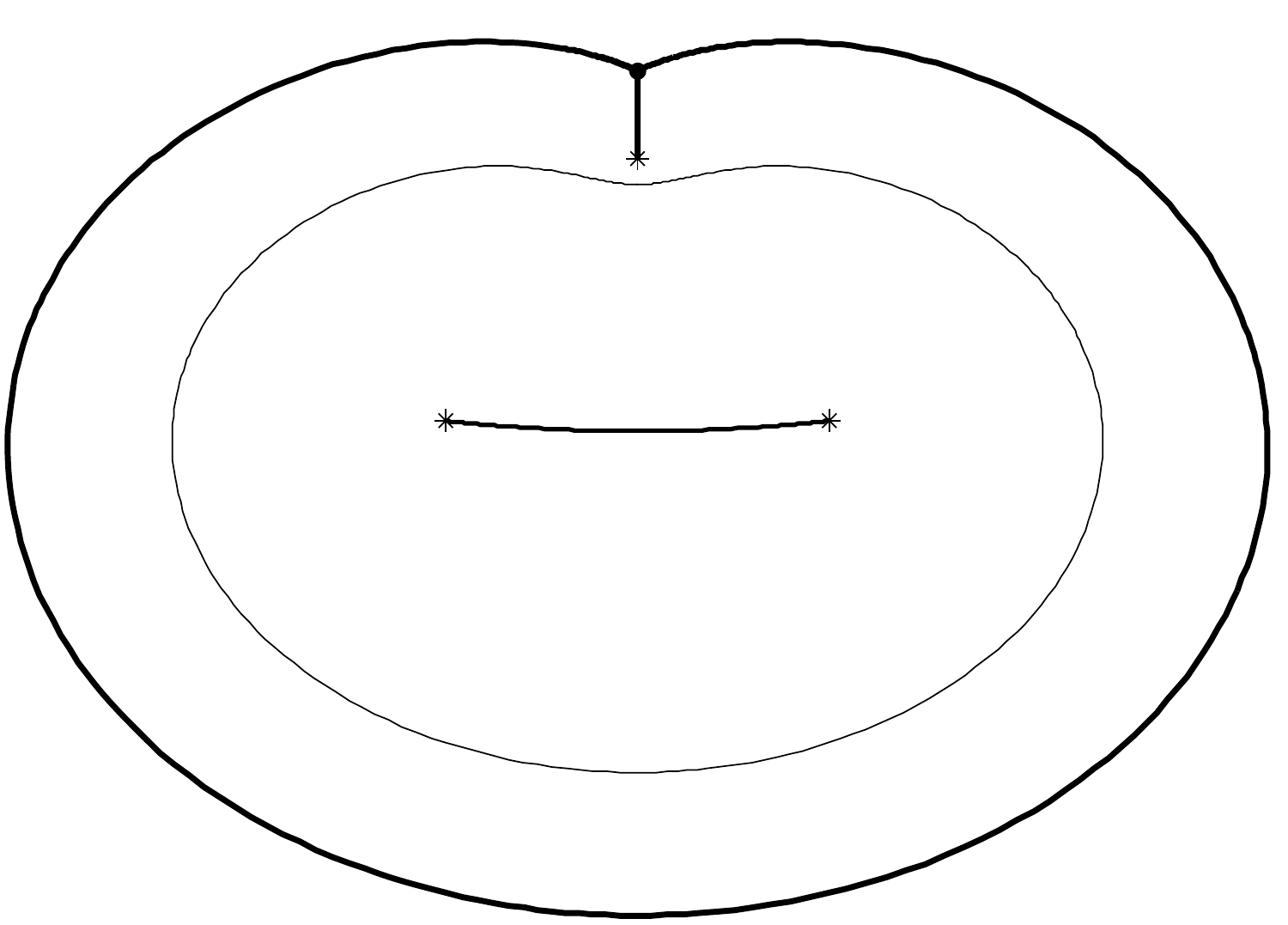}%
\end{overpic}}
\\ 
\mbox{\begin{overpic}[scale=0.29]%
{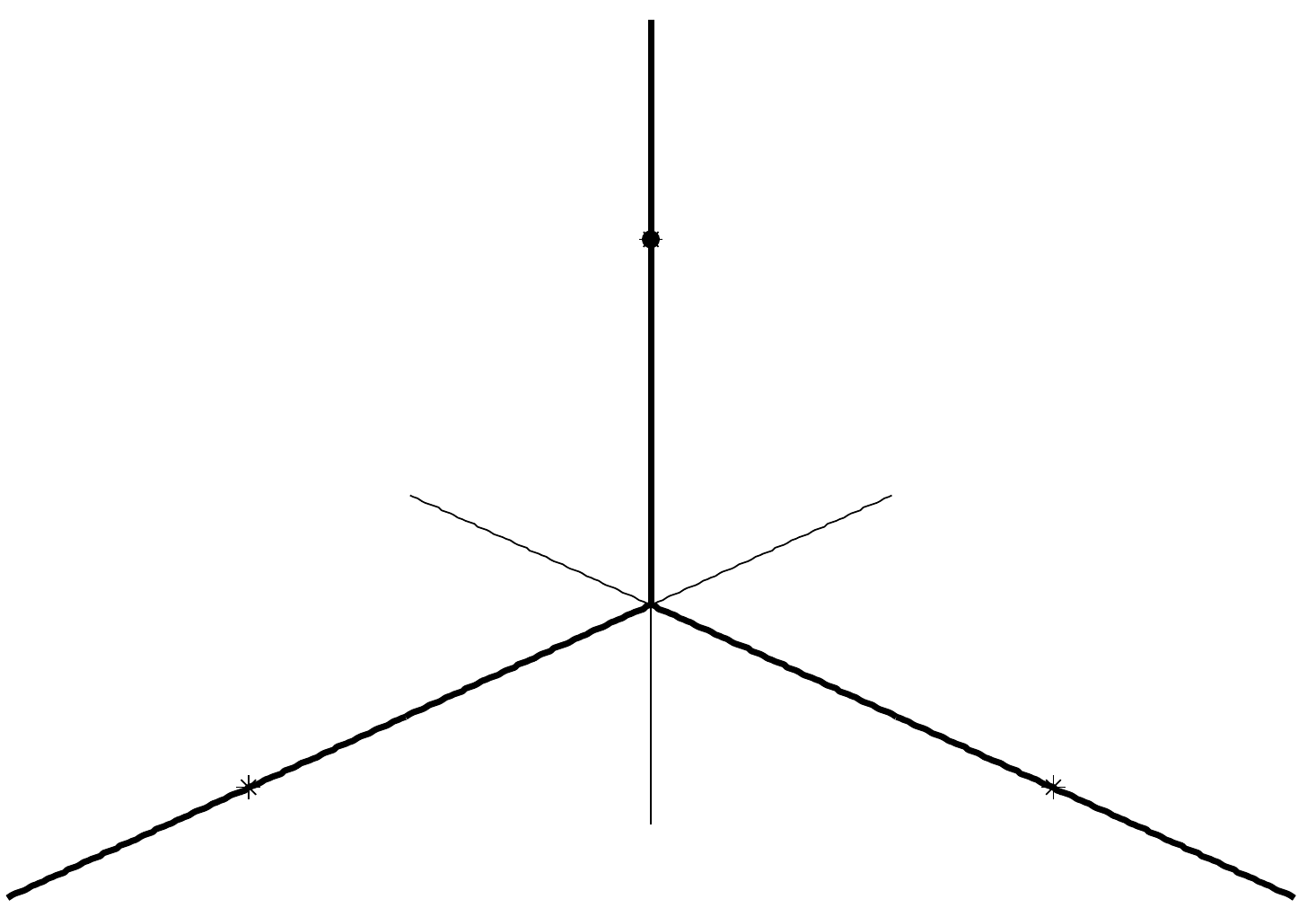}%
\put(58,45){\emph{ii)} $ \mu_v(\gamma(v))=0$}
      \put(25,49){\small $v=a_k $}
\end{overpic}} &
\mbox{\begin{overpic}[scale=0.32]%
{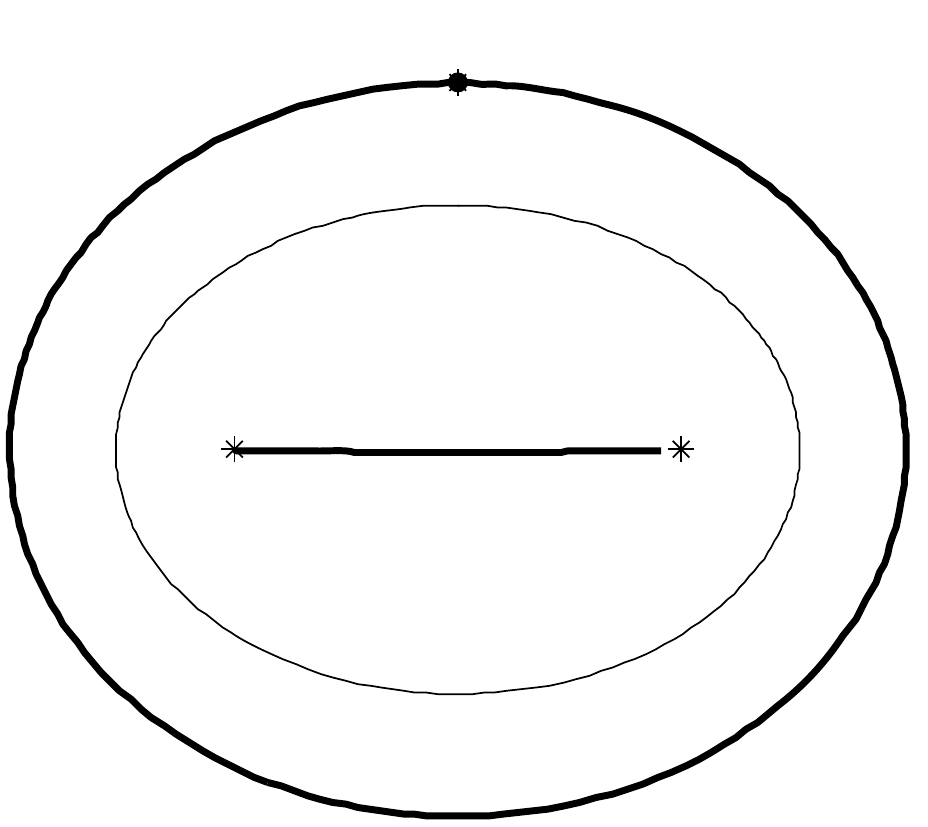}%
\end{overpic}} \\ 
\mbox{\begin{overpic}[scale=0.29]%
{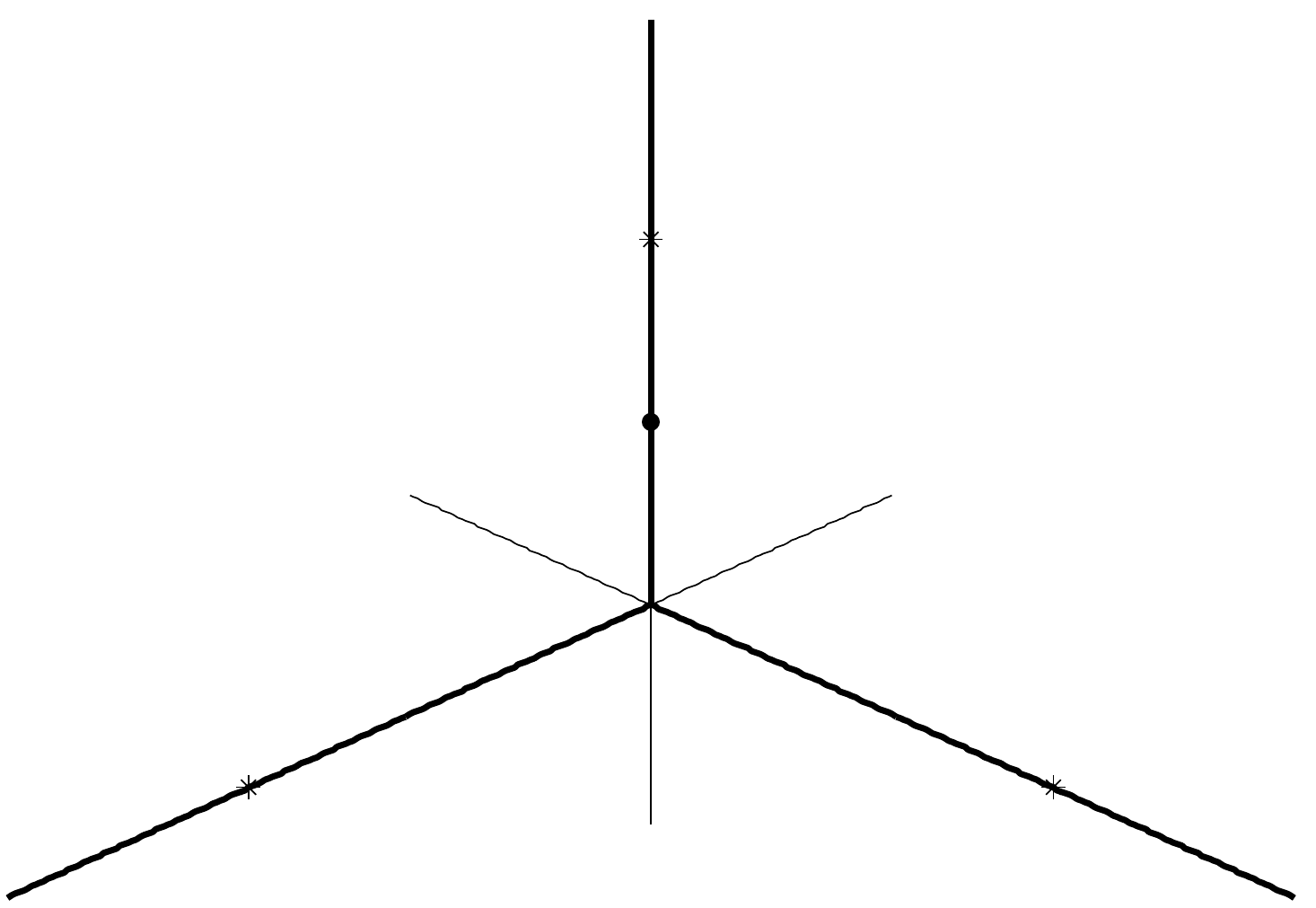}%
\put(58,45){\emph{iii)} $0<\mu_v(\gamma(v))<m_k$}
      \put(40,49){\small $a_k $}
      \put(43,35){\small $v $}
\end{overpic}}
 &
\mbox{\begin{overpic}[scale=0.32]%
{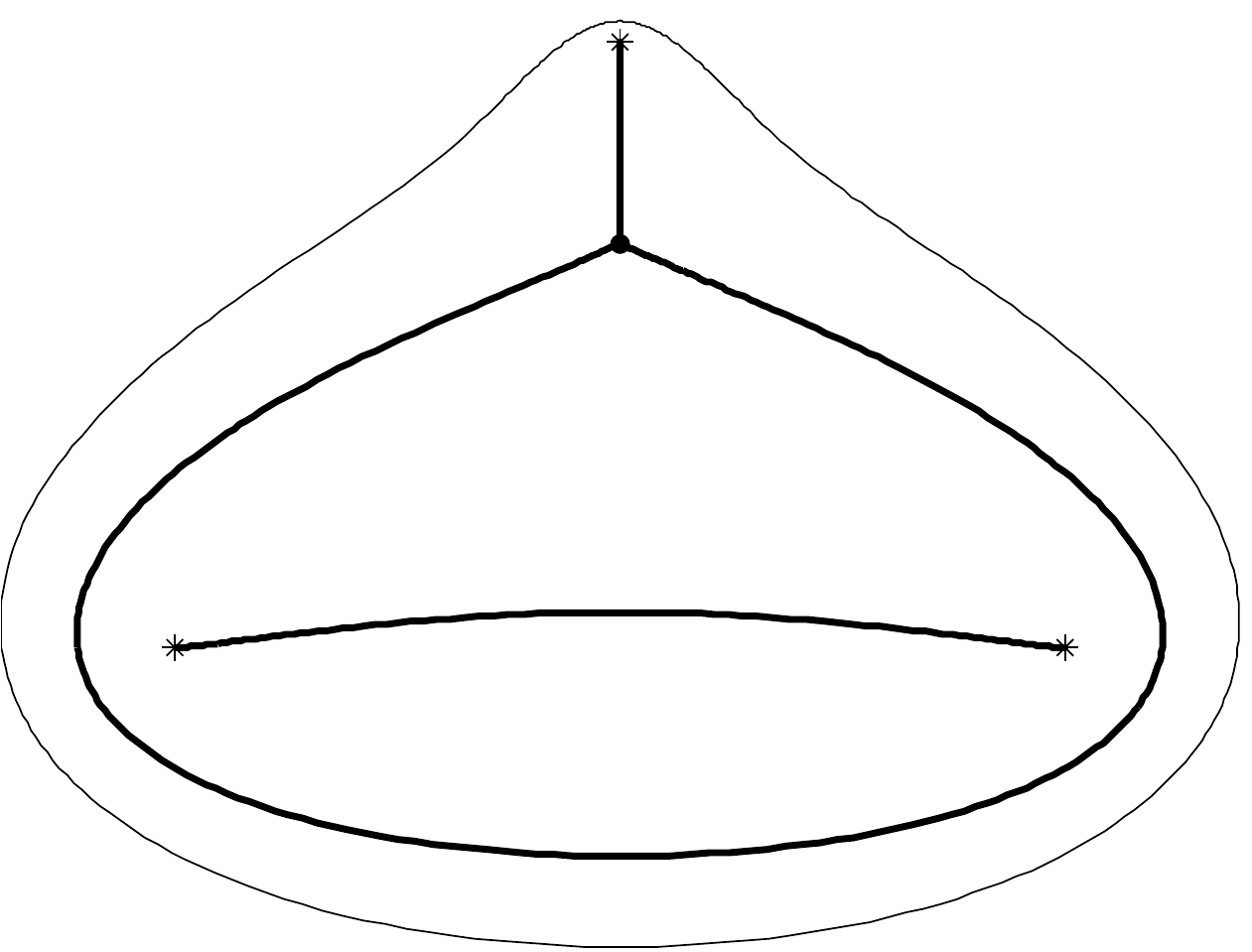}%
\end{overpic}}
 \\ 
\mbox{\begin{overpic}[scale=0.29]%
{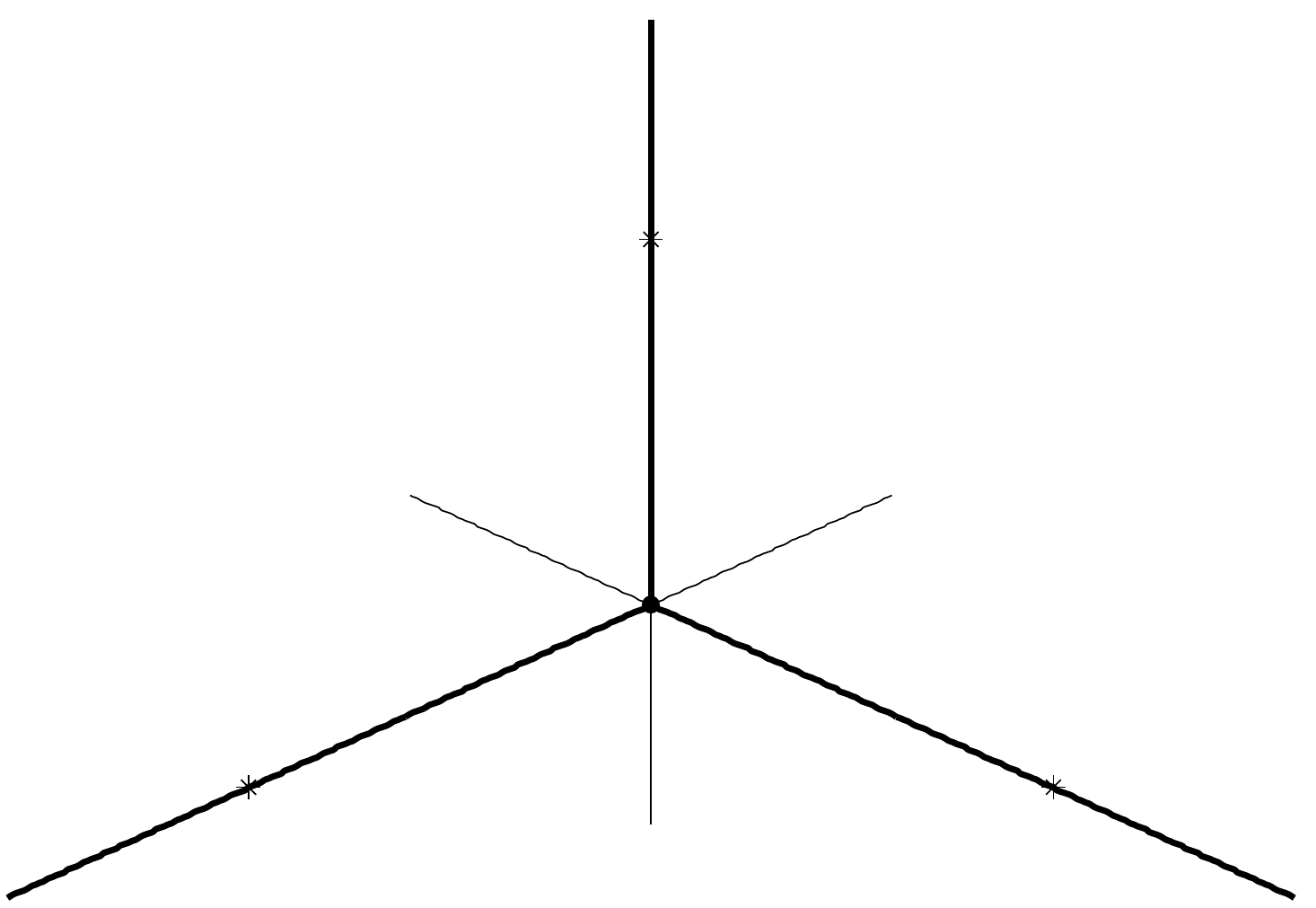}%
\put(58,45){\emph{iv)} $\mu_v(\gamma(v))=m_k$}
      \put(40,49){\small $a_k $}
      \put(23,20){\small $v=v^* $}
\end{overpic}} &
\mbox{\begin{overpic}[scale=0.32]%
{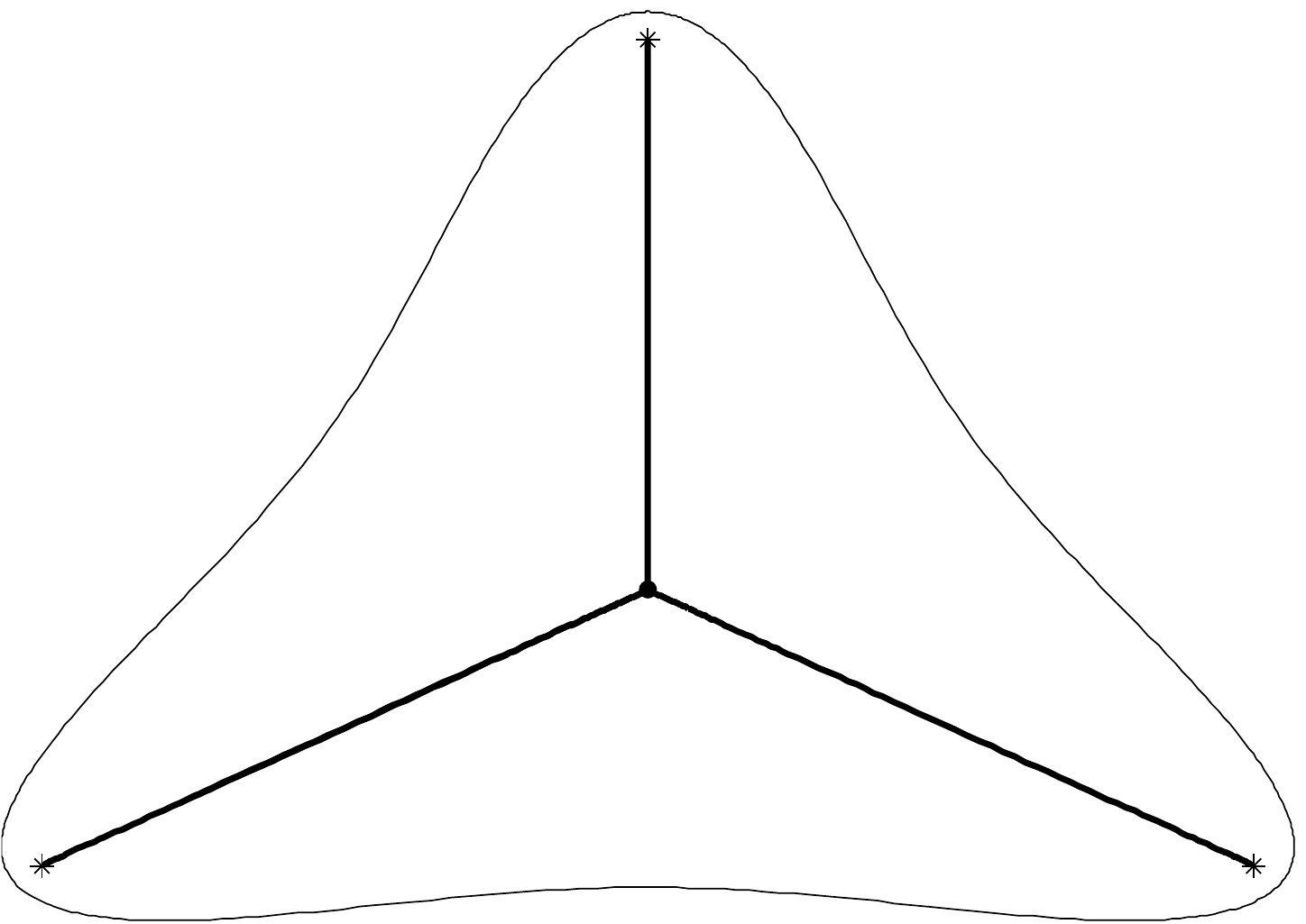}%
\end{overpic}} \\ 
\mbox{\begin{overpic}[scale=0.29]%
{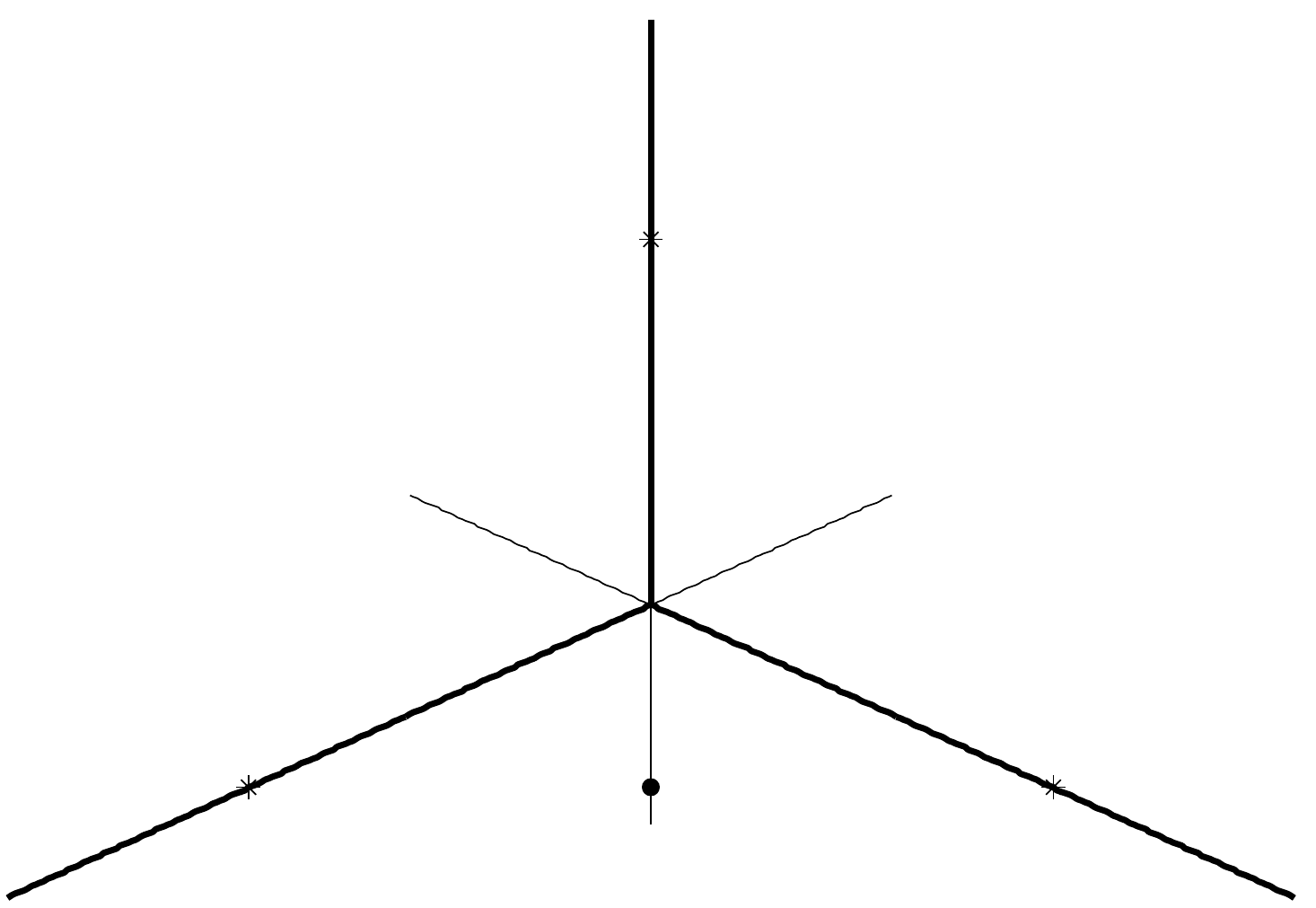}%
\put(58,45){\emph{v)} $\mu_v(\gamma'(v))<0$}
      \put(40,49){\small $a_k $}
      \put(43,8){\small $v $}
\end{overpic}} &
\mbox{\begin{overpic}[scale=0.32]%
{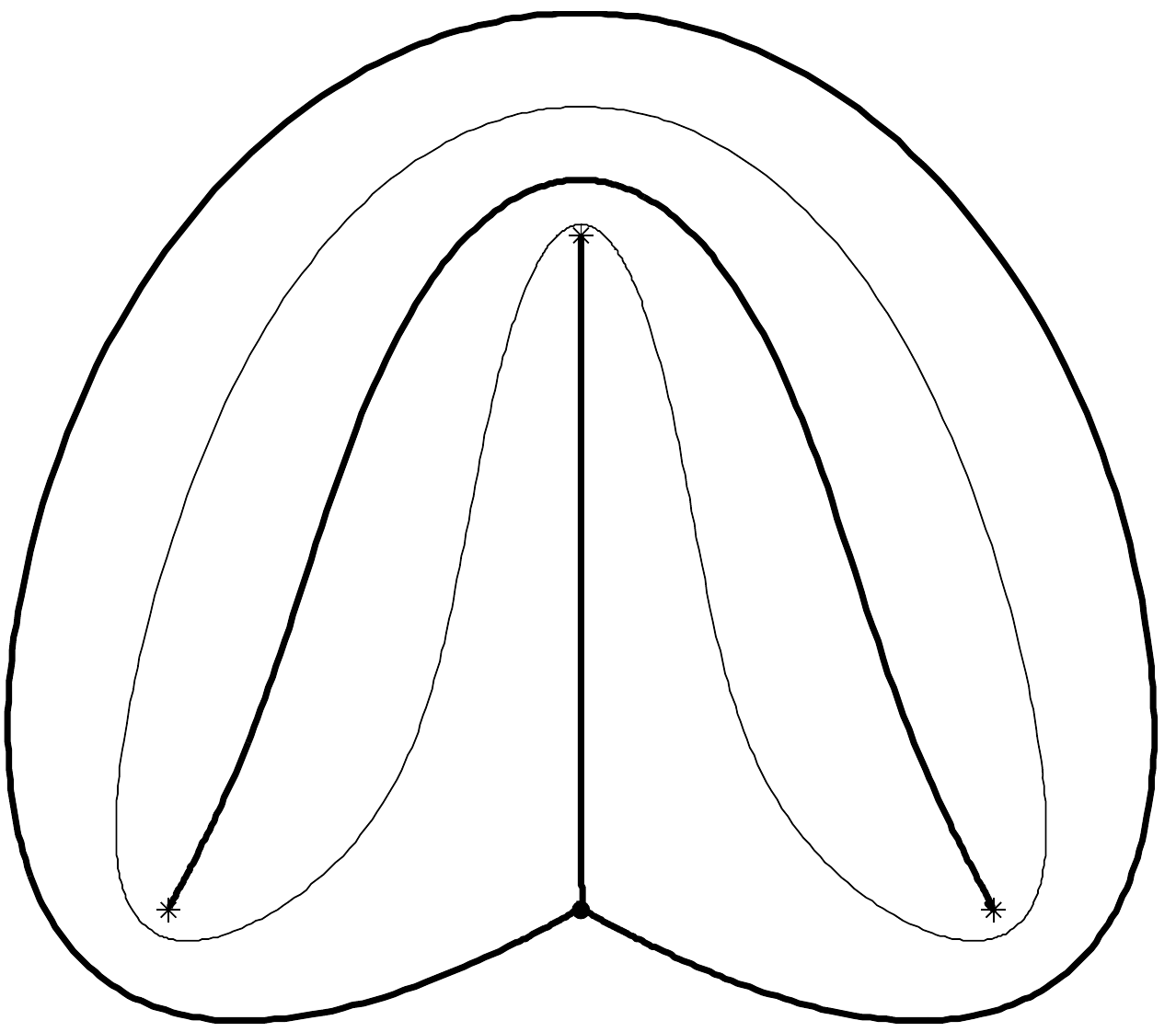}%
\end{overpic}}
\end{tabular}
\caption{Position of $v$ on $\ell_0 \cup \ell_1 \cup \ell_2$ (left) and the corresponding trajectories of the differential $\varpi_v$ in \eqref{qdp2}.}
\label{fig:cases}
\end{figure}
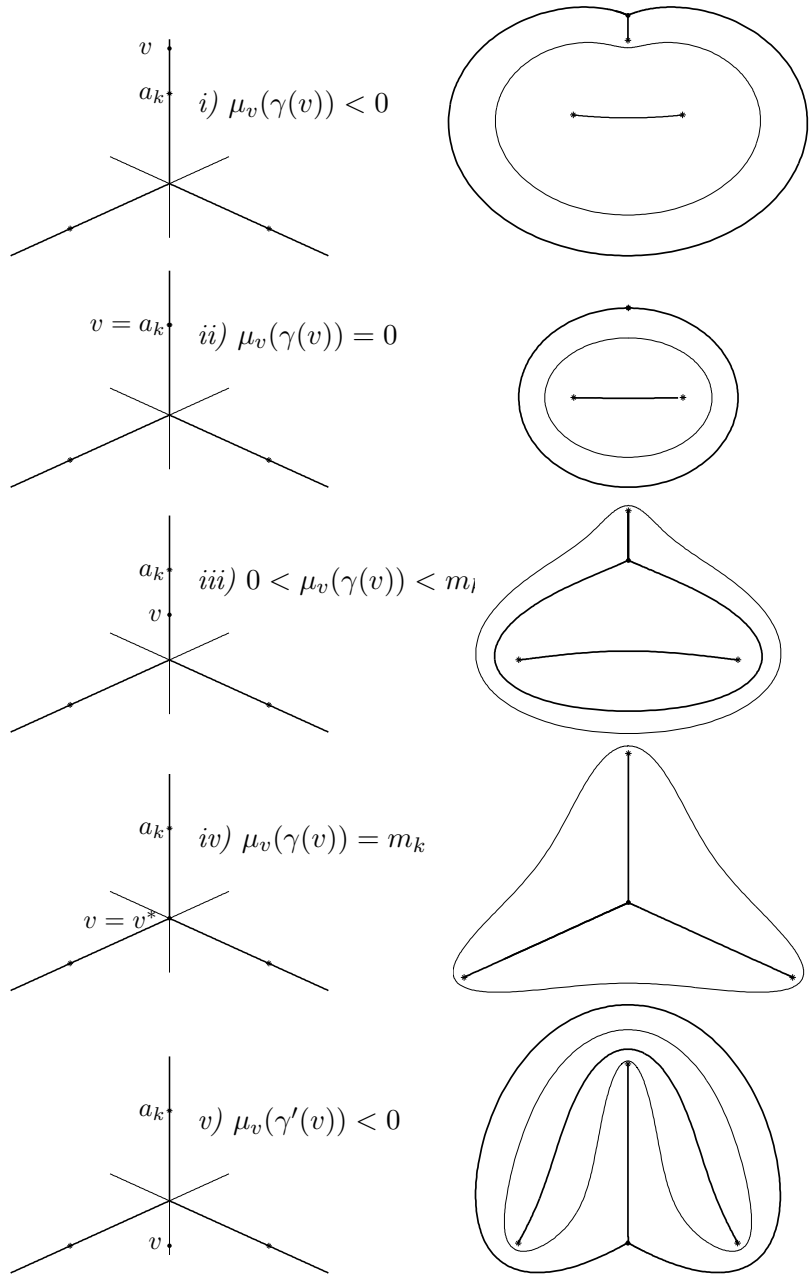

See illustration of the correspondence between the position of $v$ on $\mathcal V$ and the trajectories of $\varpi_v$ in Figure \ref{fig:cases}.

\begin{corollary}
Any positive $\AA$-critical measure $\mu_v$ has the trivial homotopy type.
\end{corollary}
\begin{rmk}
Taking into account this Corollary and the definition of $\mu_v$ we can rewrite \eqref{inequalityMu} is the following equivalent form:
$$
\frac{1}{\pi i} \, \int_{\alpha (v)} \left(\sqrt{\frac{z-v}{A(z)}}\right)_+ \, dz = \frac{1}{\pi i} \, \int_{a_{j_1}}^{a_{j_2}} \left(\sqrt{\frac{z-v}{A(z)}}\right)_+ \, dz \in [1-m_k, 1], \quad k\in \{0,1,2\},
$$
where $j_1 =\min (\{0,1,2\}\setminus \{k\})$, $j_2 =\max (\{0,1,2\}\setminus \{k\})$, and we integrate along the straight segment joining $a_{j_1}$ and $a_{j_2}$. This system of equation defines the set $\widehat{\mathcal V}_+$ completely.
\end{rmk}

Finally, in relation with our primary goal of the description of the weak-* asymptotics of the zeros of Heine-Stieltjes and Van Vleck polynomials we state the following important result, which however will not be proved completely in this paper.
\begin{theorem}
\label{thm:identification}
$\mathcal V_+= \widehat{\mathcal V}_+$.
\end{theorem}
The inclusion $\mathcal V_+ \subset \widehat{\mathcal V}_+$ follows from the definition of $\mathcal V_+$ and Theorem \ref{the:weakAsymptotics}: if $v\in \mathcal V_+$, then $\mu_v$ is a limit distribution of the zero counting measures of the Heine-Stieltjes polynomials, so that $\mu_v\geq 0$. The inverse inclusion ($\mu_v\geq 0 \Rightarrow v\in \mathcal V_+$) is also valid, but it cannot be established using methods of this paper. We plan to present the proof in a subsequent publication related to the strong asymptotics of Heine-Stieltjes polynomials.
However, as the consequence of Theorem \ref{thm:identification}, we identify the set of accumulation points of the zeros of Van Vleck polynomials $\mathcal V_+$ with the values of $ v\in \mathcal V$ making $\mu_v\geq 0$.

\begin{rmk}
Although $\mathcal V_+$ and the Chebotarev's continuum $\Gamma^*$ are topologically identical and, according to numerical experiments carried out by B.~Shapiro, metrically very close, they are not the same (an in consequence, the conjecture made in \cite{Shapiro2008b} is false). For simplicity, take $a_0=0$, $a_1=1$, $\Im a_2>0$, and define for $z$ in the upper half plane close to the origin 
$$
f_1(z)=\int_{0}^z \sqrt{\frac{z-t}{A(t)}}\, dt, \qquad f_2(z)=\int_{0}^z \sqrt{\frac{v^*-t}{A(t)}}\, dt,
$$
where we integrate along segments joining $0$ with $z$. The change of variables $t\mapsto z u$ in the integrand of $f_1$ yields the asymptotic expansion 
$$
f_1(z)=  \frac{\pi }{2\sqrt{a_2}} \,  z \, \left( 1 +     \frac{a_2+1 }{8 a_2 }     z  + \mathcal O(z^2)\right), \quad z\to 0.
$$
On the other hand, using the asymptotic expansion of the integrand of $f_2$ we get
$$
f_2(z)= 2 \sqrt{\frac{v^* z}{a_2}} \, \left( 1 + \frac{1}{6 } \left( 1+ \frac{1 }{a_2 } -\frac{1 }{ v^* }\right)  z + \mathcal O(z^2) \right), \quad z\to 0,
$$
where we take the main branch of $\sqrt{z}$.

Observe that $\Im f_1=0$ defines locally the set $\mathcal V$, while $\Im f_2 =0$ corresponds to the  Chebotarev compact of $\AA$. Assuming that both curves are tangent at the origin we conclude that $v^*/\sqrt{a_2}>0$, so that $v^*$ lies on the bisector of the interior angle formed by $1$ and $a_2$ at the origin. In order to check the second order tangency, we can invert the mapping $y=f_1(z)$ and analyze  $F(y)=f_2(f_1^{-1}(y))$ that maps the real line into itself at the origin. Setting $v^*=s (1+a_2)$, $s>0$, we get
$$
F(y)= \left(\frac{8}{\pi} \frac{v^*}{\sqrt{a_2}}\, y \right)^{1/2}  +
 \frac{1}{6}\, \left( \frac{1}{2 \pi^3  } \frac{\sqrt{a_2}}{  v^*} \right)^{1/2} \left( \frac{ 5 (a_2+1) v^* }{a_2}-8\right)\, y^{3/2}+ \mathcal O(y^{5/2}).
$$
Setting $v^*=s \sqrt{a_2}$, $s>0$, we get
$$
\frac{   (a_2+1) v^* }{a_2} =\left( \sqrt{a_2}+\frac{1}{\sqrt{a_2}}\right) s,
$$
which is real only if  $|a_2|=1$. This shows that at $a_0$ both curves $\mathcal V$ and $\Gamma^*$, however close, are not identical, at least when the triangle with vertices at $\AA$ is not isosceles.

\end{rmk}

\section{General families of $\AA$-critical measures} \label{sec:generalP}

Most of the arguments presented in Section \ref{sec:p=2} for the case $p=2$ may be carried over to the case of an arbitrary $p$ with minor modifications. However, in a certain sense the multidimensional case is significantly more complicated. The volume of this paper does not allow to develop the whole theory, covering both signed $\AA$-critical measures and closed quadratic differentials of an arbitrary homotopic type. In turn, without such a theory it is more complicated to separate positive $\AA$-critical measures from the signed ones. So, we restrict ourselves here to a less ambitious goal allowing a shorter treatment: we put forward a constructive characterization of the positive $\AA$-critical measures. We prove that the constructed measures are indeed positive, but the complete proof of the fact that there is no other positive $\AA$-critical measures is matter of a forthcoming paper.

\subsection{Mappings generated by periods of a rational quadratic differentials} \label{sec:mappings}

Let us recall the notation. We have the fixed set $\AA=\{a_0, a_1, \dots, a_p\}$ of distinct points on $\C$, $A(z)=\prod_{j=0}^p (z-a_j)$,
\begin{equation}\label{quadDiffBis}
V(z)\isdef \prod_{j=1}^{p-1} (z-v_j), \qquad R(z)\isdef \frac{V(z)}{  A(z) },  \quad \text{and} \quad \varpi  (z)=- R(z)\, (dz)^2
\end{equation}
is a rational quadratic differential on the Riemann sphere $\overline \C$. Zeros $v_j$'s of $V$ are not necessarily simple, and we denote $\vt{v}\isdef \{v_1, \dots, v_{p-1} \}\in \C^{p-1}$  with account of their multiplicity. Let also
$$
\mathcal V \isdef \{ \vt{v}:\, \varpi \text{ is closed}\}.
$$
Occasionally, it is convenient to consider $\varpi$ as a differential form $\frac{1}{i} \sqrt{R}(z) dz$ on the Riemann surface $\mathcal R$ of $\sqrt{R}$, or equivalently, on the hyperelliptic surface of genus $p-1$ given by $w^2=A(z)V(z)$.

Let $\Gamma=\gamma _1 \cup \dots \cup \gamma _p$ be a set consisting of $p$ disjoint arcs $\gamma _k$, each one connecting a pair of points from $\AA\cup \vt v$ in such a way that $\overline \C\setminus \Gamma $ is connected and $\sqrt{R}$ is holomorphic in $\overline \C\setminus \Gamma $. The Carath\'{e}odory boundary of $\overline \C\setminus \Gamma $ consists of $p$ components $\widehat{\gamma}_k \isdef \gamma _k^+ \cup \gamma _k^-$, with a positive orientation with respect to $\overline \C\setminus \Gamma $. We can consider $\widehat{\gamma}_k$ as cycles in $\overline{\C}\setminus \Gamma$ enclosing the endpoints of $\gamma _k$. Part of $\mathcal R$ over $\overline{\C}\setminus \Gamma$ splits into two disjoint sheets, so we may consider $\widehat{\gamma}_k$ as cycles on $\mathcal R$.

Let us define
\begin{equation}\label{defW_k9}
    w_k(\vt v)=w_k (\vt v, \Gamma ) \isdef \frac{1}{2\pi i}\, \oint_{\widehat{\gamma}_k} \sqrt{R(z)}dz, \quad k=1, \dots, p,
\end{equation}
where $\sqrt{R}\big|_{\widehat{\gamma}_k}$ are the boundary values of the branch of $\sqrt{R}$ in $\overline{\C}\setminus \Gamma $ defined by $\lim_{z\to \infty} z \sqrt{R(z)} =1$. Clearly, the boundary values $(\sqrt{R})_\pm$ on $\gamma_k^\pm$ are opposite in sign. Therefore, with any choice of orientation of $\gamma _k$ and a proper choice of $\sqrt{R}=(\sqrt{R})_+$ on $\gamma_k$, we will have
\begin{equation}\label{periods9}
    w_k(\vt v)=\frac{1}{\pi i}\, \int_{\gamma _k} \sqrt{R(z)} dz, \quad k=1, \dots, p.
\end{equation}
By the Cauchy residue theorem we have that $w_1+\dots+w_p=1$ for any $\vt v\in \C^{p-1}$. Thus, we can restrict the mapping $\vt v\mapsto \vt w$ to $p-1$ components of $\vt w \isdef (w_1, \dots, w_{p-1})\in \C^{p-1}$. In this way, we have defined the mapping
\begin{equation}\label{mapping9}
  \mathcal P(\cdot, \Gamma ):\, \C^{p-1} \to  \C^{p-1} \quad \text{such that} \quad  \mathcal P(\vt v, \Gamma ) = \vt w(\vt v, \Gamma ) .
\end{equation}
Each component function $w_j(v_1, \dots, v_{p-1})$ is analytic in each coordinate $v_k$ (even if $v_k$ is at one of the endpoints of $\gamma _k$). Once defined by the integral in \eqref{periods9}, this analytic germ allows an analytic continuation along any curve in $\C\setminus \AA$. Arcs $\gamma _k$ are not an obstacle for the continuation since the integral in \eqref{periods9} depends only on the homotopic class of $\Gamma$  in $\C\setminus (\AA\cup \vt v)$. The homotopy of $\Gamma $ is a continuous modification of all components simultaneously in such a way that they remain disjoint in all intermediate positions. Under this assumption we can continuously modify the selected branch of $\sqrt{R}$ in $\overline{\C}\setminus \Gamma $ along with the motion of $\Gamma $.

We note that this notion of homotopy is different from the concept of homotopic class based on a choice of a collection of Jordan contours in $\C\setminus \AA$, which is standardly used to classify closed differentials (see \cite{Strebel:84}).

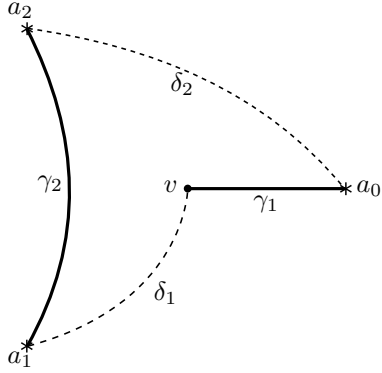
\begin{figure}[htb]
\centering \begin{overpic}[scale=0.75]%
{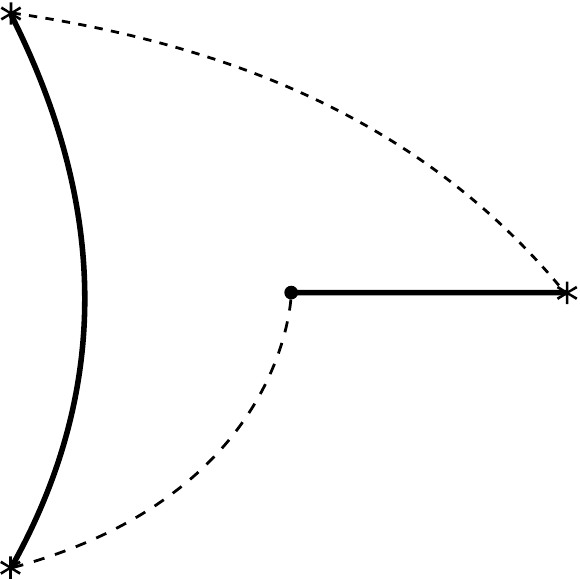}%
   \put(-4,-3){\small $a_1 $}
 \put(-4,102){\small $a_2$}
 \put(101,48.5){\small $a_0$} 
  \put(43,48.5){\small $v $}
 \put(70,44){\small $\gamma_1$}
   \put(5,50){\small $\gamma_2$}
  \put(40,16){\small $\delta_{1}$}
   \put(45,79){\small $\delta_{2}$}
\end{overpic}
\caption{An homology basis for $\mathcal R$.}
\label{fig:homology}
\end{figure}

The homology basis $\{\widehat{\gamma }_j \}_{j=1}^{p-1}$ of the domain $\overline{\C}\setminus \Gamma \subset \mathcal R$ defined above may be completed (in several ways) to form a homology basis $\widehat{\gamma }_1, \dots \widehat{\gamma }_{p-1}, \widehat{\delta  }_1, \dots, \widehat{\delta  }_{p-1}$ for $\mathcal R$. We can select the cycles $\widehat{\delta }_j$ as a lifting to $\mathcal R$ of a collection of arcs $\delta _j\subset \overline \C\setminus \Gamma $, each connecting new (different) pairs of points from $\AA\cup \vt v$ (see Figure~\ref{fig:homology}). 
We denote by $\Delta \isdef (\delta _1, \dots , \delta _{p-1})$, $\widehat \Delta  \isdef (\widehat{\delta} _1, \dots , \widehat{\delta} _{p-1})$. Accordingly, we define mappings
\begin{align}
\label{wtilde1_9} \widetilde{w}_k(\vt v) =\widetilde{w}_k(\vt v, \Delta )\isdef \frac{1}{2\pi i}\, \oint_{\widehat{\delta} _k}\sqrt{R(z)}\, dz, \quad k=1, \dots, p-1, \\
\label{wtilde2_9} \widetilde{\mathcal P}(\cdot, \Delta ):\C^{p-1}\mapsto \C^{p-1} \quad \text{such that} \quad  \widetilde{\mathcal P}(\vt v, \Delta )= (\widetilde{w}_1(v, \Delta ), \dots \widetilde{w}_{p-1}(v, \Delta )).
\end{align}
An important new mapping associated with the complete basis of homology on $\mathcal R$ is
$$
 \mathfrak P (\cdot, \Gamma , \Delta ):\, \C^{p-1} \longrightarrow \R^{2p-2}
$$
given by
\begin{equation}\label{newmapping9}
   \mathfrak  P (\vt v, \Gamma , \Delta )=  (\Im w_1, \dots, \Im w_{p-1}, \Im \widetilde{w}_1, \dots, \Im \widetilde{w}_{p-1})=(\Im (\mathcal P(\vt v, \Gamma )), \Im(\widetilde{\mathcal P}(\vt v, \Delta ))).
\end{equation}
In order to prove that both mappings, $\mathcal P$ and $\mathfrak  P $, are locally invertible we need the following lemma, which is a standard fact of the theory of the Riemann surfaces, see e.g.~\cite[Chapter 10]{Springer1957} or \cite[Chapter 5]{Jost06}.
\begin{lemma}
\label{lemma:springer}
Let $W_1, \dots, W_{p-1}$ be a basis of holomorphic differentials on $\mathcal R$ (a cohomology basis). Then
\begin{enumerate}
\item[(i)] $p-1$ vectors
$$
\left( \int_{\gamma _j} W_1, \dots, \int_{\gamma _j} W_{p-1}\right) \in \C^{p-1}, \quad j=1, \dots, p-1,
$$
are linearly independent over $\C$.
\item[(ii)] the system of $2p-2$ vectors, which consists of $p-1$ vectors in (i) above plus $p-1$ vectors
$$
\left( \int_{\delta  _j} W_1, \dots, \int_{\delta  _j} W_{p-1}\right) \in \C^{p-1}, \quad j=1, \dots, p-1,
$$
are linearly independent over $\R$.
\end{enumerate}
\end{lemma}
\begin{proposition}
\label{prop:9_1}
Both mappings $w=\mathcal P(\vt v)$ and $(\Im w, \Im \widetilde w)= \mathfrak P  (\vt v)$ are locally invertible at any $\vt v(v_1, \dots, v_{p-1})\in (\C\setminus \AA)^{p-1}$ with $v_i\neq v_j$ for $i\neq j$.
\end{proposition}
\begin{proof}
We have that for any $j, k \in \{1, \dots, p-1\}$,
$$
\frac{\partial w_j}{\partial v_k}=\frac{1}{2\pi i}\, \int_{\gamma _j} \sqrt{R(t)}\, \frac{dt}{t-v_k}=\frac{1}{4\pi i}\, \oint_{\widehat{\gamma }_j}   \frac{V_k(t)}{\sqrt{A(t) V(t)}}\, dt=\frac{1}{2}\, \oint_{\widehat{\gamma }_j}   W_k,
$$
where
$$
V_k(t)\isdef \frac{V(t)}{t-v_k} \quad \text{and} \quad W_k \isdef \frac{1}{2\pi i}\, \frac{V_k(z)}{\sqrt{A(z) V(z)}}\, dz, \quad k=1, \dots, p-1.
$$
Since $V_k$ are linearly independent polynomials of degree $p-2$ ($V$ has simple roots), $\{W_k\}_{k=1}^{p-1}$ is a basis of holomorphic differentials on $\mathcal R$. By Lemma \ref{lemma:springer} \emph{(i)}, vectors $\{\partial w /\partial v_k\}_{k=1}^{p-1}$ are independent; this means that the Jacobian of $\mathcal P$ does not vanish and $\mathcal P$ is invertible. In a similar fashion, by Lemma \ref{lemma:springer} \emph{(ii)}, vectors $\{\partial w /\partial v_k\}_{k=1}^{p-1}$  and $\{\partial \widetilde{w} /\partial v_k\}_{k=1}^{p-1}$, $k=1, \dots, p-1$, are linearly independent over $\R$. The matrix of the linear mapping
$$
\left\{ \Re \vt v, \Im \vt v \right\}\in \R^{2p-2} \longrightarrow \mathfrak P(\vt v, \Gamma , \Delta ) \in \R^{2p-2}
$$
is therefore nonsingular, and $ \mathfrak P $ is locally invertible.
\end{proof}

The following result is a multidimensional version of Proposition \ref{prop:localstructureV}.
\begin{proposition}
\label{prop:9_2}
Let $\mu^0$ be an $\AA$-critical measure such that $\Gamma=\supp(\mu ^0)$ has connected components $\gamma _1^0,  \dots , \gamma _{p}^0$, and $\overline{\C}\setminus \Gamma$ is connected. Let $\varpi_0 =R_0(z)(dz)^2$ be the quadratic differential associated with $\mu^0$, where $R_0=V_0/A$, and $\vt v^0=(v_1^0, \dots, v_{p-1}^0)$ is the vector of zeros of $V_0$. Assume that $\varpi_0$ is in a general position (that is, all $v_k$'s are pairwise distinct and disjoint with $\AA$). Then for an $\varepsilon >0$ and any $m_j \in\R$, $j\in \{1, \dots, p-1\}$, satisfying
$$
|m_j - \mu ^0(\gamma _j^0)|<\varepsilon , \quad j=1, \dots, p-1,
$$
there exists a unique solution $\vt v \in \mathcal V$ of the system
\begin{equation}\label{systemforMu9}
    w_j(\vt v,\Gamma )=m_j, , \quad j=1, \dots, p-1.
\end{equation}
The quadratic differential $\varpi = -R(z)(dz)^2$, $R(z)=\prod_{k=1}^{p-1} (z-v_k)/A(z)$, is closed. The associated $\AA$-critical measure $\mu $ with $\supp(\mu)$ homotopic to $\supp(\mu_{0})$ satisfies $\mu (\gamma _j)=m_j$, $j=1, \dots, p-1$, where $\supp(\mu )=\gamma _1\cup \dots \cup \gamma _{p }$ and $\supp(\mu)$ is homotopic to $\supp(\mu^0 )$.
\end{proposition}
\begin{proof}
We use the continuity of the dependence of short or critical trajectories of $\varpi$ from the zeros $\vt v=(v_1, \dots, v_{p-1})$ of $V$. Let, for instance, $\gamma ^0$ be an arc of $\supp(\mu ^0)$ connecting two $v^0$-points $v_1^0$ and $v_2^0$. For $\varepsilon >0$ small enough the solution of \eqref{systemforMu9} has two points, $v_1$ and $v_2$, close to $v_1^0$ and $v_2^0$, respectively; moreover, one of the trajectories of $\varpi$ comes out of $v_1$ in the direction close the the direction of $\gamma ^0$ at $v_1^0$. This trajectory will be close to $\gamma $ and will pass near $v_2$. If it does not hit $v_0$ then $\frac{1}{\pi i} \int_{v_1}^{v_2} \sqrt{R}dt\notin \R$, in contradiction with \eqref{systemforMu9}. Same (even simpler) arguments apply to any trajectories connecting two points from $\AA$ or a point from $\AA$ with another from $\mathcal V$.
\end{proof}

This result introduces a topology on the set of $\AA$-critical measures in a general position. We will call a \emph{cell} any connected component of this topological space. By Proposition \ref{prop:9_2}, a cell is a manifold of real dimension $p-1$. We can use as local coordinates either the vector $(\mu _1=\mu(\gamma _1) , \dots, \mu _{p-1}=\mu(\gamma _{p-1}))$ or $\vt v =(v_1, \dots, v_{p-1})$.
\begin{example}
 Consider the case $p=2$, studied in detail in Section \ref{sec:p=2}. In this situation the real dimension of each cell is $1$. The image of each non-distinguished cell in the $v$-plane is an analytic are connecting the Chebotarev's center $v^*$ and $\infty$. The unit positive measures, parametrized by the set $\widehat{\mathcal V}_+$ (see \eqref{defVhatPlus}), are represented on the $v$-plane as a union of three cells -- arcs $\ell_j^+$, $j=0,1,2$; their boundaries are points from $\AA^*= \{ a_0, a_1, a_2, v\}$, see Theorem~\ref{thm:structureVPlus2}.
\end{example}
For an arbitrary $p$, the boundary of a cell consists of pieces of manifolds of dimension $<p$. We come to a boundary point of a cell if either:
\begin{enumerate}
\item[\emph{(i)}] one of the components -- arcs $\gamma _j\subset \supp(\mu )$, --  degenerates to a point; this case is in turn subdivided into two subcases:
    \begin{enumerate}
    \item[\emph{(a)}] coalescence of a zero and a pole joined by an arc; reduction of $p$;
    \item[\emph{(b)}] coalescence of two zeros joined by an arc;  the polynomial $V$ gains a double zero.
    \end{enumerate}
\item[\emph{(ii)}] two components (arcs) of $\supp(\mu )$ meet (at a zero of $V$).
\end{enumerate}

We need to take a closer look at the case \emph{(ii)}. Again, a simple but important example is $p=2$, Section \ref{sec:p=2}. An arc $\ell_1^+ \in \widehat{\mathcal V}_+$ is defined by (see Theorem \ref{thm:structureVPlus2})
$$
w_1(v)=\frac{1}{\pi i}\, \int_{a_1}^v \sqrt{R(t)}\, dt =\mu_1 \in (0, m_1)\subset \R,
$$
with a selection of the proper branch of the function. This is a $\mathcal P$-parametrization of the cell $\ell_1^+$, which uses the coordinates $\mu_1=w(v,\Gamma )$, $\Gamma =[a_1, v]$. The extremal values $\mu _1=0$ and $\mu _1=m_1$ represent the boundary of the cell. Function $w(v)$ is analytic at both points; it is convenient to analyze the reconstruction of $\mu _v$ near $\mu _1=0$, that is, $v=a_1$. The boundary point $\mu _1=m_1$, corresponding to $v=v^*$, is better seen from the point of view of the $\mathfrak P$-mapping.

Let $\delta $ be an arc from $a_2$ to $v$ (homotopic to the arc $[a_2,v^*]\subset \Gamma ^*$ for $v$ close to $v^*$). Together with $w(v)=w(v,\gamma )$ we consider function
$$
\widetilde w(v)\isdef \frac{1}{\pi i}\, \int_\delta \sqrt{R(t)}\, dt .
$$
Then $v=v^*$ is uniquely defined by equations
$$
\Im w(v)=\Im \widetilde w(v)=0,
$$
and moreover, a nearby point $v$ on the arc $\Im w(v)=0$ is uniquely determined by a coordinate $h(v)=\Im \widetilde w(v)\in \R$. It is important that $h(v^*)=0$, and $h(v)$ changes sign when $v$ crosses $v^*$ along the arc (see again Figure~\ref{fig:cases}, where the reconstruction of $\supp(\mu_v)$ in dependence of $v$ is illustrated).

At this moment it is not really important to determine which sign of $h(v)$ corresponds to $\ell_1^+$. It is more convenient to introduce a number $s=\pm 1$ (depending on the branches of $w$ and $\widetilde w$) such that
$$
\ell_1^+=\{ v\in \C:\, \Im w(v)=0, \; \Im \widetilde w(v)=sh, \; h>0\}.
$$
This parametrization introduced originally around $v=v^*$, may be then extended to the whole arc $\ell_1^+$. We note that $h$ is the ``height of the cylinder'' (see \cite{Strebel:84}) and $\ell =2\mu=\Re w(v)$ is the ``length of the circumferences'', discussed in more details in Section \ref{sec:homotopicType2}.

\subsection{Structure of the set of positive $\AA$-critical measures}

Let $\Gamma ^*=\Gamma ^*(\AA)$ be the Chebotarev's continuum associated with $\AA=\{ a_0, \dots, a_p\}$, see Section \ref{sec:Chebotarevcont}; it consists of critical trajectories of $\varpi= -R^* dz^2$, $R^*=V^*/A$. We will assume again a general position for the set $\AA$, that is, $A$ and $V^*$ do not have common zeros and $V^*(z)=\prod_{k=1}^{p-1} (z-v_k^*)$ does not have multiple zeros. Thus, the critical set $\AA^*= \AA \cup \{v^*_1, \dots, v_{p-1}^* \}$ consists of $2p$ different points, and $\Gamma ^*$ is comprised of $2p-1$ arcs, that are critical trajectories of $\varpi$. Each trajectory joins two different points from $\AA^*$.

We begin the construction of positive critical measures by introducing local $\vt{v}$-coordinates. As above (see Section \ref{sec:mappings}), we identify measures $\mu \in \widehat{\mathcal V}_+$ with the corresponding polynomials $V(z)=\prod_{k=1}^{p-1} (z-v_k)$; furthermore, we define $V$ by the vector $\vt{v}=(v_1, \dots, v_{p-1}) \in \C^{p-1}$ of its zeros (numeration is not important).
Then each cell in $\widehat{\mathcal V}_+$ is a subspace of $\C^{p-1} \simeq \R^{2p-2}$, which is a manifold of the real dimension $p-1$, defined by $p-1$ real equations of the form
\begin{equation}\label{Takemura9}
    \Im w_j(\vt v)=\Im \left(\frac{1}{2\pi i}\, \int_{\widehat \gamma _j} \sqrt{R(t)}\, dt\right) =0, \quad j=1, \dots, p-1,
\end{equation}
where $\widehat{\Gamma }=\widehat{\gamma }_1\cup \dots \cup \widehat{\gamma }_{p-1} \cup \widehat{\gamma }_p$ is a union of Jordan contours $\widehat{\gamma }_k$ on $\mathcal R$ (double arcs) depending on $\vt v$, but mutually homotopically equivalent for values of $\vt v$ from the same cell. Practically, any $\vt v_0\in G$ has a neighborhood of $\vt v$ satisfying \eqref{Takemura9} with constant $\widehat{\gamma }_k$'s.

Now we come to the procedure of selection of combinatorial (rather than homotopic) types of cells; once the combinatorial type is fixed, the homotopic one will be determined from the Chebotarev's continuum, as described next.  We start with the Chebotarev's continuum $\Gamma ^*$ and the corresponding polynomial $V^*(z)=\prod_{k=1}^{p-1}(z-v_k^*)$. Each zero $v^*=v^*_k$, $k=1, \dots, p-1$, is connected by component arcs of $\Gamma ^*$ with three other points, say $a_1^*, a_2^*, a_3^*\in \AA^*$. We select one of these three arcs (for definiteness, $[v^*,a_1^*]$) and join two other arcs to make a single arc $[a_2^*, a_3^*]$, bypassing $v^*$ (we think that the arc $[a_2^*,a_3^*]$ still follows the two arcs from $\Gamma ^*$, but without touching $v^*$, instead passing infinitely close to it). This procedure, carried out at each zeros $v_k^*$ of $V^*$, creates a compact set $\Gamma $, and consequently, a cell $G(\Gamma )$ of corresponding measures $\mu\in \widehat{\mathcal V}_+$.

The selection of $\Gamma $, and hence, of the cell $G(\Gamma )$, is made by choosing one of the three connections for each $v^*_k$; there are $3^{p-1}$ ways to make the choice. Any choice splits $\Gamma ^*$ into $p$ ``disjoint'' arcs $\Gamma ^*=\gamma _1\cup \dots\cup \gamma _p$; out of them we select $p-1$ arcs to make an homology basis for $\C\setminus \Gamma ^*$, say $\gamma_{1}, \dots, \gamma_{p-1}$, and then consider the corresponding cycles $\widehat{\gamma }_k$, as described in Section \ref{sec:mappings}.

Next, together with the link $\gamma _k$ connecting $v_k^*$ to one of its neighbors from $\AA^*$ we will mark one more arc $\delta _k$ connecting $v_k^*$ with a different neighbor (i.e., a different point from $\AA^{*}$ connected by a branch to $v_{k}^{*}$, see Section~\ref{sec:Chebotarevcont}). The choice of $\gamma _k$ was arbitrary for each $k$; the choice of $\{ \delta _k\}_{k=1}^{p-1}$ has to be made in such a way that $p-1$ arcs $\gamma _k$ and $p-1$ arcs $\delta _k$ are all different. Then the corresponding cycles $\{\widehat{\Gamma }, \widehat{\Delta }\}=\{ \widehat{\gamma }_1, \dots, \widehat{\gamma }_{p-1}, \widehat{\delta  }_1, \dots, \widehat{\delta  }_{p-1} \}$ form an homology basis on $\mathcal R$ and may serve to define the mappings $\mathcal P$ and $\mathfrak P$ as in Section \ref{sec:mappings}.

We will mention first the description of the cell $G(\Gamma )$ in terms of the mapping $\mathcal P$. This way of parametrization is equivalent (this equivalence is, however, not completely on the surface) to the ``length of the circumferences'' parametrization of the closed differentials (see Section \ref{sec:p=2} for the case $p=2$). Let $w=w(\vt v)=\mathcal P(v, \widehat{\Gamma })$. We claim that the cell $G(\Gamma )$ is completely defined by the system
\begin{equation}\label{6of9}
    w_j(\vt v)=t _j \in \R_+, \quad j=1, \dots, p-1;
\end{equation}
more precisely, there exists a domain $M(\Gamma )=\{ (t _1, \dots, t _{p-1})\in \R_+^{p-1} \}$ such that for any  point $(t _1, \dots, t _{p-1})\in M(\Gamma )$ system \eqref{6of9} has a unique solution $\vt v\in \C^{p-1}$. Moreover, the corresponding measure $\mu=\mu _{\vt v}$ satisfies $\mu (\gamma _j)=t _j$, and $\supp(\mu )=\gamma _1\cup \dots \cup \gamma _{p}=\Gamma _{\vt v}$ is homotopic to $\Gamma $.

Summarizing, a rough description of the set $\widehat{\mathcal V}_+$ of unit positive $\AA$-critical measures may be made as follows. The set $\widehat{\mathcal V}_+$ is a union of $3^{p-1}$ of closed bounded cells $\overline{G}(\Gamma )$ ($\Gamma =\gamma_1\cup \dots \cup \gamma _p$ may be selected in $3^{p-1}$ ways). The interior $G(\Gamma )$ of each cell consists of measures $\mu$ in general position with $\supp(\mu )$ homotopic to $\Gamma $. Interiors of different cells are disjoint. Chebotarev's measure $\mu ^*$ (Robin measure of $\Gamma ^*$) is the only common point of all boundaries: $\mu ^*=\bigcap_\Gamma  \partial G(\Gamma )$.

The detailed proof of the assertions above and further analysis is beyond the scope of this paper. We will prove only that there exists a cell $G(\Gamma )$ with the homotopic type $\Gamma $ consisting of positive $\AA$-critical measures. It is easier to do it using the $\widetilde{\mathcal P}$- mapping (equivalent to the ``height of cylinders'' parametrization of the closed differentials).

We consider the mapping $\{\Im w, \Im \widetilde w\}=\mathfrak P(v; \Gamma , \Delta )$, described in Section~\ref{sec:mappings}, which is invertible in a neighborhood of $v^*$. We select a vector $(s_1, \dots, s_{p-1})$ of signs: each $s_j\in \{-1, +1\}$, and consider
\begin{equation}\label{7of9}
    \Im w_j(\vt v)=0, \quad \Im \widetilde{w}_j(\vt v)=s_j h_j, \quad h_{j }\in \R, \quad j=1, \dots, p-1.
\end{equation}

For $h_j=0$, system \eqref{7of9} has a unique solution $\vt v^*=(v_1^*, \dots, v_{p-1}^*)$. For sufficiently small $h_j>0$ this system is still uniquely solvable. Equations $\Im w_j(\vt v)=0$ imply that differential $\varpi$ in \eqref{quadDiffBis} is closed, the associated measure $\mu $ is $\AA$-critical, and $\supp(\mu )=\Gamma _{\vt v}=\gamma _{1,\vt v}\cup \dots \cup \gamma _{p,\vt v}$.

The homotopic type and signs of the components of $\mu $ depend on the behavior of trajectories of $\varpi$, which are originated at the points $a^*\subset \AA^*$ and close to trajectories $\delta _j$. Any such a trajectory will hit the corresponding point $v_j$ if $h_j=0$. If $h_j>0$, then it will pass from the left of $v_{j}$ or from the right of $v_{j}$, see Figure \ref{fig:ZeroEvolution}.

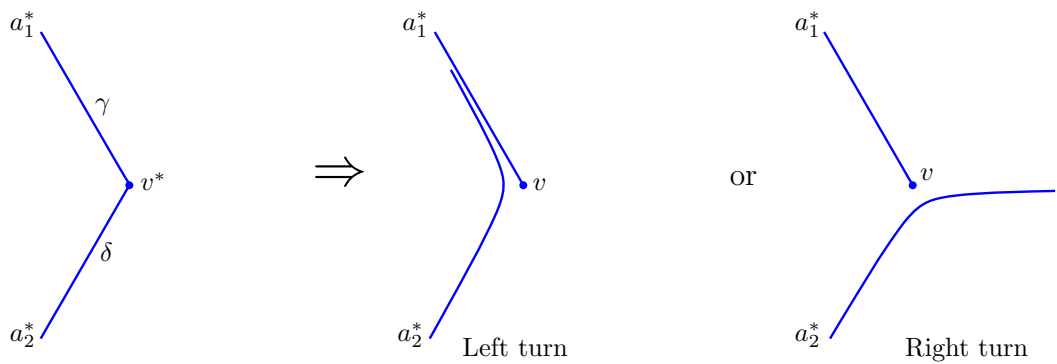
\begin{figure}[htb]
\centering \begin{tabular}{lll} \hspace{-2cm}\mbox{\begin{overpic}[scale=0.45]%
{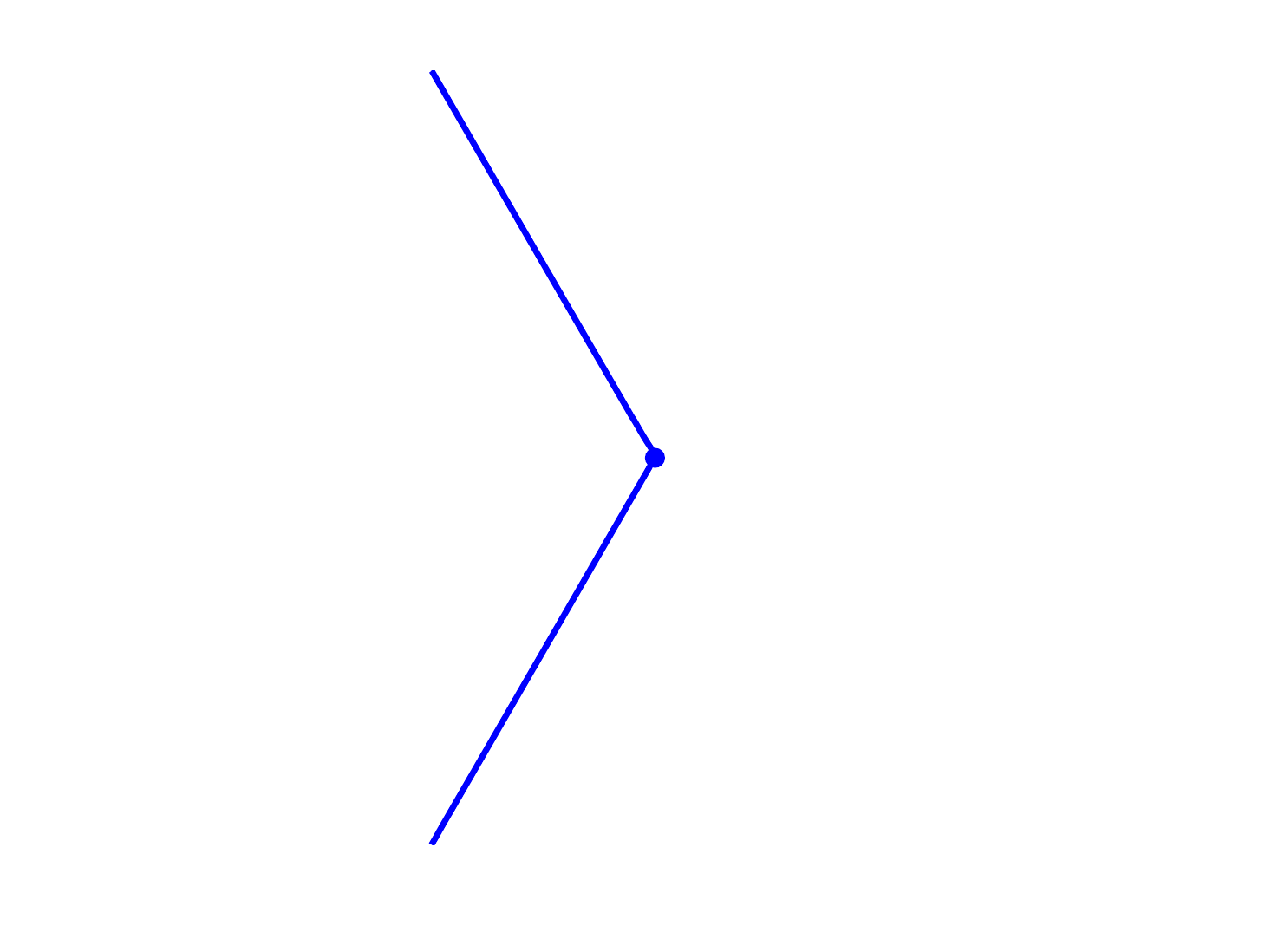}%
      \put(28,70){\small $a^*_1 $}
            \put(28,8){\small $a^*_2 $}
\put(54,38){\small $v^* $}
\put(45,54){\small $\gamma  $}
\put(46,24){\small $\delta $}
\end{overpic}} &
\hspace{-2.0cm}
\mbox{\begin{overpic}[scale=0.45]%
{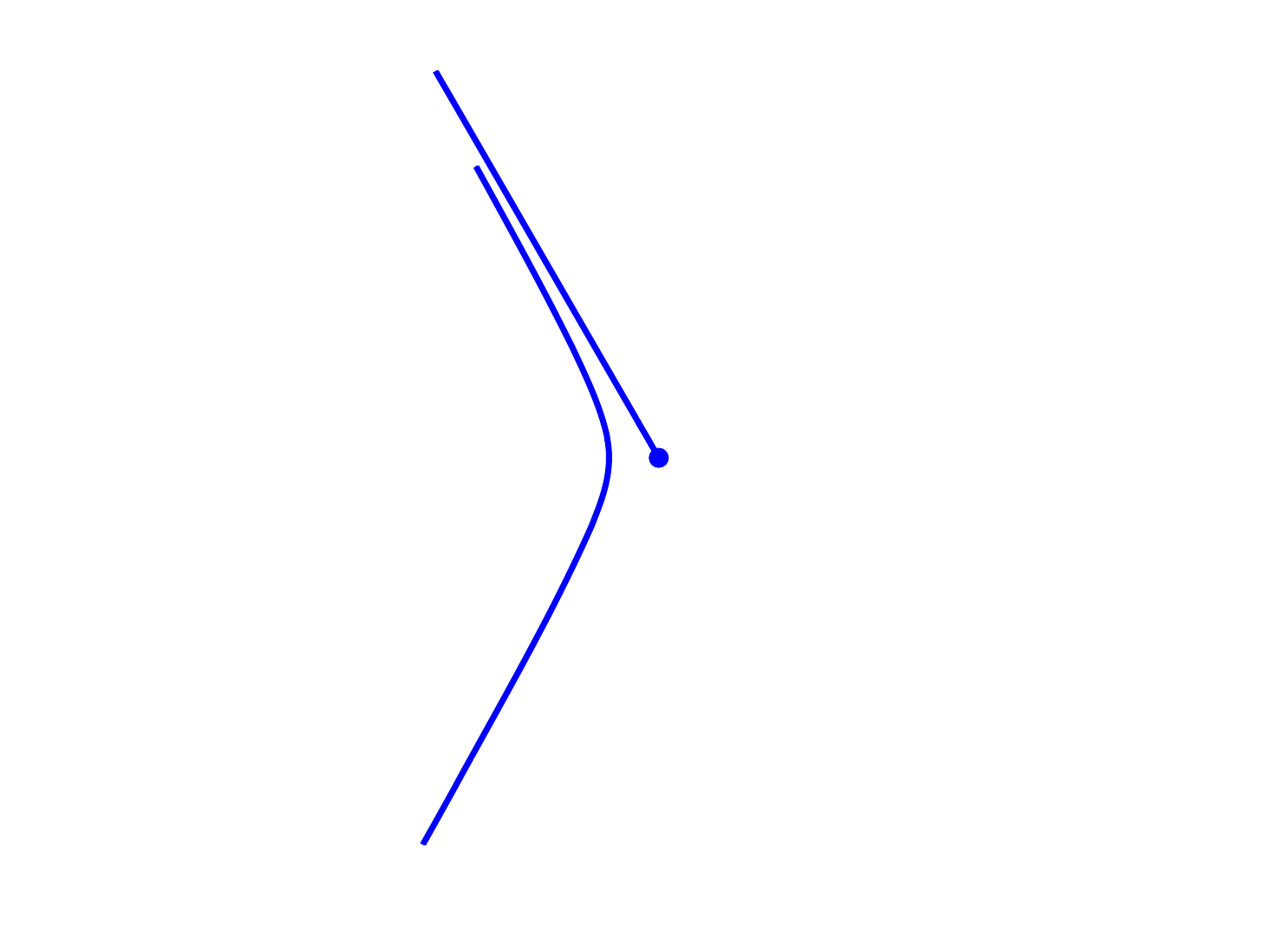}%
      \put(28,70){\small $a^*_1 $}
            \put(27,8){\small $a^*_2 $}
\put(54,38){\small $v  $}
\put(40,5){\small Left turn}
\put(10,39){\huge $\Rightarrow $}
\end{overpic}}&
\hspace{-2.0cm}
\mbox{\begin{overpic}[scale=0.45]%
{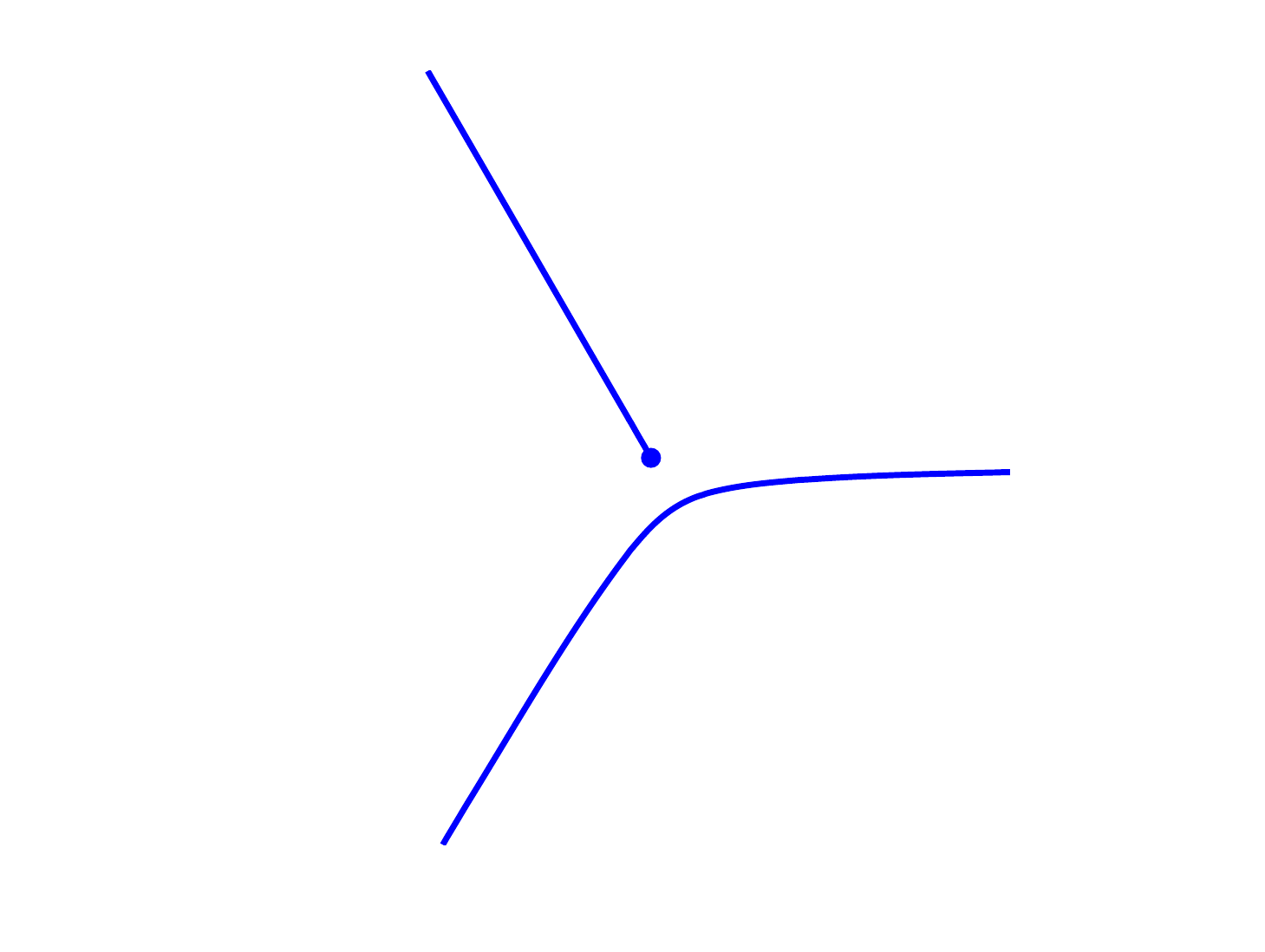}%
      \put(28,70){\small $a^*_1 $}
            \put(28,8){\small $a^*_2 $}
\put(53,40){\small $v  $}
\put(50,5){\small Right turn}
\put(15,39){\large or}
\end{overpic}}
\end{tabular}
\caption{Left and right turns.}
\label{fig:ZeroEvolution}
\end{figure}

A change from $s_j$ to $-s_j$ will change the direction of the turn. Therefore, there is a unique selection of vectors $(s_1, \dots, s_{p-1})$ such that all turns are right. Then the branch of $\sqrt{R}$ in $\overline{\C}\setminus \Gamma _{\vt v}$ will be close to the branch of $\sqrt{R^*}$ in $\overline{\C}\setminus \Gamma ^*$, and therefore the corresponding measure $\mu $ will be positive. In this sense, the cell we entered contains some positive measures. Therefore, they are positive, since $\supp(\mu )$ are all homotopic.

\section*{Acknowledgments} 
 
We are indebted to B.~Shapiro for interesting discussions and for providing us with the early version of his manuscripts \cite{Shapiro2008a} and \cite{Shapiro2008b}; after the first version of this paper was made public in the arxiv, we learned about a work in preparation of B.~Shapiro and collaborators, which has some overlappings with this paper. Fortunately, the methods and the paths we follow are very different. 

We also gratefully acknowledge many helpful conversations with H.~Stahl and A.~Vasil$'$ev, as well as useful remarks from M.~Yattselev concerning the first version of this manuscript. The software for computing the parameters of Chebotarev's compacts, provided by the authors of \cite{Ortega-Cerda2008} and freely available at their web site, was also useful for gaining some additional insight.

AMF is partially supported by Junta de Andaluc\'{\i}a, grants FQM-229, P06-FQM-01735, and P09-FQM-4643, as well as by the
research project MTM2008-06689-C02-01 from the Ministry of Science and Innovation of Spain and the European Regional Development Fund (ERDF).

EAR is partially supported by  the NSF grant DMS-9801677.


\def\cprime{$'$}

\end{document}